%% file: stable.tex
\documentclass[12pt,reqno]{amsart}
\usepackage{fullpage}

\newtheorem{theorem}{Theorem}[section]
\newtheorem{lemma}[theorem]{Lemma}

\newtheorem{cor}[theorem]{Corollary}
\newtheorem{conj}[theorem]{Conjecture}

\newtheorem{prob}[theorem]{Problem}
\usepackage{graphicx}
\usepackage{color}
\usepackage[dvipsnames]{xcolor}
\usepackage{subfigure}
\usepackage{amssymb}
\usepackage{amsmath,mathrsfs}
\usepackage{colonequals}
\usepackage{hyperref}
\theoremstyle{definition}
\newtheorem{definition}[theorem]{Definition}

\newtheorem{example}[theorem]{Example}

\newtheorem{assumption}[theorem]{Assumption}
\newtheorem{remark}[theorem]{Remark}

\renewcommand{\subset}{\subseteq}
\renewcommand{\supset}{\supseteq}
\renewcommand{\epsilon}{\varepsilon}

\newcommand{\abs}[1]{\left|#1\right|}                   % Absolute value notation
\newcommand{\vnorm}[1]{\left\|#1\right\|}    % norm notation
\newcommand{\vnormt}[1]{\left\|#1\right\|}    % norm notation
\newcommand{\vnormtf}[1]{\|#1\|}                         % norm notation, forced to be small
\newcommand{\Z}{\mathbb{Z}}                             % Blackboard notation

\newcommand{\E}{\mathbb{E}}
\renewcommand{\L}{L}
\newcommand{\mL}{\mathcal{L}}

\newcommand{\R}{\mathbb{R}}

\newcommand{\embolden}[1]{\textbf {#1}}
\newcommand{\redA}{\Sigma}
\newcommand{\redb}{\partial^{*}}
\newcommand{\redS}{\partial^{*}\Sigma}
\newcommand{\sdimn}{n}
\newcommand{\adimn}{n+1}
\newcommand{\scon}{\lambda}
\newcommand{\pcon}{\delta}
\newcommand{\epschoice}{\frac{1}{100}(m+\vnormt{w}^{2})^{-1}}

\newcommand{\pen}{\sqrt{\frac{\pi}{2}}\sum_{i=1}^{m}\vnormtf{\int_{\Omega_{i}}(x-\overline{w}^{(i)})\gamma_{\adimn}(x)dx}^{2}}   % penalty term
\newcommand{\pens}{\sqrt{\frac{\pi}{2}}\sum_{i=1}^{m}\vnormtf{\int_{\Omega_{i}^{(s)}}(x-\overline{w}^{(i)})\gamma_{\adimn}(x)dx}^{2}}   % penalty term, as a function of s
\newcommand{\pentil}{\sqrt{\frac{\pi}{2}}\sum_{i=1}^{m}\vnormtf{\int_{\widetilde{\Omega}_{i}}(x-\overline{w}^{(i)})\gamma_{\adimn}(x)dx}^{2}}   % penalty term just with tilde sets
\newcommand{\peshs}{\sqrt{\frac{\pi}{2}}\vnormtf{\int_{\Omega^{(s)}}(x-\overline{w})\gamma_{\adimn}(x)dx}^{2}}   % penalty term, as a function of s
\newcommand{\pender}{\sum_{i=1}^{m}\vnormtf{\sum_{j\in\{1,\ldots,m\}\colon j\neq i}\int_{\Sigma_{ij}}(x-\overline{w}^{(i)})f_{ij}\gamma_{\adimn}(x)dx}^{2}}   % penalty term derivative
\begin{document}

%\title{The Stability of Gaussian Minimal Bubbles}
\title{Stable Gaussian Minimal Bubbles}

\author{Steven Heilman}
\address{Department of Mathematics, University of Southern California, Los Angeles, CA 90089-2532}
\email{stevenmheilman@gmail.com}
\date{\today}
\thanks{Supported by NSF Grant DMS 1839406.}
%60E15, 60G15, 53A10, 58E30
%\subjclass[2010]{60E15, 53A10, 60G15, 58E30}
%\keywords{Gaussian, bubble, minimal surface, calculus of variations}

\begin{abstract}
It is shown that $3$ disjoint sets with fixed Gaussian volumes that partition $\mathbb{R}^{n}$ with nearly minimum total Gaussian surface area must be close to adjacent $120$ degree sectors, when $n\geq2$.  These same results hold for any number $m\leq n+1$ of sets partitioning $\mathbb{R}^{n}$, conditional on the solution of a finite-dimensional optimization problem (similar to the endpoint case of the Plurality is Stablest Problem, or the Propeller Conjecture of Khot and Naor).  When $m>3$, the minimal Gaussian surface area is achieved by the cones over a regular simplex.  We therefore strengthen the Milman-Neeman Gaussian multi bubble theorem to a ``stability'' statement.  Consequently, we obtain the first known dimension-independent bounds for the Plurality is Stablest Conjecture for three candidates for a small amount of noise (and for $m>3$ candidates, conditional on the solution of a finite-dimensional optimization problem).  In particular, we classify all stable local minima of the Gaussian surface area of $m$ sets.  We focus exclusively on volume-preserving variations of the sets, avoiding the use of matrix-valued partial differential inequalities. Lastly, we remove the convexity assumption from our previous result on the minimum Gaussian surface area of a symmetric set of fixed Gaussian volume.
\end{abstract}
\maketitle
%
%
% arxiv subjects: math.DG, math.PR, cs.CC?
%
%  MSC:    60E15, 60G15, 53A10, 58E30
%
%  keywords: Gaussian, bubble, minimal surface, calculus of variations
%
% 35? pages, 4 figures

\section{Introduction}\label{secintro}

This paper is a continuation of \cite{heilman18}.  There we showed that $3$ disjoint sets with fixed Gaussian volumes that partition $\mathbb{R}^{n+1}$ with minimum total Gaussian surface area must be adjacent $120$ degree sectors, when $n+1\geq2$, assuming an extra technical condition.  Under the same technical assumption, we also showed that $4$ disjoint sets with fixed Gaussian volumes that partition $\mathbb{R}^{n+1}$ with minimum total Gaussian surface area must cones over a regular simplex when $n+1\geq3$.  The analogous statement for any number $m$ of sets in $\R^{n+1}$ with $n+1\geq m-1$ was proven in \cite{milman18b}, unconditionally.  We refer to \cite{heilman18} for some discussion and motivation for this problem.  The present paper simplifies and improves the result of \cite{milman18b} in various ways.

When the Gaussian measure is replaced with Lebesgue measure, the main result of this paper (Theorem \ref{mainthm2} below) is not known even for two sets.  Some recent results \cite{cicalese16,cicalese17} have shown that $2$ disjoint sets in $\R^{2}$ with fixed Lebesgue measures with nearly minimum total surface area must be close to a Euclidean ``double bubble,'' whose boundary consists of three spherical caps meeting at $120$ degree angles.  The analogous statement for three or more sets seems to be quite difficult.  Even for two sets in $\R^{3}$ (using Lebesgue measure instead of the Gaussian measure), this same statement seems difficult.  On the other hand, our results (for the Gaussian measure) hold for any number of $m\leq n+2$ sets in $\R^{\adimn}$.  It is rather crucial that these results and proofs work for all values of $\adimn$ simultaneously such that $\adimn\geq m-1$.  That is, these results and proofs are dimension-independent.  The case $m>n+2$ seems nontrivial for technical reasons, but for applications we are only concerned with $m-1\leq\adimn$.  It is rather crucial that the number of independent translations of $\R^{n+1}$ (i.e. $n+1$) is greater than or equal to $m-1$.  When $m-1>n+1$, this property no longer holds, and a key part of the proof does not work.  Also, the Gaussian results imply some weak statements for other log-concave measures; see \cite[Theorem 10.7]{milman18b}.  For two sets, the main result of the present paper was proven in \cite{barchiesi16}.  In fact, our methods synthesize and elaborate upon \cite{barchiesi16,heilman18,milman18b}.  However, unlike the papers \cite{milman18a,milman18b}, we avoid the use of a matrix-valued differential inequality, and we instead focus only on (Gaussian) volume-preserving transformations.

There are two applications of an improvement of the result of \cite{milman18b}, both of which are described in more detail in \cite{isaksson11}.  The present work seems to give the first nontrivial results toward the general conjecture stated in \cite{isaksson11}, thereby achieving the first dimension-independent applications of the conjecture of \cite{isaksson11}.  First, by a Central Limit Theorem, an isoperimetric problem for the Gaussian measure can be (essentially) equivalently stated as an inequality for discrete functions.  Such an inequality is called the ``plurality is stablest'' conjecture from social choice theory.  This problem \cite{isaksson11} says that if votes are cast in an election between $m$ candidates, if every candidate has an equal chance of winning, and if no one person has a large influence on the outcome of the election, then taking the plurality is the most noise-stable way to determine the winner of the election.  That is, plurality is the voting method where the outcome is least likely to change due to independent, uniformly random changes to the votes.  The latter conjecture is a generalization of the ``majority is stablest conjecture'' proven in \cite{mossel10}.

The second application of our results is sharp hardness for the MAX-$m$-CUT problem \cite{khot07,isaksson11}.  The MAX-$m$-CUT problem asks for the partition of the vertices of an undirected graph into $m$ disjoint sets that maximizes the number of edges going between the $m$ sets.  This problem is NP-hard, so one cannot solve it in any reasonable time for a large graph, e.g. with $10^{4}$ vertices and $10^{6}$ edges.  However, one can find a partition of the vertices achieving a positive fraction of the maximum number of cut edges in polynomial time \cite{frieze95}.  In the case $m=2$, this fraction is approximately $.87856$.  And assuming the Unique Games Conjecture \cite{khot02,khot18}, the constant $.87856$ is the best possible \cite{khot07}.  In the case $m>2$, the analogous result is the conjecture \cite{isaksson11} that the present paper addresses (for a certain range of parameters).

We now state the problems of interest more formally.

We ask for the minimum total Gaussian surface area of $m$ disjoint volumes in $\R^{\adimn}$ whose union is all of $\R^{\adimn}$.  The case $m=1$ is then vacuous.  The case $m=2$ results in two half spaces.  That is, a set $\Omega\subset\R^{\adimn}$ lying on one side of a hyperplane has the smallest Gaussian surface area $\int_{\partial\Omega}\gamma_{\sdimn}(x)dx$ among all (measurable) sets of fixed Gaussian measure $\int_{\Omega}\gamma_{\adimn}(x)dx$ \cite{sudakov74}.  Here, $\forall$ $k\geq1$, we define
\begin{flalign*}
\gamma_{k}(x)&\colonequals (2\pi)^{-k/2}e^{-\vnormt{x}^{2}/2},\qquad
\langle x,y\rangle\colonequals\sum_{i=1}^{\adimn}x_{i}y_{i},\qquad
\vnormt{x}^{2}\colonequals\langle x,x\rangle,\\
&\qquad\forall\,x=(x_{1},\ldots,x_{\adimn}),y=(y_{1},\ldots,y_{\adimn})\in\R^{\adimn},\\
\int_{\partial\Omega}\gamma_{\sdimn}(x)dx&\colonequals\liminf_{\epsilon\to0^{+}}\frac{1}{\epsilon}\gamma_{\sdimn}(\{x\in\R^{\adimn}\colon x\notin\Omega\,\wedge\,\exists\,y\in\Omega,\, \vnormt{x-y}<\epsilon\}).
\end{flalign*}
If $\partial\Omega$ is a $C^{\infty}$ manifold, then we can interpret $\int_{\partial\Omega}\gamma_{\sdimn}(x)dx$ by locally parameterizing $\partial\Omega$, then locally integrating the function $\gamma_{\sdimn}(x)$ by multiplying by the Jacobian determinant of the parametrization.

\begin{remark}
Unless otherwise stated, all Euclidean sets in this work are assumed to be Lebesgue measurable.
\end{remark}

\subsection{Review of the Gaussian Multi-Bubble Theorem}

For an introduction to Gaussian isoperimetry, see \cite[Section 1.1]{heilman18}.

The following is our main problem of interest.

\begin{prob}[\embolden{Gaussian Multi-Bubble Problem}, {\cite{hutchings97,hutchings02,corneli08}}]\label{prob1}
Let $m\geq3$.  Fix $a_{1},\ldots,a_{m}>0$ such that $\sum_{i=1}^{m}a_{i}=1$.  Find measurable sets $\Omega_{1},\ldots\Omega_{m}\subset\R^{\adimn}$ with $\cup_{i=1}^{m}\Omega_{i}=\R^{\adimn}$ and $\gamma_{\adimn}(\Omega_{i})=a_{i}$ for all $1\leq i\leq m$ that minimize
$$\sum_{1\leq i<j\leq m}\int_{(\partial\Omega_{i})\cap(\partial\Omega_{j})}\gamma_{\sdimn}(x)dx,$$
subject to the above constraints.
\end{prob}

\begin{figure}[ht!]
\centering
\def\svgwidth{.4\textwidth}
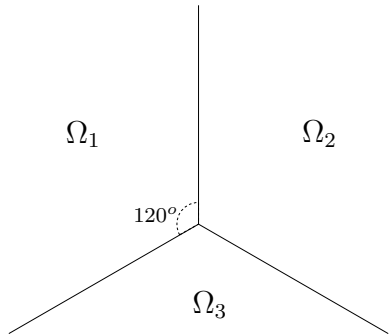
\caption{Optimal Sets for Conjecture \ref{conj0} in the case $m=3$, $\adimn=2$.}
\end{figure}

\begin{conj}[\embolden{Gaussian Multi-Bubble Conjecture} {\cite{hutchings02,corneli08,isaksson11}}]\label{conj0}
Let $\Omega_{1},\ldots\Omega_{m}\subset\R^{\adimn}$ minimize Problem \ref{prob1}.  Assume that $m-1\leq\adimn$.  Let $z_{1},\ldots,z_{m}\in\R^{\adimn}$ be the vertices of a regular simplex in $\R^{\adimn}$ centered at the origin.  Then $\exists$ $w\in\R^{\adimn}$ such that, for all $1\leq i\leq m$,
\begin{equation}\label{wdef}
\Omega_{i}=w+\{x\in\R^{\adimn}\colon\langle x,z_{i}\rangle=\max_{1\leq j\leq m}\langle x,z_{j}\rangle\},\qquad\gamma_{\adimn}(\Omega_{i})=a_{i}.
\end{equation}
\end{conj}

We sometimes refer to sets $\Omega_{1},\ldots,\Omega_{m}$ that minimize Problem \ref{prob1} as \textbf{Gaussian minimal bubbles}.  We refer to any sets satisfying \eqref{wdef} as \textbf{cones over a regular simplex}.

Conjecture \ref{conj0} was proven for $m=3$ in \cite{milman18a}.  Conjecture \ref{conj0} was then proven for $m=3,4$ in \cite{heilman18} assuming that the set of triple junction points of the optimal sets has subexponential volume growth.  And Conjecture \ref{conj0} was proven for any $m\geq3$ in \cite{milman18b}, unconditionally.

Still, it would be desirable to strengthen the result of \cite{milman18b}, and show that nearly minimizing sets are close to the simplicial cones described in Conjecture \ref{conj0}.  This goal was accomplished for $m=2$ in \cite{barchiesi16} using a second variation argument, along the lines of the arguments used in \cite{milman18a,heilman18,milman18b}, but with an added penalty term.  Below, we adapt the argument of \cite{barchiesi16} to the case $m=3$ unconditionally, and $m>3$ conditional on solving a finite-dimensional optimization problem (see Section \ref{secprop} and \ref{prob3} below.)

As in \cite{barchiesi16}, we add a ``penalty'' term to the surface area functional in Problem \ref{prob1}.  The penalty term is smaller when the sets are far from those predicted in Conjecture \ref{conj0}.  So, the surface area term and penalty term ``oppose'' each other, since they are each large or different kinds of sets.  If the effect of the penalty is small enough relative to the surface area term, then the minimizers of Problems \ref{prob1} and \ref{prob1p} will hopefully be the same.  If so, the penalty term indicates how far a set is from minimizing Problem \ref{prob1}.

\begin{prob}[\embolden{Gaussian Multi-Bubble Problem, with Stability Term}]\label{prob1p}
Let $m\geq3$.  Let $\epsilon>0$.  Assume that $m-1\leq\adimn$.  Fix $a_{1},\ldots,a_{m}>0$ such that $\sum_{i=1}^{m}a_{i}=1$.  Let $w\in\R^{\adimn}$ be defined by \eqref{wdef}.  For all $1\leq i\leq m$, let $\overline{w}^{(i)}\colonequals w/a_{i}$.  Find measurable sets $\Omega_{1},\ldots\Omega_{m}\subset\R^{\adimn}$ with $\cup_{i=1}^{m}\Omega_{i}=\R^{\adimn}$ and $\gamma_{\adimn}(\Omega_{i})=a_{i}$ for all $1\leq i\leq m$ that minimize
$$\sum_{1\leq i<j\leq m}\int_{(\partial\Omega_{i})\cap(\partial\Omega_{j})}\gamma_{\sdimn}(x)dx+\epsilon\pen,$$
subject to the above constraints.
\end{prob}

The $\sqrt{\pi/2}$ factor is added for purely aesthetic reasons in ensuing arguments.

\begin{theorem}[\embolden{Main Theorem}]\label{mainthm}
Let $\Omega_{1},\ldots,\Omega_{m}$ minimize Problem \ref{prob1p}.  Let
$$\epsilon\colonequals\epschoice.$$
If $m\leq 4$, then $\Omega_{1},\ldots,\Omega_{m}$ are cones over a regular simplex.

If $m\geq 5$, then $\exists$ $\epsilon>0$ in Problem \ref{prob1p} such that $\Omega_{1},\ldots,\Omega_{m}$ are cones over a regular simplex (and $\epsilon$ can depend on $a_{1},\ldots,a_{m}$ and on $m$ but not on $\adimn$.)
\end{theorem}

Theorem \ref{mainthm} then immediately implies the following improvement to Conjecture \ref{conj0}.

\begin{theorem}[\embolden{Main Corollary, Stable Version of Conjecture \ref{conj0}}]\label{mainthm2}
Let $\Omega_{1},\ldots,\Omega_{m}$ be cones over a regular simplex.  If $m>3$, assume that Conjecture \ref{conj5} holds.  Define $w\in\R^{\adimn}$ as in \eqref{wdef}.  For all $1\leq i\leq m$, let $\overline{w}^{(i)}\colonequals w/a_{i}$.  Let $\Omega_{1}',\ldots,\Omega_{m}'$ such that $\gamma_{\adimn}(\Omega_{i})=\gamma_{\adimn}(\Omega_{i}')$ for all $1\leq i\leq m$ and $\cup_{i=1}^{m}\Omega_{i}'=\R^{\adimn}$.  Then
$$\sum_{i=1}^{m}\gamma_{\sdimn}(\partial\Omega_{i}')
\geq\sum_{i=1}^{m}\gamma_{\sdimn}(\partial\Omega_{i})
 +\epsilon\sum_{i=1}^{m}\vnormt{\int_{\Omega_{i}}(x-\overline{w}^{(i)})\gamma_{\adimn}(x)dx}^{2}-\vnormt{\int_{\Omega_{i}'}(x-\overline{w}^{(i)})\gamma_{\adimn}(x)dx}^{2}.$$
Moreover, if $m\leq 4$, $\epsilon$ can be chosen to be
$$\epsilon\colonequals\epschoice.$$
If $m\geq 5$, then $\exists$ $\epsilon>0$ such that the above inequality holds (and $\epsilon$ can depend on $a_{1},\ldots,a_{m}$ and on $m$ but not on $\adimn$.)
\end{theorem}

\begin{example}\label{threex}
In the case $m=3$, $\sum_{i=1}^{m}\vnormtf{\int_{\Omega_{i}}(x-\overline{w}^{i})\gamma_{\adimn}(x)dx}^{2}$ is maximized by the sets described in Conjecture \ref{conj0} \cite{khot09}.  So, the term on the right tells us how far away the sets are from the optimal ones.  For example, the right side can be lower bounded using an elementary rearrangement argument \cite[Lemma 2.8]{heilman12}, implying following estimate (with a non-optimal exponent) for $w=0$ in Theorem \ref{mainthm2}, i.e. $\gamma_{\adimn}(\Omega_{i})=1/3$ for all $1\leq i\leq 3$ :
% if the right quantity is less than b, then distance to simplex is less than 6 b^1/8.  that is, distance to simplex is bounded by 6(right quantity)^1/8.
%  that is, right quantity is larger than (distance to simplex)^8 6^-8 = (sum of measure differences)^4  6^-8
$$\sum_{i=1}^{3}\gamma_{\sdimn}(\partial\Omega_{i}')
\geq\sum_{i=1}^{3}\gamma_{\sdimn}(\partial\Omega_{i})
 +10^{-10}\inf_{\substack{\mathrm{rotations}\\ R\colon\R^{\adimn}\to\R^{\adimn}}}
 \Big(\sum_{i=1}^{3}\gamma_{\adimn}(R\Omega_{i}\setminus\Omega_{i}')+\gamma_{\adimn}(\Omega_{i}'\setminus R\Omega_{i})\Big)^{4}.$$
\end{example}

When $m>3$, the quantity $\sum_{i=1}^{m}\vnormtf{\int_{\Omega_{i}}(x-\overline{w}^{i})\gamma_{\adimn}(x)dx}^{2}$ should be maximized by the sets described in Conjecture \ref{conj0}, but this is still an open problem \cite{isaksson11}.  See Problem \ref{prob3} in Section \ref{secprop}.

\subsection{Noise Stability}

For applications in computer science and voting theory \cite{khot07,isaksson11}, a generalization of Conjecture \ref{conj0} is most relevant, where the Gaussian surface area of a set is replaced with a quantity referred to as \textit{noise stability} with parameter $\rho\in(-1,1)$.  In the limit as $\rho\to1^{-}$, the Gaussian surface area can be recovered from noise stability.  An analyst might refer to noise stability as (Gaussian) heat content.

Let $f\colon\R^{\adimn}\to[0,1]$ be measurable and let $\rho\in(-1,1)$,  define the \textit{Ornstein-Uhlenbeck operator with correlation $\rho$} applied to $f$ by
\begin{equation}\label{oudef}
\begin{aligned}
T_{\rho}f(x)
&\colonequals\int_{\R^{\adimn}}f(x\rho+y\sqrt{1-\rho^{2}})\gamma_{\adimn}dy\\
&=(1-\rho^{2})^{-(\adimn)/2}(2\pi)^{-(\adimn)/2}\int_{\R^{\adimn}}f(y)e^{-\frac{\vnorm{y-\rho x}^{2}}{2(1-\rho^{2})}}dy,
\qquad\forall x\in\R^{\adimn}.
\end{aligned}
\end{equation}
$T_{\rho}$ is a parametrization of the Ornstein-Uhlenbeck operator.  $T_{\rho}$ is not a semigroup, but it satisfies $T_{\rho_{1}}T_{\rho_{2}}=T_{\rho_{1}\rho_{2}}$ for all $\rho_{1},\rho_{2}\in(0,1)$.  We have chosen this definition since the usual Ornstein-Uhlenbeck operator is only defined for $\rho\in[0,1]$.

\begin{definition}[\embolden{Noise Stability}]\label{noisedef}
Let $\Omega\subset\R^{\adimn}$.  Let $\rho\in(-1,1)$.  We define the \textit{noise stability} of the set $\Omega$ with correlation $\rho$ to be
$$\int_{\R^{\adimn}}1_{\Omega}(x)T_{\rho}1_{\Omega}(x)\gamma_{\adimn}(x)dx
\stackrel{\eqref{oudef}}{=}(2\pi)^{-(\adimn)}(1-\rho^{2})^{-(\adimn)/2}\int_{\Omega}\int_{\Omega}e^{\frac{-\|x\|^{2}-\|y\|^{2}+2\rho\langle x,y\rangle}{2(1-\rho^{2})}}dxdy.$$
Equivalently, if $X=(X_{1},\ldots,X_{\adimn}),Y=(Y_{1},\ldots,Y_{\adimn})\in\R^{\adimn}$ are $(\adimn)$-dimensional jointly Gaussian distributed random vectors with $\E X_{i}Y_{j}=\rho\cdot1_{(i=j)}$ for all $i,j\in\{1,\ldots,\adimn\}$, then
$$\int_{\R^{\adimn}}1_{\Omega}(x)T_{\rho}1_{\Omega}(x)\gamma_{\adimn}(x)dx=\mathbb{P}((X,Y)\in \Omega\times \Omega).$$
\end{definition}

Problem \ref{prob1} can then be generalized as follows.

\begin{prob}[\embolden{Standard Simplex Problem}, {\cite{isaksson11}}]\label{prob2}
Let $m\geq3$.  Fix $a_{1},\ldots,a_{m}>0$ such that $\sum_{i=1}^{m}a_{i}=1$.  Fix $\rho\in(0,1)$.  Find measurable sets $\Omega_{1},\ldots\Omega_{m}\subset\R^{\adimn}$ with $\cup_{i=1}^{m}\Omega_{i}=\R^{\adimn}$ and $\gamma_{\adimn}(\Omega_{i})=a_{i}$ for all $1\leq i\leq m$ that maximize
$$\sum_{i=1}^{m}\int_{\R^{\adimn}}1_{\Omega_{i}}(x)T_{\rho}1_{\Omega_{i}}(x)\gamma_{\adimn}(x)dx,$$
subject to the above constraints.
\end{prob}

Conjecture \ref{conj0} is generalized to the following.

\begin{conj}[\embolden{Standard Simplex Conjecture} {\cite{isaksson11}}]\label{conj2}
Let $\Omega_{1},\ldots\Omega_{m}\subset\R^{\adimn}$ maximize Problem \ref{prob1}.  Assume that $m-1\leq\adimn$.  Fix $\rho\in(0,1)$.  Let $z_{1},\ldots,z_{m}\in\R^{\adimn}$ be the vertices of a regular simplex in $\R^{\adimn}$ centered at the origin.  Then $\exists$ $w\in\R^{\adimn}$ such that, for all $1\leq i\leq m$,
$$\Omega_{i}=w+\{x\in\R^{\adimn}\colon\langle x,z_{i}\rangle=\max_{1\leq j\leq m}\langle x,z_{j}\rangle\}.$$
\end{conj}

It is well known that, as $\rho\to1^{-}$, the noise stability (when normalized appropriately) converges to Gaussian surface area.  That is, if $\partial \Omega$ is a $C^{\infty}$ manifold, then \cite[Lemma 3.1]{kane11} \cite[Proposition 8.5]{ledoux96} \cite{de17}
$$
\lim_{\rho\to 1^{-}}\frac{\sqrt{2\pi}}{\cos^{-1}(\rho)}\left[\gamma_{\adimn}(\Omega)-\int_{\R^{\adimn}}1_{\Omega}(x)T_{\rho}1_{\Omega}(x)\gamma_{\adimn}(x)dx\right]
=\int_{\partial \Omega}\gamma_{\adimn}(x)dx.
$$
Consequently, Conjecture \ref{conj2} reduces to Conjecture \ref{conj0} when $\rho\to1^{-}$.  And we can use this relation together with Theorem \ref{mainthm2} to obtain the first known dimension-independent bounds for Conjecture \ref{conj2} when $\rho$ is close to $1$.

\begin{cor}[\embolden{Weak Plurality is Stablest/Standard Simplex Conjecture for Small Noise}]\label{mainthm3}
Let $\Omega_{1},\ldots,\Omega_{m}$ be cones over a regular simplex.  Let $w\in\R^{\adimn}$ be defined by \eqref{wdef}.  For all $1\leq i\leq m$, let $\overline{w}^{(i)}\colonequals w/a_{i}$.  Assume that $m-1\leq\adimn$.  Let $\Omega_{1}',\ldots,\Omega_{m}'$ such that $\gamma_{\adimn}(\Omega_{i})=\gamma_{\adimn}(\Omega_{i}')$ for all $1\leq i\leq m$ and $\cup_{i=1}^{m}\Omega_{i}'=\R^{\adimn}$.  Then for all $1/2<\rho<1$,
\begin{flalign*}
&\sum_{i=1}^{m}\int_{\R^{\adimn}}1_{\Omega_{i}'}(x)T_{\rho}1_{\Omega_{i}'}(x)\gamma_{\adimn}(x)dx
\leq\sum_{i=1}^{m}\int_{\R^{\adimn}}1_{\Omega_{i}}(x)T_{\rho}1_{\Omega_{i}}(x)\gamma_{\adimn}(x)dx+o(\sqrt{1-\rho^{2}})\\
&\qquad- \epsilon\sqrt{1-\rho^{2}}\sqrt{\frac{\pi}{2}}\sum_{i=1}^{m}\vnormt{\int_{\Omega_{i}}(x-\overline{w}^{(i)})\gamma_{\adimn}(x)dx}^{2}-\vnormt{\int_{\Omega_{i}'}(x-\overline{w}^{(i)})\gamma_{\adimn}(x)dx}^{2}.
\end{flalign*}
If $m\leq 4$, then we can choose $\epsilon\colonequals\epschoice$.  If $m\geq 5$, then $\exists$ $\epsilon>0$ such that the above inequality holds (and $\epsilon$ can depend on $a_{1},\ldots,a_{m}$ and on $m$ but not on $\adimn$.)
\end{cor}
\begin{remark}
The error term $o(\sqrt{1-\rho^{2}})$ can be written explicitly as
$$
\sum_{i=1}^{m}\sum_{j\in\{1,\ldots,m\}\colon j\neq i}
\int_{\eta=\rho}^{\eta=1}\frac{1}{\eta}\int_{\partial\Omega_{i}'}\Big\langle N_{ij}(x)-\overline{N}_{ij}(x), N_{ij}(x)\Big\rangle\gamma_{\adimn}(x)dxd\eta,
$$
where
$$
\overline{N}_{ij}(x)\colonequals\frac{\rho}{\sqrt{1-\rho^{2}}}\int_{\R^{\adimn}}y 1_{\Omega_{i}'^{c}}(\rho x+y\sqrt{1-\rho^{2}})\gamma_{\adimn}(y)dy,\qquad\forall,x\in\R^{\adimn}
$$
%\snote{Missing some constants}
\end{remark}
\begin{example}
For illustrative purposes, we consider just the case $m=3$ and $\gamma_{\adimn}(\Omega_{i})=1/3$ for all $1\leq i\leq 3$.  Using Example \ref{threex}, Corollary \ref{mainthm3} can then be written as
\begin{flalign*}
&\sum_{i=1}^{m}\int_{\R^{\adimn}}1_{\Omega_{i}'}(x)T_{\rho}1_{\Omega_{i}'}(x)\gamma_{\adimn}(x)dx
\leq\sum_{i=1}^{m}\int_{\R^{\adimn}}1_{\Omega_{i}}(x)T_{\rho}1_{\Omega_{i}}(x)\gamma_{\adimn}(x)dx+o(\sqrt{1-\rho^{2}})\\
&\qquad-10^{-10}\sqrt{1-\rho^{2}}\sqrt{\frac{\pi}{2}}\inf_{\substack{\mathrm{rotations}\\ R\colon\R^{\adimn}\to\R^{\adimn}}}
 \Big(\sum_{i=1}^{3}\gamma_{\adimn}(R\Omega_{i}\setminus\Omega_{i}')+\gamma_{\adimn}(\Omega_{i}'\setminus R\Omega_{i})\Big)^{4}
\end{flalign*}
\end{example}

The only previous bounds for $m\geq3$ in Problem \ref{prob2} seem to be the author's \cite{heilman12} result for $m=3$ and $\rho$ small (depending on the ambient dimension $\adimn$), and the estimates of \cite{de17,ghazi18} that bound the ambient dimension $\adimn$ at which the nearly optimal sets become computable.

Unfortunately, Conjecture \ref{conj2} is false whenever $(a_{1},\ldots,a_{m})\neq(1/m,\ldots,1/m)$ \cite{heilman14}.  Therefore, any proof of Conjecture \ref{conj2} when $(a_{1},\ldots,a_{m})=(1/m,\ldots,1/m)$ must somehow make special use of this assumption.  Almost paradoxically, the proof of Conjecture \ref{conj0} by \cite{milman18b} and the proof of the more general Theorem \ref{mainthm} in this work are oblivious to the choice of the measure restriction $(a_{1},\ldots,a_{m})$.  So, at present, it is unclear if the arguments of \cite{milman18a,heilman18,milman18b} can be extended to deal with all cases of Conjecture \ref{conj2}.

\subsection{Plurality is Stablest Conjecture}

The Standard Simplex Conjecture \cite{isaksson11} stated in Conjecture \ref{conj2} is essentially equivalent to an inequality for functions on discrete product spaces known as the Plurality is Stablest Conjecture.   After making several definitions, we state this conjecture in Conjecture \ref{prob4} below.  When $m=2$, Conjecture \ref{prob4} is known as the Majority is Stablest Theorem, proven in \cite{mossel10}.

If $g\colon\{1,\ldots,m\}^{\sdimn}\to\R$ and $1\leq i\leq\sdimn$, we denote
$$ \E(g)\colonequals m^{-\sdimn}\sum_{\omega\in\{1,\ldots,m\}^{\sdimn}} g(\omega)$$
$$\E_{i}(g)(\omega_{1},\ldots,\omega_{i-1},\omega_{i+1},\ldots,\omega_{\sdimn})\colonequals m^{-1}\sum_{\omega_{i}\in\{1,\ldots,m\}} g(\omega_{1},\ldots,\omega_{n})$$
$$\qquad\qquad\qquad\qquad\qquad\qquad\qquad\qquad\qquad\forall (\omega_{1},\ldots,\omega_{i-1},\omega_{i+1},\ldots,\omega_{\sdimn})\in\{1,\ldots,m\}^{\sdimn}.$$
$$\mathrm{Inf}_{i}(g)\colonequals \E [(g-\E_{i}g)^{2}].$$
Let
$$\Delta_{m}\colonequals\{(y_{1},\ldots,y_{m})\in\R^{m}\colon y_{1}+\cdots+y_{m}=1,\,\forall\,1\leq i\leq m,\,y_{i}\geq0\}.$$
If $f\colon\{1,\ldots,m\}^{\sdimn}\to\Delta_{m}$, we denote the coordinates of $f$ as $f=(f_{1},\ldots,f_{m})$.  For any $\omega\in\Z^{\sdimn}$, we denote $\vnormt{\omega}_{0}$ as the number of nonzero coordinates of $\omega$.  The \textbf{noise stability} of $g\colon\{1,\ldots,m\}^{\sdimn}\to\R$ with parameter $-1<\rho<1$ is
\begin{flalign*}
S_{\rho} g
&\colonequals m^{-\sdimn}\sum_{\omega\in\{1,\ldots,m\}^{\sdimn}} g(\omega)\E_{\rho} g(\delta)\\
&=m^{-\sdimn}\sum_{\omega\in\{1,\ldots,m\}^{\sdimn}} g(\omega)\sum_{\sigma\in\{1,\ldots,m\}^{\sdimn}}\left(\frac{1-(m-1)\rho}{m}\right)^{\sdimn-\vnormt{\sigma-\omega}_{0}}
\left(\frac{1-\rho}{m}\right)^{\vnormt{\sigma-\omega}_{0}} g(\sigma).
\end{flalign*}
Equivalently, conditional on $\omega$, $\E_{\rho}g(\delta)$ is defined so that for all $1\leq i\leq\sdimn$, $\delta_{i}=\omega_{i}$ with probability $\frac{1-(m-1)\rho}{m}$, and $\delta_{i}$ is equal to any of the other $(m-1)$ elements of $\{1,\ldots,m\}$ each with probability $\frac{1-\rho}{m}$, and so that $\delta_{1},\ldots,\delta_{\sdimn}$ are independent.

The \textbf{noise stability} of $f\colon\{1,\ldots,m\}^{\sdimn}\to\Delta_{m}$ with parameter $-1<\rho<1$ is
$$S_{\rho}f\colonequals\sum_{i=1}^{m}S_{\rho}f_{i}.$$

Let $m\geq2$, $k\geq3$.  For each $j\in\{1,\ldots,m\}$, let $e_{j}=(0,\ldots,0,1,0,\ldots,0)\in\R^{m}$ be the $j^{th}$ unit coordinate vector.  Define the \textbf{plurality} function $\mathrm{PLUR}_{m,\sdimn}\colon\{1,\ldots,m\}^{\sdimn}\to\Delta_{m}$ for $m$ candidates and $\sdimn$ voters such that for all $\omega\in\{1,\ldots,m\}^{\sdimn}$.
$$\mathrm{PLUR}_{m,\sdimn}(\omega)
\colonequals\begin{cases}
e_{j}&,\mbox{if }\abs{\{i\in\{1,\ldots,m\}\colon\omega_{i}=j\}}>\abs{\{i\in\{1,\ldots,m\}\colon\omega_{i}=r\}},\\
&\qquad\qquad\qquad\qquad\forall\,r\in\{1,\ldots,m\}\setminus\{j\}\\
\frac{1}{m}\sum_{i=1}^{m}e_{i}&,\mbox{otherwise}.
\end{cases}
$$

\begin{conj}[\embolden{Plurality is Stablest, Discrete Version}]\label{prob4}
%For any $m\geq2$, $\rho\in[-\frac{1}{m-1},1]$, $\epsilon>0$, there exists $\tau>0$ such that if $f\colon\{1,\ldots,m\}^{\sdimn}\to\Delta_{m}$ satisfies $\mathrm{Inf}_{i}(f_{j})\leq\tau$ for all $1\leq i\leq\sdimn$ and for all $1\leq j\leq m$, then
%\begin{itemize}
%\item[(a)]  If $\rho\in(0,1]$, and if $\frac{1}{k^{n}}\sum_{\sigma\in\{1,\ldots,k\}^{n}}f(\sigma)=\frac{1}{k}\sum_{i=1}^{k}e_{i}$, then
%$$
%S_{\rho}f\leq S_{\rho}\mathrm{PLUR}_{m,k}(\sigma), T_{\rho}(\mathrm{PLUR}_{m,k})(\sigma)\rangle+\epsilon.
%$$
%\item[(b)] If $\rho\in[-1/(k-1),0)$, then
%$$
%S_{\rho}f\geq S_{\rho}\mathrm{PLUR}_{m,k}-\epsilon.
%$$
%\end{itemize}
For any $m\geq2$, $\rho\in[0,1]$, $\epsilon>0$, there exists $\tau>0$ such that if $f\colon\{1,\ldots,m\}^{\sdimn}\to\Delta_{m}$ satisfies $\mathrm{Inf}_{i}(f_{j})\leq\tau$ for all $1\leq i\leq\sdimn$ and for all $1\leq j\leq m$, and if $\E f=\frac{1}{m}\sum_{i=1}^{m}e_{i}$, then
$$
S_{\rho}f\leq \lim_{\sdimn\to\infty}S_{\rho}\mathrm{PLUR}_{m,\sdimn}+\epsilon.
$$

%\end{itemize}
\end{conj}

Corollary \ref{mainthm3} also has an equivalent statement for functions on discrete product spaces.  However, the error term becomes rather unwieldy so we will not write down the discrete version of Corollary \ref{mainthm3}.  In order to obtain such a result, one begins with a function $f\colon\{1,\ldots,m\}^{\sdimn}\to\Delta_{m}$ and then applies the procedure of \cite[Lemma 2.14, Lemma 2.15, Theorem 7.1]{isaksson11}: express $f$ as a multilinear polynomial of products of an orthonormal basis of $\{1,\ldots,m\}$,  smooth the polynomial slightly by applying the discrete Ornstein-Uhlenbeck operator, project its values (in $\R^{m}$) to the closest point in $\Delta_{m}$, and then adjust the values of the smoothed polynomial slightly so it takes values in the extreme points of $\Delta_{m}$.  So, the sets $\Omega_{1}',\ldots,\Omega_{m}'$ appearing in the error term from Corollary \ref{mainthm3} would be
$$\Omega_{i}'\colonequals\{x\in\R^{\sdimn(m-1)}\colon \vnorm{Q_{f}(x)-e_{i}}=\min_{j\in\{1,\ldots,m\}}\vnorm{Q_{f}(x)-e_{j}}\},\,\,\forall\,1\leq i\leq m,$$
where $Q_{f}$ is the (slightly smoothed) multilinear polynomial defined by $f$; an additional small error would appear in the inequality as well.  We omit the details and refer instead to \cite{isaksson11}

\subsection{Propeller Conjecture}\label{secprop}

A variant of the following conjecture was stated in \cite{khot09,khot11}.  See \cite{khot09,khot11}where this problem is motivated by a kernel clustering problem from machine learning and by generalized Grothendieck inequalities.  Problem \ref{prob3} is more closely related to the $\rho\to0$ case of Conjecture \ref{conj2}.

\begin{prob}\label{prob3}
Let $m>3$.  Fix $a_{1},\ldots,a_{m}>0$ such that $\sum_{i=1}^{m}a_{i}=1$.  Let $w\in\R^{\adimn}$ be defined by \eqref{wdef}.  For all $1\leq i\leq m$, let $\overline{w}^{(i)}\colonequals w/a_{i}$.  Find measurable sets $\Omega_{1},\ldots\Omega_{m}\subset\R^{\adimn}$ with $\cup_{i=1}^{m}\Omega_{i}=\R^{\adimn}$ and $\gamma_{\adimn}(\Omega_{i})=a_{i}$ for all $1\leq i\leq m$ that maximize
$$\pen,$$
subject to the above constraints.
\end{prob}

\begin{conj}\label{conj5}
Assume that $m-1\leq\adimn$.  The sets $\Omega_{1},\ldots\Omega_{m}\subset\R^{\adimn}$ maximizing Problem \ref{prob3} are simplicial cones over a regular simplex.
\end{conj}

We have restricted Problem \ref{prob3} to $m>3$ since the case $m=3$ was solved already in \cite{khot09,khot11}.  As it is stated, this problem is not a finite-dimensional optimization problem.  However, a first variation argument (see \eqref{nine3} or \cite[Lemma 3.3]{khot09}) implies that the optimal sets in Problem \ref{prob3} have boundaries that are the union of a bounded number of simplices.  Therefore, Problem \ref{prob3} really is a finite-dimensional optimization problem.

\subsection{Symmetric Sets}

The following problem apparently appeared first in \cite{barthe01}.  See \cite{heilman17} for a discussion of its significance and relation to other problems; see also \cite{barchiesi18} for a subsequent result that solves Problem \ref{prob6} for values of $a$ close to $0$ or $1$ by modifying and extending the techniques of \cite{barchiesi16}.

\begin{prob}[\embolden{Symmetric Gaussian Problem}, {\cite{barthe01}}]\label{prob6}
Let $0<a<1$.  Find measurable $\Omega\subset\R^{\adimn}$ with $\Omega=-\Omega$ and $\gamma_{\adimn}(\Omega)=a$ that minimizes
$$\int_{\partial\Omega}\gamma_{\sdimn}(x)dx.$$
\end{prob}

As discussed in \cite{heilman17}, if $p_{\adimn}=p_{\adimn}(a)$ is the minimum value of Problem \ref{prob6} for fixed $\adimn$ and fixed $0<a<1$, it is expected that $p_{\adimn}=p_{\sdimn+2}=p_{\sdimn+3}=\cdots$ for sufficiently large $\adimn$, unless $a=1/2$.

We point out an improvement to the main result of \cite{heilman17} that removes the convexity assumption there.

\begin{theorem}\label{symthm}
Let $0<a<1$.  Assume that $p_{\adimn}(a)=p_{\sdimn+2}(a)=p_{\sdimn+3}(a)=\cdots$ for sufficiently large $\adimn$.  Let $\Omega$ minimize Problem \ref{prob1} and let $\Sigma\colonequals\partial\Omega$.  If
$$\int_{\Sigma}(\vnormt{A_{x}}^{2}-1)\gamma_{\sdimn}(x)dx>0,$$
or
$$\int_{\Sigma}(\vnormt{A_{x}}^{2}-1+2\sup_{y\in\Sigma}\vnormt{A_{y}}_{2\to 2}^{2})\gamma_{\sdimn}(x)dx<0,$$
then, after rotating $\Omega$, $\exists$ $r>0$ and $\exists$ $0\leq k\leq \sdimn$ so that $\Sigma=r S^{k}\times\R^{\sdimn-k}$.
\end{theorem}

In this Theorem, $S^{k}\colonequals\{(x_{1},\ldots,x_{k+1})\in\R^{k+1}\colon x_{1}^{2}+\cdots+x_{k+1}^{2}=1\}$ and $A=A_{x}$ is the second fundamental form of $x\in\Sigma$ (see Section \ref{seccurvature} and \eqref{three1}.)

Thanks to a discussion with Frank Morgan, there is a very specific guess for the optimal sets in Problem \ref{prob6}.  That is, there exists a sequence of intervals $\cdots[a_{3},a_{2})\cup[a_{2},a_{1})\cup[a_{1},1]$ with $a_{0}=1$ such that $[1/2,1]\supset\cup_{N=1}^{\infty}[a_{N},a_{N-1}]$ and such that, for every $N\geq1$, and for every $a\in[a_{N},a_{N-1}]$, the minimum value of Problem \ref{prob6} is achieved by the ball centered at the origin in $\R^{N}$ (for any $\adimn$).

\subsection{Organization}

\begin{itemize}
\item A simplified proof of the Multi-Bubble Theorem (Conjecture \ref{conj0}) is in Section \ref{secsim}.
\item Theorem \ref{symthm} is proven at the end of Section \ref{secsim}.
\item Ancillary results for the Main Theorem \ref{mainthm} appear in ensuing sections, especially Sections \ref{secred} and \ref{secflat}.
\item The Main Theorem \ref{mainthm} is proven in Section \ref{secmain}.  Theorem \ref{mainthm2} is an immediate consequence of Theorem \ref{mainthm}.
\item Corollary \ref{mainthm3} is proven in Section \ref{secnoise}.
\end{itemize}

\section{Simplified Proof of Multi-Bubble Conjecture}\label{secsim}

In this section, we combine parts of the arguments of \cite{milman18b,heilman18} and \cite{mcgonagle15}, providing a simplified proof of Conjecture \ref{conj0}.  Some of the material is repeated from \cite{heilman18} in order for this proof to be self-contained.

\begin{proof}[Proof of Conjecture \ref{conj0}]
For brevity, we occasionally omit the arguments of functions and vector fields.

Let $\Omega_{1},\ldots,\Omega_{m}\subset\R^{\adimn}$ minimize Problem \ref{prob1}.  For any $1\leq i<j\leq m$, denote $\Sigma_{ij}\colonequals(\redb\Omega_{i})\cap(\redb\Omega_{j})$ (see Definition \ref{defred}) and for any $x\in\Sigma_{ij}$, let $N_{ij}\in \R^{\adimn}$ be the normal vector at $x$ pointing from $\Omega_{i}$ into $\Omega_{j}$ with unit length, so that $\vnorm{N_{ij}}=1$.  Let $X\colon\R^{\adimn}\to\R^{\adimn}$ be a vector field and $\forall$ $1\leq i<j\leq m$ denote $f_{ij}\colonequals\langle X, N_{ij}\rangle$.  For any $s\in(-1,1)$, and for any $1\leq i\leq m$, we consider the sets perturbed according to the vector field $X$:
$$\Omega_{i}^{(s)}\colonequals \{x+sX(x)\colon x\in\Omega_{i}\}.$$
Standard calculations (Lemma \ref{lemma10}) give the first and second derivatives with respect to $s$ of the Gaussian volumes and surface areas of these sets \cite[Lemma 3.2]{heilman18}.  For example,
\begin{equation}\label{zero1}
\frac{d}{ds}|_{s=0}\gamma_{\adimn}(\Omega_{i}^{(s)})=\int_{\redb\Omega_{i}}\sum_{j\in\{1,\ldots,m\}\colon j\neq i}\langle X,N_{ij}\rangle\gamma_{\adimn}(x)dx,\qquad\forall\,1\leq i\leq m.
\end{equation}
If $X$ is chosen to be volume-preserving, i.e. if $\frac{d}{ds}|_{s=0}\gamma_{\adimn}(\Omega_{i}^{(s)})=0$ for all $1\leq i\leq m$, then the second derivative of Gaussian surface area satisfies \cite[Lemma 3.10]{heilman18} (Lemma \ref{lemma28} with $\epsilon=0$)
\begin{equation}\label{zero2}
\begin{aligned}
&\frac{d^{2}}{ds^{2}}|_{s=0}\sum_{1\leq i<j\leq m}\int_{\Sigma_{ij}^{(s)}}\gamma_{\sdimn}(x)dx
=\sum_{1\leq i<j\leq m}-\int_{\Sigma_{ij}}f_{ij}L_{ij}f_{ij}\gamma_{\sdimn}(x)dx\\
&\qquad\qquad\qquad+\sum_{1\leq i<j<k\leq m}\int_{\redS_{ij}\cap\redS_{jk}\cap \redS_{ki}}
\Big([\nabla_{\nu_{ij}}f_{ij}+q_{ij}f_{ij}]f_{ij}+[\nabla_{\nu_{jk}}f_{jk}+q_{jk}f_{jk}]f_{jk}\\
 &\qquad\qquad\qquad\qquad\qquad\qquad\qquad\qquad\qquad+[\nabla_{\nu_{ki}}f_{ki}+q_{ki}f_{ki}]f_{ki}\Big)\gamma_{\sdimn}(x)dx.
\end{aligned}
\end{equation}
Here
\begin{equation}\label{zero4}
L_{ij}\colonequals \Delta-\langle x,\nabla \rangle+\vnorm{A}^{2}+1,\qquad
\mathcal{L}_{ij}\colonequals \Delta-\langle x,\nabla \rangle,\qquad
\forall\,1\leq i<j\leq m,
\end{equation}
where $\Delta$ is the Laplacian (sum of second derivatives) on the surface $\Sigma_{ij}$, $\nabla$ is the gradient on $\Sigma_{ij}$, $A=A_{x}$ is the matrix of first order partial derivatives of the unit exterior normal vector $N(x)$ (see \eqref{three1}), $\vnorm{A}^{2}$ is the sum of the squares of entries of $A$ at $x\in\Sigma_{ij}$, $\nu_{ij}$ is the unit exterior normal to $\partial\Sigma_{ij}$ and $q_{ij}\colonequals[\langle \nabla_{\nu_{kj}}\nu_{kj},N_{kj}\rangle+\langle \nabla_{\nu_{ki}}\nu_{ki},N_{ki}\rangle]/\sqrt{3}$, so that
\begin{equation}\label{zero1.5}
q_{ij}+q_{jk}+q_{ki}=0,\qquad\forall\,1\leq i<j<k\leq m,
\end{equation}
since $N_{ij}=-N_{ji}$ by Definition of $N_{ij}$.% and $q_{ij}=q_{ji}$.

We now state two Lemmas.  Assuming these Lemmas, we prove Conjecture \ref{conj0}, thereby simplifying the proof of \cite{milman18b}.

\begin{lemma}[\embolden{Dimension Reduction}]\label{dimredthm}
Suppose $\Omega_{1},\ldots\Omega_{m}\subset\R^{\adimn}$ minimize Problem \ref{prob1}.  Then there exists $0\leq \ell\leq m-1$ and there exist $\Omega_{1}',\ldots,\Omega_{m}'\subset\R^{\ell}$ such that, after rotating $\Omega_{1},\ldots,\Omega_{m}$, we have
$$\Omega_{i}=\Omega_{i}'\times\R^{\sdimn-\ell+1}.$$
Moreover $\ell$ can be chosen to be the dimension of the span of
$$\Big\{\Big(\int_{\redb\Omega_{1}}\sum_{\substack{j\in\{1,\ldots,m\}\colon\\ j\neq 1}}\langle v,N_{1j}\rangle\gamma_{\sdimn}(x)dx,
\ldots,\int_{\redb\Omega_{m}}\sum_{\substack{j\in\{1,\ldots,m\}\colon\\ j\neq m}}\langle v,N_{mj}\rangle\gamma_{\sdimn}(x)dx\Big)\in\R^{m}\colon v\in\R^{\adimn}\Big\}.
$$
\end{lemma}

\begin{lemma}[\embolden{Flatness}]\label{nolowdim}
For all $1\leq i\leq m$, $\partial\Omega_{i}$ consists of a union of relatively open subsets of hyperplanes.
\end{lemma}

We now complete the proof of Conjecture \ref{conj0}.  By Lemma \ref{nolowdim}, $L_{ij}$ simplifies to just $\Delta-\langle x,\nabla\rangle+1$.  For every connected component $C$ of $\Omega_{1},\ldots,\Omega_{m}$, let $X_{C}$ be the vector field that is equal to the exterior unit normal vector field of $C$ on $\redb C$ and $X_{C}=0$ on every other connected component of $\Omega_{1},\ldots,\Omega_{m}$ that does not intersect $C$.  (Such a vector field exists by Lemma \ref{lemma52.6}.)  Let $U$ be the linear span of all such vector fields $X_{C}$, as $C$ ranges over the connected components of $\Omega_{1},\ldots,\Omega_{m}$. Note that $\frac{d^{2}}{ds^{2}}|_{s=0}\sum_{1\leq i<j\leq m}\int_{\Sigma_{ij}^{(s)}}\gamma_{\sdimn}(x)dx<0$ for all nonzero $X\in U$ by \eqref{zero2}, since the second term of \eqref{zero2} is zero for all $X\in U$.   If there are more than $m$ connected components of $\Omega_{1},\ldots,\Omega_{m}$, then we can form a nontrivial linear combination of these vector fields to get nonzero $(f_{ij})_{1\leq i<j\leq m}$ such that $\frac{d}{ds}|_{s=0}\gamma_{\adimn}(\Omega_{i}^{(s)})=0$ for all $1\leq i\leq m-1$ (and also for $i=m$ since $\cup_{i=1}^{m}\Omega_{i}=\R^{\adimn}$), and $\frac{d^{2}}{ds^{2}}|_{s=0}\sum_{1\leq i<j\leq m}\int_{\Sigma_{ij}^{(s)}}\gamma_{\sdimn}(x)dx<0$.  This contradicts that $\Omega_{1},\ldots,\Omega_{m}$ minimize Problem \ref{prob1}.  So, there must be exactly $m$ connected components of $\Omega_{1},\ldots,\Omega_{m}$.  The regularity condition, Lemma \ref{lemma52.6} then concludes the proof.  We know that each of $\Omega_{1},\ldots,\Omega_{m}$ is connected with flat boundary pieces, the sets $\Omega_{1},\ldots,\Omega_{m}$ meet in threes at $120$ degree angles, and they meet in fours like the cone over the three-dimensional regular simplex.  In the case $m=3$, there are only three possible configurations for the sets (up to rotation).  And for general $m$, there are only finitely many possible configurations of the sets (up to rotation).  So, we can conclude the proof by either (i) checking the Gaussian surface area of each such case directly, (ii) using the matrix-valued partial differential inequality and maximal principle argument of \cite{milman18b}, or (iii) appealing directly to the result of \cite{milman18b}.
\end{proof}
\begin{remark}
The last paragraph of the above proof actually classifies all stable local minima of $\sum_{1\leq i<j\leq m}\int_{\Sigma_{ij}}\gamma_{\sdimn}(x)dx$.  Such stable local minima must satisfy:
\begin{itemize}
\item $\forall$ $1\leq i\leq m$, $\Omega_{i}$ has exactly one connected component.
\item $\forall$ $1\leq i\leq m$, $\partial\Omega_{i}$ is the union of at most $m-1$ relatively open subsets of hyperplanes.
\item Assumption \ref{as1} holds (i.e. the regularity result stated there holds).
\end{itemize}
For a fixed $m$, there are then only finitely many stable local minima of the Gaussian surface area, up to rotations.
\end{remark}
\begin{remark}
This statement is still insufficient to prove the stronger Theorem \ref{mainthm2}.  A priori, there could be a sequence of sets with nearly minimal Gaussian perimeter that only become close to the minimal sets at a very slow rate.
\end{remark}

\begin{proof}[Proof of Lemma \ref{dimredthm}]
For any $v\in\R^{\adimn}$, define
$$T(v)\colonequals\Big(\int_{\redb\Omega_{1}}\sum_{j\in\{1,\ldots,m\}\colon j\neq 1}\langle v,N_{1j}\rangle\gamma_{\sdimn}(x)dx,
\ldots,\int_{\redb\Omega_{m}}\sum_{j\in\{1,\ldots,m\}\colon j\neq m}\langle v,N_{mj}\rangle\gamma_{\sdimn}(x)dx\Big).$$
Then $T\colon\R^{\adimn}\to\R^{m}$ is linear.  By the rank-nullity theorem, the dimension of the kernel of $T$ plus the dimension of the image of $T$ is $\adimn$.  Since the sum of the indices of $T(v)$ is zero for any $v\in\R^{\adimn}$ (since $N_{ij}=-N_{ji}$ $\forall$ $1\leq i<j\leq m$ by Definition \ref{defnote}), the dimension $\ell$ of the image of $T$ is at most $m-1$.

Let $v$ in the kernel of $T$.  For any $1\leq i<j\leq m$, let $f_{ij}\colonequals\phi\langle v,N_{ij}\rangle$.  Let $X\colonequals \phi v$ be the chosen vector field.  Since $\Omega_{1},\ldots,\Omega_{m}$ minimize Problem \ref{prob1},
$$0\leq \frac{d^{2}}{ds^{2}}|_{s=0}\sum_{1\leq i<j\leq m}\int_{\Sigma_{ij}^{(s)}}\gamma_{\sdimn}(x)dx.$$
From Lemmas \ref{lemma28} (with $\epsilon=w=0$), \ref{lemma29.9}, \ref{lemma97}, \ref{lemma27}, and then letting $\phi$ increase monotonically to $1$ (as in Lemma \ref{lemma97}),
\begin{equation}\label{zero2.555}
0\leq\sum_{1\leq i<j\leq m}-\int_{\Sigma_{ij}}f_{ij}L_{ij}f_{ij}\gamma_{\sdimn}(x)dx.
\end{equation}
When $v\in\R^{\adimn}$ is a constant and $X\colonequals v$ is a constant vector field, a calculation shows that \cite{mcgonagle15,barchiesi16} \cite[Lemma 4.2]{heilman17}, Lemma \ref{lemma45} with $\epsilon=0$,
\begin{equation}\label{zero2.5}
L_{ij}\langle v,N_{ij}\rangle=\langle v,N_{ij}\rangle,\qquad\forall\,1\leq i<j\leq m.
\end{equation}
That is, $(\langle v, N_{ij}\rangle)_{1\leq i<j\leq m}$ is an eigenfunction of $(L_{ij})_{1\leq i<j\leq m}$ with eigenvalue $1$ (see Remark \ref{rk30}).  So, \eqref{zero2.555} says
$$0\leq\sum_{1\leq i<j\leq m}-\int_{\Sigma_{ij}}f_{ij}^{2}\gamma_{\sdimn}(x)dx.$$
The last quantity must then be zero.  In summary, for any $v$ in the kernel of $T$, $\forall$ $1\leq i<j\leq m$, $f_{ij}(x)=\langle v,N_{ij}(x)\rangle=0$ for all $x\in\Sigma_{ij}$.  That is, $\exists$ $0\leq\ell\leq m-1$ as stated in the conclusion of Theorem \ref{dimredthm}, since the image of $T$ is the span of
$$\Big\{\Big(\int_{\redb\Omega_{1}}\sum_{\substack{j\in\{1,\ldots,m\}\colon\\ j\neq 1}}\langle v,N_{1j}\rangle\gamma_{\sdimn}(x)dx,
\ldots,\int_{\redb\Omega_{m}}\sum_{\substack{j\in\{1,\ldots,m\}\colon\\ j\neq m}}\langle v,N_{mj}\rangle\gamma_{\sdimn}(x)dx\Big)\in\R^{m}\colon v\in\R^{\adimn}\Big\}.$$

\end{proof}

\begin{proof}[Proof of Lemma \ref{nolowdim}]  %  n+2-m>0.  n+1>m-1
Assume for now that $\adimn>m-1$ so that $\sdimn-\ell+1\geq\sdimn-(m-1)+1>0$ in Theorem \ref{dimredthm}.  Then after rotating $\Omega_{1},\ldots,\Omega_{m}$, we have
$$\Omega_{i}=\Omega_{i}'\times\R.$$
Let $\{\alpha_{ij}\}_{1\leq i<j\leq m}$ be constants guaranteed to exist by Lemma \ref{lemma55}.  Let $f_{ij}\colonequals\alpha_{ij}$ for all $1\leq i<j\leq m$. Define now a new function $g_{ij}\colonequals x_{\adimn}f_{ij}$.  Since $f_{ij}$ is only a function of the variables $x_{1},\ldots,x_{\sdimn}$, we have $\langle\nabla f,\nabla(x_{\adimn})\rangle=0$.  So, for any $1\leq i<j\leq m$, the product rule for $L_{ij}$ (Remark \ref{rk20}) gives
$$
L_{ij}g_{ij}\stackrel{\eqref{zero4}}{=}x_{\adimn}L_{ij}f_{ij}+f_{ij}\mathcal{L}_{ij}x_{\adimn}+\langle\nabla f,\nabla(x_{\adimn})\rangle
=x_{\adimn}L_{ij}f_{ij}-f_{ij}x_{\adimn}
\stackrel{\eqref{zero4}}{=}\vnormt{A}^{2}g_{ij}.
$$
Here $\mathcal{L}_{ij}\colonequals\Delta-\langle x,\nabla\rangle$.  By \eqref{zero1} and Fubini's Theorem, if we first integrate with respect to $x_{\adimn}$, we see that $(g_{ij})_{1\leq i<j\leq m}$ is automatically Gaussian volume-preserving, so that \eqref{zero1} is zero for all $1\leq i<j\leq m$.  Then the second-variation condition (Lemmas \ref{lemma28} with $\epsilon=w=0$) applies, and we get by \eqref{zero2}
\begin{flalign*}
&\frac{d^{2}}{ds^{2}}|_{s=0}\sum_{1\leq i<j\leq m}\int_{\Sigma_{ij}^{(s)}}\gamma_{\sdimn}(x)dx
= -\sum_{1\leq i<j\leq m}\int_{\Sigma_{ij}}x_{\adimn}^{2}\alpha_{ij}^{2}\vnormt{A}^{2}\gamma_{\sdimn}(x)dx\\
&\qquad\qquad\qquad+\sum_{1\leq i<j<k\leq m}\int_{\redS_{ij}\cap\redS_{jk}\cap \redS_{ki}}
x_{\adimn}^{2}\Big(\alpha_{ij}^{2}q_{ij}+\alpha_{jk}^{2}q_{jk}+\alpha_{ki}^{2}q_{ki}\Big)\gamma_{\sdimn}(x)dx.
\end{flalign*}
After integrating in $x_{\adimn}$ and applying Fubini's Theorem, this becomes
\begin{flalign*}
&\frac{d^{2}}{ds^{2}}|_{s=0}\sum_{1\leq i<j\leq m}\int_{\Sigma_{ij}^{(s)}}\gamma_{\sdimn}(x)dx
= -\sum_{1\leq i<j\leq m}\int_{\Sigma_{ij}'}\alpha_{ij}^{2}\vnormt{A}^{2}\gamma_{\sdimn}(x)dx\\
&\qquad\qquad\qquad+\sum_{1\leq i<j<k\leq m}\int_{\redS_{ij}'\cap\redS_{jk}'\cap \redS_{ki}'}
\Big(\alpha_{ij}^{2}q_{ij}+\alpha_{jk}^{2}q_{jk}+\alpha_{ki}^{2}q_{ki}\Big)\gamma_{\sdimn}(x)dx.
\end{flalign*}
By summing over all circular permutations of $(\alpha_{ij})_{1\leq i<j\leq m}$, the last term becomes zero by \eqref{zero1.5}.  So, there exists $(\alpha_{ij})_{1\leq i<j\leq m}$ such that the last term is nonpositive.  That is, there exists nonzero $(\alpha_{ij})_{1\leq i<j\leq m}$ such that
$$
\frac{d^{2}}{ds^{2}}|_{s=0}\sum_{1\leq i<j\leq m}\int_{\Sigma_{ij}^{(s)}}\gamma_{\sdimn}(x)dx
\leq -\sum_{1\leq i<j\leq m}\int_{\Sigma_{ij}}\alpha_{ij}^{2}\vnormt{A}^{2}\gamma_{\sdimn}(x)dx
$$

So, in the case $\adimn>m-1$, we must have $\vnormt{A}=0$ identically, for all $1\leq i<j\leq m$.  Since the minimum value of Problem \ref{prob1} can only decrease as $\adimn$ increases, we also have $\vnormt{A}=0$ in the remaining case $\adimn=m-1$.

\end{proof}

\subsection{Symmetric Sets}

It seems appropriate to now prove Theorem \ref{symthm}.

\begin{proof}[Proof of Theorem \ref{symthm}]
Let $\Omega\subset\R^{\adimn}$ minimize Problem \ref{prob6}.  Let $\Sigma\colonequals\partial\Omega$.  The second inequality of the Theorem was already shown in \cite{heilman17}.  So, it suffices to show the first inequality.  It further suffices by the argument of \cite[Corollary 6.2]{heilman17} to show that $\delta\leq2$ where
$$
\delta=\delta(\Sigma)
\colonequals\sup_{\substack{f\colon\Sigma\to\R\\ f\,\,\mathrm{a}\,\,C^{\infty}\,\,\mathrm{compactly}\\\mathrm{supported}\,\,\mathrm{function}}}
\frac{\int_{\Sigma}fLf\gamma_{\sdimn}(x)dx}{\int_{\Sigma}f^{2}\gamma_{\sdimn}(x)dx}.
$$
We therefore argue by contradiction.  Assume $\delta>2$.  From \cite[Lemma 5.1]{heilman17}, there is a nonzero compactly supported Dirichlet eigenfunction $f$ of $L$ such that $Lf=\delta f$ and such that $f(x)=f(-x)$ for all $x\in\Sigma$.    By the assumption of Theorem \ref{symthm}, we may assume that $\Omega=\Omega'\times\R$ and $f$ does not depend on $x_{\adimn}$.  Consider the function $g(x)\colonequals(x_{\adimn}^{2}-1)f(x)$.  This function also satisfies $g(x)=g(-x)$ for all $x\in\R^{\adimn}$.  By Fubini's Theorem, this function automatically preserves the Gaussian volume of $\Omega$, i.e.
$$\int_{\Sigma}(x_{\adimn}^{2}-1)f\gamma_{\sdimn}(x)dx=\int_{\R}(x_{\adimn}^{2}-1)\gamma_{1}(x_{\adimn})\int_{\Sigma'}f\gamma_{\sdimn-1}(x_{1},\ldots,x_{\sdimn})dx_{1}\cdots dx_{\sdimn}=0.$$
As in the above proof, since $f$ does not depend on $x_{\adimn}$, $\langle\nabla f,\nabla(x_{\adimn}^{2}-1)\rangle=0$, and by the product rule for $L$ (Remark \ref{rk20}) we have
\begin{flalign*}
Lg&\stackrel{\eqref{zero4}}{=}(x_{\adimn}^{2}-1)Lf+f\mathcal{L}(x_{\adimn}^{2}-1)+\langle\nabla f,\nabla(x_{\adimn}^{2}-1)\rangle\\
&=(x_{\adimn}^{2}-1)Lf+f(2-2x_{\adimn}^{2})
\stackrel{\eqref{zero4}}{=}(\delta-2)g.
\end{flalign*}
From Lemmas \ref{lemma28} (with $\epsilon=w=0$) and \ref{lemma27},
$$\frac{d^{2}}{ds^{2}}|_{s=0}\int_{\Sigma^{(s)}}\gamma_{\sdimn}(x)dx=\int_{\Sigma}-fLf\gamma_{\sdimn}(x)dx=-(\delta-2)\int_{\Sigma}f^{2}\gamma_{\sdimn}(x)dx<0.$$
This inequality violates the minimality of $\Omega$, achieving a contradiction, and completing the proof.
\end{proof}

\section{Preliminaries and Notation}\label{secpre}

We say that $\Sigma\subset\R^{\adimn}$ is an $\sdimn$-dimensional $C^{\infty}$ manifold with boundary if $\Sigma$ can be locally written as the graph of a $C^{\infty}$ function on a relatively open subset of $\{(x_{1},\ldots,x_{\sdimn})\in\R^{\sdimn}\colon x_{\sdimn}\geq0\}$.  For any $(\adimn)$-dimensional $C^{\infty}$ manifold $\Omega\subset\R^{\adimn}$ with boundary, we denote
\begin{equation}\label{c0def}
\begin{aligned}
C_{0}^{\infty}(\Omega;\R^{\adimn})
&\colonequals\{f\colon \Omega\to\R^{\adimn}\colon f\in C^{\infty}(\Omega;\R^{\adimn}),\, f(\partial\partial \Omega)=0,\\
&\qquad\qquad\qquad\exists\,r>0,\,f(\Omega\cap(B(0,r))^{c})=0\}.
\end{aligned}
\end{equation}
We also denote $C_{0}^{\infty}(\Omega)\colonequals C_{0}^{\infty}(\Omega;\R)$.  We let $\mathrm{div}$ denote the divergence of a vector field in $\R^{\adimn}$.  For any $r>0$ and for any $x\in\R^{\adimn}$, we let $B(x,r)\colonequals\{y\in\R^{\adimn}\colon\vnormt{x-y}\leq r\}$ be the closed Euclidean ball of radius $r$ centered at $x\in\R^{\adimn}$.  Here $\partial\partial\Omega$ refers to the $(\sdimn-1)$-dimensional boundary of $\Omega$.

\begin{definition}[\embolden{Reduced Boundary}]\label{defred}
A measurable set $\Omega\subset\R^{\adimn}$ has \embolden{locally finite surface area} if, for any $r>0$,
$$\sup\left\{\int_{\Omega}\mathrm{div}(X(x))dx\colon X\in C_{0}^{\infty}(B(0,r),\R^{\adimn}),\, \sup_{x\in\R^{\adimn}}\vnormt{X(x)}\leq1\right\}<\infty.$$
Equivalently, $\Omega$ has locally finite surface area if $\nabla 1_{\Omega}$ is a vector-valued Radon measure such that, for any $x\in\R^{\adimn}$, the total variation
$$
\vnormt{\nabla 1_{\Omega}}(B(x,1))
\colonequals\sup_{\substack{\mathrm{partitions}\\ C_{1},\ldots,C_{m}\,\mathrm{of}\,B(x,1) \\ m\geq1}}\sum_{i=1}^{m}\vnormt{\nabla 1_{\Omega}(C_{i})}
$$
is finite \cite{cicalese12}.  If $\Omega\subset\R^{\adimn}$ has locally finite surface area, we define the \embolden{reduced boundary} $\redb \Omega$ of $\Omega$ to be the set of points $x\in\R^{\adimn}$ such that
$$N(x)\colonequals-\lim_{r\to0^{+}}\frac{\nabla 1_{\Omega}(B(x,r))}{\vnormt{\nabla 1_{\Omega}}(B(x,r))}$$
exists, and it is exactly one element of $S^{\sdimn}\colonequals\{x\in\R^{\adimn}\colon\vnorm{x}=1\}$.
\end{definition}

\begin{lemma}[\embolden{Existence}]\label{lemma51p}
There exist measurable $\Omega_{1},\ldots,\Omega_{m}\subset\R^{\adimn}$ minimizing Problem \ref{prob1p}.
\end{lemma}

\begin{lemma}[\embolden{Regularity}]\label{lemma52.6}
Let $\Omega_{1},\ldots\Omega_{m}$ minimize Problem \ref{prob1p}.  Then Assumption \ref{as1} holds.
\end{lemma}
\begin{proof}
We need to verify that the sets $\Omega_{1},\ldots,\Omega_{m}$ satisfy the assumptions of the regularity results of \cite{david09,david10} for $m=2,3$,  \cite{white90s} announced for any $m\geq2$, proved in \cite{colombo17}.
%\snote{Need to show almost minimal, as in Barchiesi}
In this proof, we let $P$ denote the surface area of a set, and we let $P_{\gamma}$ denote the Gaussian surface area of a set.  Fix $x\in\partial\Omega_{i'}$ for some $i'\in\{1,\ldots,m\}$ and fix $0<r<1$.  Since $\Omega_{1},\ldots\Omega_{m}$ minimize problem \ref{prob1p}, for every disjoint $F_{1},\ldots,F_{m}\subset\R^{\adimn}$ with locally finite perimeter such that $F_{i}\Delta \Omega_{i}\subset B_{2r}(x)$ it holds that
\begin{equation}\label{reg1}
\sum_{i=1}^{m}P_{\gamma}(\Omega_{i}) \leq \sum_{i=1}^{m}P_{\gamma}(F_{i})  + C\gamma_{\adimn}(F_{i}\Delta\Omega_{i})
\end{equation}
for some constant $C$ depending on $\vnormt{x}$. If we choose $F_{1}\colonequals \Omega_{1}\cup  B_{r}(x)$ and $F_{i}\colonequals \Omega_{i}\setminus  B_{r}(x)$ for all $2\leq i\leq m$, we get
$$
\sum_{i=1}^{m}P_{\gamma}(\Omega_{i}) \leq P_{\gamma}(\Omega_{1}\cup B_{r}(x))+\sum_{i=2}^{m}P_{\gamma}(\Omega_{i}\setminus B_{r}(x))  + C\gamma_{\adimn}(B_{r}(x)).
$$
On the other hand, arguing as in \cite[Lemma 12.22]{maggi12},
$$
P_{\gamma}(\Omega_{1}\cup B_{r}(x)) +P_{\gamma}(\Omega_{1}\cap B_{r}(x)) \leq P_{\gamma}(\Omega_{1})  + P_{\gamma}(B_{r}(x)).
$$
$$
P_{\gamma}(\Omega_{i}\cup B_{r}(x)^{c}) +P_{\gamma}(\Omega_{i}\cap B_{r}(x)^{c}) \leq P_{\gamma}(\Omega_{i})  + P_{\gamma}(B_{r}(x)^{c}).\qquad\forall\,2\leq i\leq m.
$$
The previous three inequalities give
$$
P_{\gamma}(\Omega_{1}\cap B_{r}(x))+\sum_{i=2}^{m}P_{\gamma}(\Omega_{i}\cup B_{r}(x)^{c}) \leq \sum_{i=1}^{m} P_{\gamma}(B_{r}(x))+C\gamma_{\adimn}(B_{r}(x))\leq Cr^{\sdimn}.
$$
The left-hand side can be estimated by
$$
P_{\gamma}(\Omega_{1}\cap B_{r}(x))+\sum_{i=2}^{m}P_{\gamma}(\Omega_{i}\cup B_{r}(x)^{c}) \geq ce^{-\vnormt{x}^{2}}\sum_{i=1}^{m}P(\Omega_{i}\cap B_{r}(x)).
$$
Therefore we obtain
\begin{equation}\label{reg2}
\sum_{i=1}^{m}P(\Omega_{i}\cap B_{r}(x))\leq C_{0}r^{\sdimn}
\end{equation}
for some constant $C_{0}=C_{0}(\vnormt{x})$. Note that for every $y\in B_{r}(x)$ and $0<r<1$ we have
$$
\abs{e^{-\vnormt{x}^{2}/2}-e^{-\vnormt{y}^{2}/2}}\leq Cr.
$$
for some constant $C$. Therefore, \eqref{reg1} and \eqref{reg2} imply that for all sets $F_{1},\ldots,F_{m}\subset\R^{\adimn}$ with locally finite perimeter such that $F_{i}\Delta \Omega_{i}\subset B_{r}(x)$ and $0<r<1$ it holds that
$$
\sum_{i=1}^{m}P(\Omega_{i}\cap B_{r}(x)) \leq \sum_{i=1}^{m}P(F_{i}\cap B_{r}(x)) + Cr^{\adimn}.
$$
for some constant $C$ depending on $\vnormt{x}$.
\end{proof}

Let $Y_{1}\subset\R^{2}$ denote three half-lines meeting at a single point with $120$-degree angles between them.  Let $T'\subset\R^{3}$ be the one-dimensional boundary of a regular tetrahedron centered at the origin, and let $T_{2}\subset\R^{3}$ be the cone generated by $T'$, so that $T_{2}=\{rx\in\R^{3}\colon r\geq0,\, x\in T'\}$.

\begin{assumption}\label{as1}
The sets $\Omega_{1},\ldots\Omega_{m}\subset\R^{\adimn}$ satisfy the following conditions.  First, $\cup_{i=1}^{m}\Omega_{i}=\R^{\adimn}$, $\sum_{i=1}^{m}\gamma_{\adimn}(\Omega_{i})=1$.  Also, $\Sigma\colonequals\cup_{i=1}^{m}\partial\Omega_{i}$ can be written as the disjoint union $M_{\sdimn}\cup M_{\sdimn-1}\cup M_{\sdimn-2}\cup M_{\sdimn-3}$ where $0<\alpha<1$ and
\begin{itemize}
\item[(i)] $M_{\sdimn}$ is a locally finite union of embedded $C^{\infty}$ $\sdimn$-dimensional manifolds.
\item[(ii)] $M_{\sdimn-1}$ is a locally finite union of embedded $C^{\infty}$ $(\sdimn-1)$-dimensional manifolds, near which $M$ is locally diffeomorphic to $Y_{1}\times\R^{\sdimn-1}$.
\item[(iii)] $M_{\sdimn-2}$ is a locally finite union of embedded $C^{1,\alpha}$ $(\sdimn-2)$-dimensional manifolds, near which $M$ is locally diffeomorphic to $T_{2}\times\R^{\sdimn-2}$.
\item[(iv)] $M_{\sdimn-3}$ is relatively closed, $(\sdimn-3)$-rectifiable, with locally finite $(n-3)$-dimensional Hausdorff measure.
\end{itemize}
\end{assumption}

Below, when $\Sigma_{ij}\colonequals(\redb\Omega_{i})\cap(\redb\Omega_{j})$ for some $1\leq i<j\leq m$, we denote $\redS_{ij}\colonequals M_{\sdimn-2}\cap(\partial\Omega_{i})\cap(\partial\Omega_{j})$, where $M_{\sdimn-2}$ is defined in Assumption \ref{as1}.

\subsection{Submanifold Curvature}\label{seccurvature}

Here we cover some basic definitions from differential geometry of submanifolds of Euclidean space.

Let $\nabla$ denote the standard Euclidean connection, so that if $X,Y\in C_{0}^{\infty}(\R^{\adimn},\R^{\adimn})$, if $Y=(Y_{1},\ldots,Y_{\adimn})$, and if $u_{1},\ldots,u_{\adimn}$ is the standard basis of $\R^{\adimn}$, then $\nabla_{X}Y\colonequals\sum_{i=1}^{\adimn}(X (Y_{i}))u_{i}$.  Let $N$ be the outward pointing unit normal vector of an $\sdimn$-dimensional orientable hypersurface $\Sigma\subset\R^{\adimn}$.  For any vector $x\in\Sigma$, we write $x=x^{T}+x^{N}$, so that $x^{N}\colonequals\langle x,N\rangle N$ is the normal component of $x$, and $x^{T}$ is the tangential component of $x\in\Sigma$.

Let $e_{1},\ldots,e_{\sdimn}$ be a (local) orthonormal frame of $\Sigma\subset\R^{\adimn}$.  That is, for a fixed $x\in\Sigma$, there exists a neighborhood $U$ of $x$ such that $e_{1},\ldots,e_{\sdimn}$ is an orthonormal basis for the tangent space of $\Sigma$, for every point in $U$ \cite[Proposition 11.17]{lee03}.

Define the \embolden{mean curvature} of $\Sigma$ by
\begin{equation}\label{three0.5}
H\colonequals\mathrm{div}(N)=\sum_{i=1}^{\sdimn}\langle\nabla_{e_{i}}N,e_{i}\rangle.
\end{equation}

Define the \embolden{second fundamental form} $A=(a_{ij})_{1\leq i,j\leq\sdimn}$ of $\Sigma$ so that
\begin{equation}\label{three1}
a_{ij}=\langle\nabla_{e_{i}}e_{j},N\rangle,\qquad\forall\,1\leq i,j\leq \sdimn.
\end{equation}
Compatibility of the Riemannian metric says $a_{ij}=\langle\nabla_{e_{i}}e_{j},N\rangle=-\langle e_{j},\nabla_{e_{i}}N\rangle+ e_{i}\langle N,e_{j}\rangle=-\langle e_{j},\nabla_{e_{i}}N\rangle$, $\forall$ $1\leq i,j\leq \sdimn$.  So, multiplying by $e_{j}$ and summing this equality over $j$ gives
\begin{equation}\label{three2}
\nabla_{e_{i}}N=-\sum_{j=1}^{\sdimn}a_{ij}e_{j},\qquad\forall\,1\leq i\leq \sdimn.
\end{equation}

Using $\langle\nabla_{N}N,N\rangle=0$,
\begin{equation}\label{three4}
H\stackrel{\eqref{three0.5}}{=}\sum_{i=1}^{\sdimn}\langle \nabla_{e_{i}} N,e_{i}\rangle
\stackrel{\eqref{three2}}{=}-\sum_{i=1}^{\sdimn}a_{ii}.
\end{equation}

When $\Sigma\colonequals\partial\Omega$ itself has a boundary that is a $C^{\infty}$ $(\sdimn-1)$-dimensional manifold, we let $\nu$ denote the unit normal of $\partial\Sigma$ pointing exterior to $\Sigma$.

\subsection{Colding-Minicozzi Theory for Mean Curvature Flow}

A key aspect of the Colding \!-Minicozzi theory is the study of eigenfunctions of\ the differential operator $L$, defined for any $C^{\infty}$ function $f\colon\Sigma\to\R$ by
\begin{equation}\label{three4.5}
L f\colonequals \Delta f-\langle x,\nabla f\rangle+f+\vnormt{A}^{2}f.
\end{equation}
\begin{equation}\label{three4.3}
\mathcal{L}f\colonequals \Delta f-\langle x,\nabla f\rangle.
\end{equation}
Note that there is a factor of $2$ difference between our definition of $L$ and the definition of $L$ in \cite{colding12a}.
Here $e_{1},\ldots,e_{\sdimn}$ is a (local) orthonormal frame for an orientable $C^{\infty}$ $\sdimn$-dimensional hypersurface $\Sigma\subset\R^{\adimn}$ with $\redS=\emptyset$, $\Delta\colonequals\sum_{i=1}^{\sdimn}\nabla_{e_{i}}\nabla_{e_{i}}$ be the Laplacian associated to $\Sigma$, $\nabla\colonequals\sum_{i=1}^{\sdimn}e_{i}\nabla_{e_{i}}$ is the gradient associated to $\Sigma$, $A=A_{x}$ is the second fundamental form of $\Sigma$ at $x$, and $\vnormt{A_{x}}^{2}$ is the sum of the squares of the entries of the matrix $A_{x}$.  Let $\mathrm{div}_{\tau}\colonequals\sum_{i=1}^{\sdimn}\nabla_{e_{i}}\langle\cdot,e_{i}\rangle$ be the (tangential) divergence of a vector field on $\Sigma$.  Note that $L$ is an Ornstein-Uhlenbeck-type operator.  In particular, if $\Sigma$ is a hyperplane, then $A_{x}=0$ for all $x\in\Omega$, so $L$ is exactly the usual Ornstein-Uhlenbeck operator, plus the identity map.
%(More detailed definitions will be given in Section \ref{seccurvature} below.)

%\snote{Make sure this is applied correctly below.}
\begin{lemma}[\embolden{Linear Eigenfunction of $L$}, {\cite{mcgonagle15,barchiesi16}} {\cite[Lemma 4.2]{heilman17}}]\label{lemma45}
Let $\Sigma\subset\R^{\adimn}$ be an orientable $C^{\infty}$ $\sdimn$-dimensional hypersurface.  Let $\scon\in\R$, $z\in\R^{\adimn}$.  Suppose
\begin{equation}\label{three0n}
H(x)=\langle x,N\rangle+\scon-\epsilon\langle x,z\rangle,\qquad\forall\,x\in\Sigma.
\end{equation}
Let $v\in\R^{\adimn}$.  Then %\snote{what term to add?}
\begin{equation}\label{three0p}
L\langle v,N\rangle=\langle v,N\rangle-\epsilon\langle v,z\rangle+\epsilon\langle v,N\rangle\langle z,N\rangle.
\end{equation}
%\snote{Need to re-do calculation, but also add extra term}
\end{lemma}
\begin{proof}
Let $1\leq i\leq \sdimn$.  Then
\begin{equation}\label{three5g}
\nabla_{e_{i}}\langle v,N\rangle
=\langle v,\nabla_{e_{i}}N\rangle
\stackrel{\eqref{three2}}{=}-\sum_{j=1}^{\sdimn}a_{ij}\langle v,e_{j}\rangle.
\end{equation}
Fix $x\in\Sigma$.  Choosing the frame such that $\nabla_{e_{k}}^{T}e_{j}=0$ at $x$ for every $1\leq j,k\leq \sdimn$, we then have $\nabla_{e_{k}}e_{j}=a_{kj}N$ at $x$ by \eqref{three1}, so using also Codazzi's equation,
\begin{equation}\label{three6.5}
\begin{aligned}
\nabla_{e_{i}}\nabla_{e_{i}}\langle v,N\rangle
&=-\sum_{j=1}^{\sdimn}\nabla_{e_{i}}a_{ij}\langle v,e_{j}\rangle
-\sum_{j=1}^{\sdimn}a_{ij}\langle v,\nabla_{e_{i}}e_{j}\rangle\\
&=-\sum_{j=1}^{\sdimn}\nabla_{e_{j}}a_{ii}\langle v,e_{j}\rangle
-\sum_{j=1}^{\sdimn}a_{ij}a_{ij}\langle v,N\rangle.
\end{aligned}
\end{equation}
Therefore,
$$\Delta\langle v,N\rangle
=\sum_{i=1}^{\sdimn}\nabla_{e_{i}}\nabla_{e_{i}}\langle v,N\rangle
\stackrel{\eqref{three4}\wedge\eqref{three6.5}}{=}\langle v,\nabla H\rangle-\vnormt{A}^{2}\langle v,N\rangle.$$
So far, we have not used any of our assumptions.  Using now \eqref{three0n}, and that $A$ is symmetric,
$$
\langle v,\nabla H\rangle
\stackrel{\eqref{three2}}{=}-\sum_{i,j=1}^{\sdimn}\langle x,e_{j}\rangle a_{ij}\langle v,e_{i}\rangle-\epsilon\sum_{i=1}^{\sdimn}\langle v,e_{i}\rangle\langle e_{i},z\rangle
\stackrel{\eqref{three5g}}{=}\langle x,\nabla\langle v,N\rangle\rangle-\epsilon\langle v,z^{T}\rangle
$$
Since $z=z^{T}+z^{N}$, $z^{T}=z-\langle z,N\rangle N$.  So, in summary,
$$\Delta\langle v,N\rangle
=\langle x,\nabla\langle v,N\rangle\rangle-\vnormt{A}^{2}\langle v,N\rangle
-\epsilon\langle v,z\rangle
+\epsilon\langle v,N\rangle\langle z,N\rangle.
$$
We conclude by \eqref{three4.5}.
\end{proof}%
\begin{remark}\label{rk20}
Let $f,g\in C^{\infty}(\Sigma)$.  Using \eqref{three4.5}, we get the following product rule for $L$.
\begin{flalign*}
L(fg)
&=f\Delta g+g\Delta f+2\langle \nabla f,\nabla g\rangle-f\langle x,\nabla g\rangle-g\langle x,\nabla f\rangle+\vnormt{A}^{2}fg+fg\\
&=fLg+g\mathcal{L}f+2\langle \nabla f,\nabla g\rangle.
\end{flalign*}
\end{remark}

\section{First and Second Variation}\label{fsvar}

We will apply the calculus of variations to solve Problem \ref{prob1p}. Here we present the rudiments of the calculus of variations.

Some of the results in this section are well known to experts in the calculus of variations, and many of these results were re-proven in \cite{barchiesi16}, or adapted from \cite{hutchings02}.

Let $\Omega\subset\R^{\adimn}$ be an $(\adimn)$-dimensional $C^{2}$ submanifold with reduced boundary $\Sigma\colonequals\redb \Omega$.  Let $N\colon\redA\to S^{\sdimn}$ be the unit exterior normal to $\redA$.  Let $X\colon\R^{\adimn}\to\R^{\adimn}$ be a vector field.

Let $\mathrm{div}$ denote the divergence of a vector field.  We write $X$ in its components as $X=(X_{1},\ldots,X_{\adimn})$, so that $\mathrm{div}X=\sum_{i=1}^{\adimn}\frac{\partial}{\partial x_{i}}X_{i}$.  Let $\Psi\colon\R^{\adimn}\times(-1,1)\to\R^{\adimn}$ such that
\begin{equation}\label{nine2.3}
\Psi(x,0)=x,\qquad\qquad\frac{d}{ds}\Psi(x,s)=X(\Psi(x,s)),\quad\forall\,x\in\R^{\adimn},\,s\in(-1,1).
\end{equation}
For any $s\in(-1,1)$, let $\Omega^{(s)}\colonequals\Psi(\Omega,s)$.  Note that $\Omega^{(0)}=\Omega$.  Let $\Sigma^{(s)}\colonequals\redb\Omega^{(s)}$, $\forall$ $s\in(-1,1)$.
\begin{definition}
We call $\{\Omega^{(s)}\}_{s\in(-1,1)}$ as defined above a \embolden{variation} of $\Omega\subset\R^{\adimn}$.  We also call $\{\Sigma^{(s)}\}_{s\in(-1,1)}$ a \embolden{variation} of $\Sigma=\redb\Omega$.
\end{definition}

Equations \eqref{nine1} and \eqref{nine2} below are proven in e.g. \cite[Lemma 3.2]{heilman18}

\begin{lemma}[\embolden{First Variation}]\label{lemma10}  Let $X\in C_{0}^{\infty}(\R^{\adimn},\R^{\adimn})$.  Let $f(x)=\langle X(x),N(x)\rangle$ for any $x\in\redA$.  Then
\begin{equation}\label{nine1}
\frac{d}{ds}|_{s=0}\gamma_{\adimn}(\Omega^{(s)})=\int_{\redA}f(x) \gamma_{\adimn}(x)dx.
\end{equation}
\begin{equation}\label{nine2}
\frac{d}{ds}|_{s=0}\int_{\Sigma^{(s)}}\gamma_{\sdimn}(x)dx
=\int_{\Sigma}(H(x)-\langle N(x),x\rangle)f(x)\gamma_{\sdimn}(x)dx
+\int_{\redS}\langle X,\nu\rangle\gamma_{\sdimn}(x)dx.
\end{equation}
%\snote{do i need these translations?}
%\snote{Are the $\sqrt{2/\pi}$ factors right here?}
Let $z\colonequals\int_{\Omega}x\gamma_{\adimn}(x)dx$.  Let $w\in\R^{\adimn}$ and define $\overline{w}\colonequals w/\int_{\Omega}\gamma_{\adimn}(x)dx$.  Then
\begin{equation}\label{nine3}
\frac{d}{ds}|_{s=0}\peshs=\int_{\redA}\langle z-w,x-\overline{w}\rangle f(x)\gamma_{\sdimn}(x)dx.
\end{equation}
\end{lemma}
\begin{proof}
We only prove \eqref{nine3}.  Let $J\Psi(x,s)\colonequals\abs{\mathrm{det}(D\Psi(x,s))}$ be the Jacobian determinant of $\Psi$, $\forall$ $x\in\R^{\adimn}$, $\forall$ $s\in(-1,1)$.  Then \cite[Equation (2.28)]{chokski07} says
\begin{equation}\label{nine1.0}
\frac{d}{ds}|_{s=0}J\Psi(x,s)=\mathrm{div}X(x),\qquad\forall\,x\in\redA.
\end{equation}
So, the Chain Rule, $J\Psi(x,0)=1$ (which follows by \eqref{nine2.3}), and \eqref{nine2.3} imply
\begin{equation}\label{nine1.01}
\begin{aligned}
&\frac{1}{2}\frac{d}{ds}|_{s=0}\vnormt{\int_{\Omega^{(s)}}(x-\overline{w})\gamma_{\adimn}(x)dx}^{2}
=\frac{1}{2}\frac{d}{ds}|_{s=0}\vnormt{\int_{\Omega}(\Psi(x,s)-\overline{w})J\Psi(x,s)\gamma_{\adimn}(\Psi(x,s))dx}^{2}\\
&\quad=\int_{\Omega}\Big(\langle z-w,X(x)\rangle+\langle z-w,x-\overline{w}\rangle(\mathrm{div}(X(x))-\langle X(x),x\rangle)\Big)\gamma_{\adimn}(x)dx\\
&\quad=\int_{\Omega}\mathrm{div}(X(x)\langle z-w,x-\overline{w}\rangle\gamma_{\adimn}(x))dx
=\int_{\redA}\langle X(x),N(x)\rangle\langle z-w,x-\overline{w}\rangle\gamma_{\adimn}(x)dx.
\end{aligned}
\end{equation}
In the last line, we used the divergence theorem.  We conclude by writing $\gamma_{\adimn}=(2\pi)^{-1/2}\gamma_{\sdimn}$.
\end{proof}

\begin{definition}\label{defnote}
We let $X\in C_{0}^{\infty}(\R^{\adimn},\R^{\adimn})$ and for any $1\leq i<j\leq m$, we denote $f_{ij}(x)\colonequals\langle X(x),N_{ij}(x)\rangle$ for all $x\in\Sigma_{ij}$.  Recall that $N_{ij}$ is the unit normal vector pointing from $\Omega_{i}$ into $\Omega_{j}$, and $H_{ij}\colonequals\mathrm{div}(N_{ij})$ is the mean curvature of $N_{ij}$.  And $\nu_{ij}$ is the unit normal to $\redb\Sigma_{ij}$ pointing exterior to $\Sigma_{ij}$.
\end{definition}

\begin{figure}[h]\label{notationfig}
\centering
\def\svgwidth{.4\textwidth}
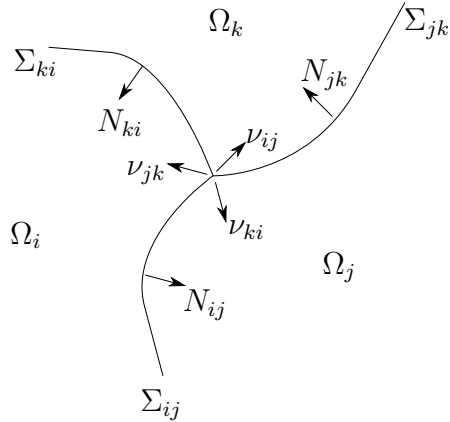
\caption{Notation for sets and normal vectors.}
\end{figure}

The following Lemma is a modification of \cite[Lemma 3.5]{heilman18}.

\begin{lemma}[\embolden{First Variation for Minimizers}]\label{lemma24}
Suppose $\Omega_{1},\ldots,\Omega_{m}$ minimize Problem \ref{prob1p}.  For all $1\leq i\leq m$, let $z^{(i)}\colonequals\int_{\Omega_{i}}x\gamma_{\adimn}(x)dx\in\R^{\adimn}$, let $w\in\R^{\adimn}$ and let $\overline{w}^{(i)}\colonequals w/\int_{\Omega_{i}}\gamma_{\adimn}(x)dx$, $\overline{z}^{(i)}\colonequals z^{(i)}/\int_{\Omega_{i}}\gamma_{\adimn}(x)dx$.  Then $\forall$ $1\leq i<j\leq m$, $\exists$ $\lambda_{ij}\in\R$ such that
\begin{equation}\label{nine5.3}
\left.
\begin{aligned}
\lambda_{ij}&=H_{ij}(x)-\langle x,N_{ij}(x)\rangle+\epsilon\langle x,z^{(i)}-z^{(j)}\rangle+\epsilon\langle w,\overline{w}^{(i)}-\overline{w}^{(j)}+\overline{z}^{(j)}-\overline{z}^{(i)}\rangle,\\
&\qquad\qquad\qquad\qquad\qquad\forall\,x\in\Sigma_{ij},\quad\forall\,1\leq i<j\leq m\\
0&=\lambda_{ij}+\lambda_{jk}+\lambda_{ki},\qquad\forall\,1\leq i<j<k\leq m\,\,\mathrm{such}\,\,\mathrm{that}\,\,\Sigma_{ij}\cap\Sigma_{jk}\cap\Sigma_{ki}\neq\emptyset\\
0&=\nu_{ij}+\nu_{jk}+\nu_{ki},\qquad\forall\,1\leq i<j<k\leq m\,\,\mathrm{such}\,\,\mathrm{that}\,\,\Sigma_{ij}\cap\Sigma_{jk}\cap\Sigma_{ki}\neq\emptyset.
\end{aligned}
\right\}
\end{equation}
\end{lemma}
\begin{proof}
From Lemma \ref{lemma52.6}, Assumption \ref{as1} holds.  From \eqref{nine2} and \eqref{nine3}, if $X\in C_{0}^{\infty}(\R^{\adimn},\R^{\adimn})$
\begin{flalign*}
&\frac{d}{ds}|_{s=0}\Big(\sum_{1\leq i<j\leq m}\int_{\Sigma_{ij}^{(s)}}\gamma_{\sdimn}(x)dx+\epsilon\sqrt{\frac{\pi}{2}}\sum_{i=1}^{m}\vnormtf{\int_{\Omega_{i}^{(s)}}(x-\overline{w}^{(i)})\gamma_{\adimn}(x)dx}^{2}\Big)\\
&\qquad=\sum_{1\leq i<j\leq m}\int_{\Sigma_{ij}}(H_{ij}-\langle N_{ij},x\rangle)f_{ij}\gamma_{\sdimn}(x)dx\\
&\qquad+\sum_{1\leq i<j<k\leq m}\int_{(\redb\Sigma_{ij})\cap (\redb\Sigma_{jk})\cap (\redb\Sigma_{ki})}\langle X,\nu_{ij}+\nu_{jk}+\nu_{ki}\rangle \gamma_{\sdimn}(x)dx\\
&\qquad+\epsilon\sum_{i=1}^{m}\sum_{j\in\{1,\ldots,m\}\colon j\neq i}\int_{\Sigma_{ij}}\langle x-\overline{w}^{(i)},z^{(i)}-w\rangle f_{ij}\gamma_{\sdimn}(x)dx.
\end{flalign*}
We can then choose $X$ to be supported in the neighborhood of two or three points to deduce the above.  The final assertion follows by Lemma \ref{lemma52.6}, i.e. Assumption \ref{as1}(ii).  We also rewrite the last term as (using $N_{ij}=-N_{ji}$ by Definition \ref{defnote}),
\begin{flalign*}
&\sum_{i=1}^{m}\sum_{j\in\{1,\ldots,m\}\colon j\neq i}\int_{\Sigma_{ij}}\langle x-\overline{w}^{(i)},z^{(i)}-w\rangle f_{ij}\gamma_{\sdimn}(x)dx\\
&\qquad=\sum_{1\leq i<j\leq m}\int_{\Sigma_{ij}}f_{ij}\Big(\langle x,z^{(i)}-z^{(j)}\rangle
+\big\langle w, \Big(\frac{1}{\gamma_{\adimn}(\Omega_{i})}-\frac{1}{\gamma_{\adimn}(\Omega_{i})}\Big)w\\
&\qquad\qquad\qquad\qquad\qquad\qquad-\frac{z^{(i)}}{\gamma_{\adimn}(\Omega_{i})}+\frac{z^{(j)}}{\gamma_{\adimn}(\Omega_{j})}\Big\rangle\Big)
\gamma_{\sdimn}(x)dx.
\end{flalign*}
\end{proof}

For any $1\leq i<j\leq m$, let $f\colon C_{0}^{\infty}(\Sigma_{ij})\to\R$ and define

\begin{equation}\label{one3.5n}
L_{ij}f(x)\colonequals \Delta f(x)-\langle x,\nabla f(x)\rangle+\vnormt{A_{x}}^{2}f(x)+f(x),\qquad\forall\,x\in\Sigma_{ij}.
\end{equation}
\begin{equation}\label{one3.5o}
\mathcal{L}_{ij}f(x)\colonequals \Delta f(x)-\langle x,\nabla f(x)\rangle,\qquad\forall\,x\in\Sigma_{ij}.
\end{equation}

The Lemma below can be compared with the corresponding \cite[Proposition 3.3]{hutchings02} and \cite[Proposition 3]{barchiesi16}.

\begin{lemma}[\embolden{Second Variation for Minimizers}]\label{lemma21}
Let $X\in C_{0}^{\infty}(\R^{\adimn},\R^{\adimn})$.  Suppose $\Omega_{1},\ldots,\Omega_{m}$ minimize Problem \ref{prob1}.  Define $z^{(i)}\colonequals\int_{\Omega_{i}}x\gamma_{\adimn}(x)dx$.  Let $(\lambda_{ij})_{1\leq i<j\leq m}$ from \eqref{nine5.3}.  Then
\begin{equation}\label{nine5}
\begin{aligned}
&\frac{d^{2}}{ds^{2}}|_{s=0}\sum_{1\leq i<j\leq m}\int_{\Sigma_{ij}^{(s)}}\gamma_{\sdimn}(x)dx
+\epsilon\pens\\
&=\sum_{1\leq i<j\leq m}\Big(-\int_{\Sigma_{ij}}f_{ij}L_{ij}f_{ij}\gamma_{\sdimn}(x)dx
+\lambda_{ij}\frac{d}{ds}|_{s=0}\int_{\Sigma_{ij}^{(s)}}f_{ij}(x)\gamma_{\sdimn}(x)dx\\
&\quad+\frac{d}{ds}|_{s=0}\int_{\redS_{ij}^{(s)}}\langle X,\nu_{ij}\rangle\gamma_{\sdimn}(x)dx\Big)
+\epsilon\Big(\sum_{i=1}^{m}\sqrt{2\pi}\vnormtf{\sum_{j\in\{1,\ldots,m\}\colon j\neq i}\int_{\Sigma_{ij}}(x-\overline{w}^{(i)})f_{ij}\gamma_{\adimn}(x)dx}^{2}\\
&+\sum_{i=1}^{m}\sum_{j\in\{1,\ldots,m\}\colon j\neq i}\int_{\Sigma_{ij}}f_{ij}\langle X,z^{(i)}-w\rangle\gamma_{\sdimn}(x)dx\\
&+\sum_{i=1}^{m}\sum_{j\in\{1,\ldots,m\}\colon j\neq i}\int_{\Sigma_{ij}}\langle x-\overline{w}^{(i)},z^{(i)}-w\rangle(\mathrm{div}(X)-\langle x,X\rangle)f_{ij}\gamma_{\sdimn}(x)dx
\Big).
\end{aligned}
\end{equation}
\end{lemma}
\begin{proof}
Let $\Sigma$ be an $n$-dimensional $C^{\infty}$ hypersurface with boundary.  We let $'$ denote $\frac{\partial}{\partial s}|_{s=0}$.  From Lemma \ref{lemma10} we have
\begin{flalign*}
&\frac{d^{2}}{ds^{2}}|_{s=0}\int_{\Sigma^{(s)}}\gamma_{\sdimn}(x)dx
=\int_{\Sigma}(H(x)-\langle N(x),x\rangle)'f(x)\gamma_{\sdimn}(x)dx\\
&\qquad+\int_{\Sigma}(H(x)-\langle N(x),x\rangle)[f(x)\gamma_{\sdimn}(x)dx]'
+\frac{d}{ds}|_{s=0}\int_{\redS^{(s)}}\langle X,\nu\rangle\gamma_{\sdimn}(x)dx.
\end{flalign*}
From \eqref{nine2.3}, $x'=X=X^{N}+X^{T}=fN+X^{T}$.  Also, $H'=-\Delta f-\vnormt{A}^{2}f$, $N'=-\nabla f$, \cite[A.3, A.4]{colding12a} (the latter calculations require writing $\Sigma^{(s)}$ in the form $\{x+ sN(x)+O_{x}(s^{2})\colon x\in\Sigma\}$).  So,
$$(H-\langle N,x\rangle)'=-\Delta f-\vnormt{A}^{2}f-\langle N,fN+X^{T}\rangle-\langle x,\nabla f\rangle\stackrel{\eqref{three4.5}}{=}-Lf.$$
In summary,
\begin{flalign*}
&\frac{d^{2}}{ds^{2}}|_{s=0}\int_{\Sigma^{(s)}}\gamma_{\sdimn}(x)dx
=-\int_{\Sigma}fLf\gamma_{\sdimn}(x)dx\\
&\qquad\qquad\qquad+\int_{\Sigma}(H(x)-\langle N(x),x\rangle)[f_{ij}(x)\gamma_{\sdimn}(x)dx]'
+\frac{d}{ds}|_{s=0}\int_{\redS^{(s)}}\langle X,\nu\rangle\gamma_{\sdimn}(x)dx.
\end{flalign*}
Summing over all $1\leq i<j\leq m$, and applying \eqref{nine5.3},
\begin{flalign*}
&\frac{d^{2}}{ds^{2}}|_{s=0}\sum_{1\leq i<j\leq m}\int_{\Sigma_{ij}^{(s)}}\gamma_{\sdimn}(x)dx
=\sum_{1\leq i<j\leq m}-\int_{\Sigma_{ij}}f_{ij}L_{ij}f_{ij}\gamma_{\sdimn}(x)dx\\
&\qquad\qquad\qquad\qquad+\lambda_{ij}\int_{\Sigma_{ij}}[f_{ij}(x)\gamma_{\sdimn}(x)dx]'
+\frac{d}{ds}|_{s=0}\int_{\redS_{ij}^{(s)}}\langle X,\nu_{ij}\rangle\gamma_{\sdimn}(x)dx.
\end{flalign*}
Finally, adding the terms from Lemma \ref{lemma5} completes the proof.
\end{proof}

Below, we need the following combinatorial Lemma, the case $m=3$ being treated in \cite[Proposition 3.3]{hutchings02}.
\begin{lemma}\label{lemma25.3}
Let $m\geq3$.  Let
$$D_{1}\colonequals \{(x_{ij})_{1\leq i\neq j\leq m}\in\R^{\binom{m}{2}}\colon \forall\,1\leq i\neq j\leq m,\quad x_{ij}=-x_{ji},\,\sum_{j\in\{1,\ldots,m\}\colon j\neq i}x_{ij}=0\}.$$
\begin{flalign*}
D_{2}\colonequals \{(x_{ij})_{1\leq i\neq j\leq m}\in\R^{\binom{m}{2}}
&\colon  \forall\,1\leq i\neq j\leq m,\quad x_{ij}=-x_{ji},\\
&\forall\,1\leq i<j<k\leq m\quad x_{ij}+x_{jk}+x_{ki}=0\}.
\end{flalign*}
Let $x\in D_{1}$ and let $y\in D_{2}$.  Then $\sum_{1\leq i<j\leq m}x_{ij}y_{ij}=0$.
\end{lemma}

It is well-known that compactly supported variations such that $\frac{d}{ds}|_{s=0}\gamma_{\adimn}(\Omega_{i}^{(s)})=0$ for all $1\leq i\leq m$ can be modified such that $\frac{d^{k}}{ds^{k}}|_{s=0}\gamma_{\adimn}(\Omega_{i}^{(s)})=0$ for all $1\leq i\leq m$ and for all $k\geq1$ while preserving the second variation.  Such an application of the implicit function theorem appears e.g. in \cite[Proposition 3.3]{hutchings02} or \cite[Lemma 2.4]{barbosa84}.  This argument can be extended to noncompact variations \cite[Lemma 1]{barchiesi16}.

The Lemma below can be compared with the corresponding result \cite[Lemma 3.2]{hutchings02} and \cite[Lemmas 4.12 and 5.2]{milman18a}.

\begin{lemma}[\embolden{Extension Lemma for Existence of Volume-Preserving Variations}]\label{lemma27}
For any $1\leq i<j\leq m$, let $f_{ij}\in C_{0}^{\infty}(\Sigma_{ij})$ satisfy
\begin{equation}\label{eight1}
\forall\,1\leq i<j<k\leq m,\,\forall\,x\in\Sigma_{ij}\cap \Sigma_{jk}\cap\Sigma_{ki},\quad f_{ij}(x)+f_{jk}(x)+f_{ki}(x)=0.
\end{equation}
Then there exists a vector field $X\in C_{0}^{\infty}(\R^{\adimn},\R^{\adimn})$ such that
\begin{equation}\label{eight3}
\forall\,1\leq i<j\leq m,\quad\forall\, x\in\Sigma_{ij},\,\langle X(x),N_{ij}(x)\rangle=f_{ij}(x).
\end{equation}
If additionally
\begin{equation}\label{eight2}
\forall\,1\leq i\leq m,\quad \sum_{j\in\{1,\ldots,m\}\colon j\neq i}\int_{\Sigma_{ij}}f_{ij}\gamma_{\sdimn}(x)dx=0,
\end{equation}
then $X$ can also be chosen to be volume preserving:
$$\forall\,1\leq i\leq m,\quad\forall\,s\in(-\epsilon,\epsilon),\quad \gamma_{\adimn}(\Omega_{i}^{(s)})=\gamma_{\adimn}(\Omega_{i}).$$
\end{lemma}
\begin{proof}
By assumption, $\exists$ a vector field $Z\colon\R^{\adimn}\to\R^{\adimn}$ such that
$$\langle Z(x),N_{ij}(x)\rangle=f_{ij}(x),\qquad\forall\, x\in C,\quad\forall\,1\leq i<j\leq m.$$
Then $Z$ can be extended to all of $\cup_{1\leq i<j\leq m}\Sigma_{ij}$ by e.g. Whitney Extension.  Let $I$ be a subset of $\{(i,j)\colon 1\leq i<j\leq m\}$ of size $m-1$.  For all $(i,j)\in I$, let $g_{ij}\colon\Sigma_{ij}\to\R$ be compactly supported, nonnegative, $C^{\infty}$ functions and let $\widetilde{g}_{ij}$ be any smooth extension of $g_{ij}$ to $\R^{\adimn}$ that is supported in a neighborhood of the interior $\Sigma_{ij}$, disjoint from all $\Sigma_{i'j'}$ with $(i',j')\neq(i,j)$.  Similarly, let $\widetilde{N}_{ij}$ be any smooth extension of $N_{ij}$ to $\R^{\adimn}$.  Consider the map $\widetilde{\Psi}\colon\R^{\adimn}\times(-1,1)\times(-1,1)^{m-1}\to\R^{\adimn}$ defined by
$$\widetilde{\Psi}(x,s,\{t_{ij}\}_{(i,j)\in I})\colonequals x+sZ+\sum_{(i,j)\in I}t_{ij}\widetilde{g}_{ij}\widetilde{N}_{ij},
\qquad\forall\,s\in(-1,1),\,\{t_{ij}\}_{(i,j)\in I}\in(-1,1)^{m-1}.$$
And consider the vector-valued function
$$V(s,\{t_{ij}\}_{(i,j)\in I})\colonequals\Big(\gamma_{\adimn}(\Omega_{1}^{(s,\{t_{ij}\}_{(i,j)\in I})}),\ldots,\gamma_{\adimn}(\Omega_{m}^{(s,\{t_{ij}\}_{(i,j)\in I})})\Big).$$
Then $V\colon\R^{m}\to\R^{m}$, and the image of $V$ is at most $(m-1)$-dimensional, since the sum of the entries of $V$ is equal to $1$.  Consider the equation $V=\mathrm{constant}$.  Then the Jacobian of $V$ has maximal rank.  So, by the Implicit Function Theorem, for every $(i,j)\in I$, there exists a function $t_{ij}\colon(-1,1)\to\R$ such that $V(s,\{t_{ij}(s)\}_{(i,j)\in I})=\mathrm{constant}$ for all $s\in(-1,1)$.  Since the Jacobian of $V$ has maximal rank and \eqref{eight2} holds, it follows from the chain rule that $t_{ij}'(0)=0$ for all $(i,j)\in I$.  So, if we let $X$ be the vector field for $\Psi(x,s)\colonequals\widetilde{\Psi}(x,s,\{t_{ij}(s)\}_{(i,j)\in I})$ satisfying \eqref{nine2.3}.  Then \eqref{eight3} holds for $X$.
\end{proof}

\begin{lemma}[\embolden{Volume-Preserving Second Variation for Minimizers}]\label{lemma28}
Let $\Omega_{1},\ldots,\Omega_{m}$ minimize Problem \ref{prob1}.  $\forall$ $1\leq i<j\leq m$, let $f_{ij}\in C_{0}^{\infty}(\Sigma_{ij})$ satisfy \eqref{eight1} and \eqref{eight2}.  Let $X$ be the vector field guaranteed to exist from Lemma \ref{lemma27}.  For all $1\leq i\leq m$, let $z^{(i)}\colonequals\int_{\Omega_{i}}x\gamma_{\adimn}(x)dx$.  Then
\begin{flalign*}
&\frac{d^{2}}{ds^{2}}|_{s=0}\Big(\sum_{1\leq i<j\leq m}\int_{\Sigma_{ij}^{(s)}}\gamma_{\sdimn}(x)dx+\epsilon\pens\Big)\\
&\qquad=\sum_{1\leq i<j\leq m}-\int_{\Sigma_{ij}}f_{ij}[L_{ij}f_{ij}-\epsilon\langle X,z^{(i)}-z^{(j)}\rangle]\gamma_{\sdimn}(x)dx
+\frac{d}{ds}|_{s=0}\int_{\redS_{ij}^{(s)}}\langle X,\nu_{ij}\rangle\gamma_{\sdimn}(x)dx\\
&\qquad\qquad\qquad\qquad\qquad\qquad+\epsilon\sqrt{2\pi}\pender.
\end{flalign*}
\end{lemma}
\begin{proof}
Assumption \ref{as1} holds by Lemma \ref{lemma52.6}.  From Lemma \ref{lemma24}, \eqref{nine5.3} holds.  Since the volumes are preserved, for any $1\leq i\leq m$, we have $\sum_{j\neq i}\frac{d}{ds}|_{s=0}\int_{\Sigma_{ij}^{(s)}}f_{ij}(x)\gamma_{\sdimn}(x)dx=0$.  Combining Lemmas \ref{lemma21} and \ref{lemma25.3} shows that the middle term from Lemma \ref{lemma21} vanishes.
\end{proof}

\begin{remark}\label{rk5}
$\forall$ $1\leq i<j<k\leq m$, and $\forall$ $x\in(\redS_{ij})\cap (\redS_{jk})\cap (\redS_{ki})$, define
$$q_{ij}(x)\colonequals[\langle \nabla_{\nu_{kj}}\nu_{kj},N_{kj}\rangle+\langle \nabla_{\nu_{ki}}\nu_{ki},N_{ki}\rangle]/\sqrt{3}.$$
Note that $q_{ij}+q_{jk}+q_{ki}=0$ since $N_{ij}=-N_{ji}$ by Definition \ref{defnote} and $q_{ij}=q_{ji}$.

Compared to \cite{hutchings02}, note that we have the opposite sign convention for the second fundamental form and for $\nu_{ij}$.
\end{remark}
\begin{lemma}[{\cite[Lemma 3.6]{hutchings02}}]\label{lemma29.9}
 $\forall$ $1\leq i<j\leq m$, let $f_{ij}\in C_{0}^{\infty}(\Sigma_{ij})$ satisfy \eqref{eight1}.  Then
\begin{flalign*}
&\frac{d}{ds}|_{s=0}\sum_{1\leq i<j\leq m}\int_{\Sigma_{ij}^{(s)}}\langle X,\nu_{ij}\rangle\gamma_{\sdimn}(x)dx\\
&=\sum_{1\leq i<j<k\leq m}\int_{\redS_{ij}\cap\redS_{jk}\cap \redS_{ki}}\langle X,\nabla_{X}(\nu_{ij}+\nu_{jk}+\nu_{ki})\rangle\gamma_{\sdimn}(x)dx\\
 &=\sum_{1\leq i<j<k\leq m}\int_{\redS_{ij}\cap\redS_{jk}\cap \redS_{ki}}
\Big([\nabla_{\nu_{ij}}f_{ij}+q_{ij}f_{ij}]f_{ij}+[\nabla_{\nu_{jk}}f_{jk}+q_{jk}f_{jk}]f_{jk}\\
 &\qquad\qquad\qquad\qquad\qquad\qquad+[\nabla_{\nu_{ki}}f_{ki}+q_{ki}f_{ki}]f_{ki}\Big)\gamma_{\sdimn}(x)dx.
 \end{flalign*}
\end{lemma}

Applying the above Lemma, we get

\begin{lemma}\label{lemma36.5}
If for all $1\leq i<j<k\leq m$ we have $f_{ij}\in C_{0}^{\infty}(\Sigma_{ij})$ satisfying \eqref{eight1} and
$$\nabla_{\nu_{ij}}f_{ij}+q_{ij}f_{ij}=\nabla_{\nu_{jk}}f_{jk}+q_{jk}f_{jk}=\nabla_{\nu_{ki}}f_{ki}+q_{ki}f_{ki},\qquad\forall\, 1\leq i<j<k\leq m,$$
then
$$\frac{d}{ds}|_{s=0}\sum_{1\leq i<j\leq m}\int_{\Sigma_{ij}^{(s)}}\langle X,\nu_{ij}\rangle\gamma_{\sdimn}(x)dx=0.$$
\end{lemma}

\section{Second Variation as a Quadratic Form}\label{seceigs}

\begin{definition}[\embolden{Admissible Functions}]\label{def1}
Define $\mathcal{F}$ be the set of functions $(f_{ij})_{1\leq i<j\leq m}$ such that
\begin{itemize}
\item  $\forall$ $1\leq i<j\leq m$, $f_{ij}\colon\Sigma_{ij}\to\R$, $\int_{\Sigma_{ij}}f_{ij}^{2}\gamma_{\sdimn}(x)dx<\infty$ and $\int_{\Sigma_{ij}}\vnormt{\nabla f_{ij}}^{2}\gamma_{\sdimn}(x)dx<\infty$.
\item$\forall\,1\leq i<j<k\leq m,\,\forall\,x\in\Sigma_{ij}\cap \Sigma_{jk}\cap\Sigma_{ki},\, f_{ij}(x)+f_{jk}(x)+f_{ki}(x)=0$.
\end{itemize}
The second condition is well-defined by e.g. a (local) Sobolev Trace inequality \cite{feo13}.
\end{definition}

\begin{definition}[\embolden{Quadratic Form Associated to Second Variations}]\label{def2}
For any $F=(f_{ij})_{1\leq i<j\leq m},G=(g_{ij})_{1\leq i<j\leq m}\in\mathcal{F}$, define the following quantities if they exist:
\begin{equation}\label{quaddef}
\begin{aligned}
&Q(F,G)\colonequals
\sum_{1\leq i<j\leq m}-\int_{\Sigma_{ij}}g_{ij}L_{ij}f_{ij}\gamma_{\sdimn}(x)dx+\sum_{1\leq i<j<k\leq m}\int_{\redS_{ij}\cap\redS_{jk}\cap \redS_{ki}}\\
&\qquad\Big([\nabla_{\nu_{ij}}f_{ij}+q_{ij}f_{ij}]g_{ij}+[\nabla_{\nu_{jk}}f_{jk}+q_{jk}f_{jk}]g_{jk}
+[\nabla_{\nu_{ki}}f_{ki}+q_{ki}f_{ki}]g_{ki}\Big)\gamma_{\sdimn}(x)%\\
%&\qquad+\penq.
 \end{aligned}
 \end{equation}

 \begin{equation}\label{ipdef}
 \langle F,G\rangle\colonequals\sum_{1\leq i<j\leq m}\int_{\Sigma_{ij}}f_{ij}g_{ij}\gamma_{\sdimn}(x)dx.
 \end{equation}
 \end{definition}

Using \eqref{one3.5n}, define $\L\colon \cup_{1\leq i<j\leq m}C_{0}^{\infty}(\Sigma_{ij})\to \cup_{1\leq i<j\leq m}C_{0}^{\infty}(\Sigma_{ij})$ by
\begin{equation}\label{seven00}
\L((f_{ij})_{1\leq i<j\leq m})\colonequals (L_{ij}f_{ij})_{1\leq i<j\leq m}.
\end{equation}
Using also \eqref{one3.5o} define
\begin{equation}\label{seven00p}
\mL((f_{ij})_{1\leq i<j\leq m})\colonequals (\mathcal{L}_{ij}f_{ij})_{1\leq i<j\leq m}.
\end{equation}

\begin{lemma}[\embolden{Integration by Parts}]\label{lemma32.5}
Let $F,G\in\mathcal{F}\cap C_{0}^{\infty}(\cup_{1\leq i<j\leq m}\Sigma_{ij})$.  Then
\begin{flalign*}
&Q(F,G)\colonequals
\sum_{1\leq i<j\leq m}\int_{\Sigma_{ij}}[\langle\nabla f_{ij},\nabla g_{ij}\rangle -f_{ij}g_{ij}(\vnormt{A}^{2}+1)]\gamma_{\sdimn}(x)dx\\
&\qquad+\sum_{1\leq i<j<k\leq m}\int_{\redS_{ij}\cap\redS_{jk}\cap \redS_{ki}}
 [q_{ij}f_{ij}g_{ij}+q_{jk}f_{jk}g_{jk}+q_{ki}f_{ki}g_{ki}]\gamma_{\sdimn}(x)dx.
 \end{flalign*}
 In particular, $Q(F,G)=Q(G,F)$, so that $Q$ is symmetric.
\end{lemma}
\begin{proof}
From the divergence theorem for an $\sdimn$-dimensional $C^{\infty}$ orientable hypersurface $\Sigma$ with $C^{\infty}$ boundary, if $f,g\colon\Sigma\to\R$, then
\begin{flalign*}
&\int_{\Sigma}(\mathcal{L}f)g\gamma_{\sdimn}(x)dx
\stackrel{\eqref{three4.3}}{=}\int_{\Sigma}(\Delta f-\langle x,\nabla f\rangle)g\gamma_{\sdimn}(x)dx
=\int_{\Sigma}\mathrm{div}_{\tau}(\gamma_{\sdimn}(x)\nabla f)gdx\\
&=\int_{\Sigma}\Big([\mathrm{div}_{\tau}(g\gamma_{\sdimn}(x)\nabla f)]-\langle\nabla f,\nabla g\rangle\Big)\gamma_{\sdimn}(x)dx
=\int_{\partial \Sigma}\langle\nabla f,\nu\rangle g\gamma_{\sdimn}(x)-\int_{\Sigma}\langle\nabla f,\nabla g\rangle\gamma_{\sdimn}(x)dx.
\end{flalign*}
As usual, $\nu$ denotes the exterior pointing unit normal to $\partial\Sigma$.  Substituting into the definition of $Q(F,G)$ and using \eqref{three4.3} and \eqref{three4.5} completes the proof.
\end{proof}

\begin{lemma}[{\cite[Lemma 1]{barchiesi16}}]\label{lemma60}
Let $\Omega_{1},\ldots,\Omega_{m}$ satisfy Assumption \ref{as1}.  Then there exists a sequence of $C^{\infty}$ functions $\eta_{1}\leq \eta_{2}\leq\cdots\colon\cup_{i=1}^{m}\redb\Omega_{i}\to[0,1]$ supported in $M_{\sdimn}\cup M_{\sdimn-1}\cup M_{\sdimn-2}$ (using the notation of Assumption \ref{as1}) such that
$$\forall\,x\in \cup_{i,j=1}^{m}\redb\Omega_{ij},\quad\lim_{u\to\infty}\eta_{u}(x)=1,$$
$$\lim_{u\to\infty}\sum_{1\leq i<j\leq m}\int_{\Sigma_{ij}}[(1-\eta_{u})^{2}+\vnorm{\nabla(1-\eta_{u})}^{2}]\gamma_{\sdimn}(x)dx=0.$$
\end{lemma}
\begin{proof}
By Assumption \ref{as1}, $\cup_{i=1}^{m}\partial\Omega_{i}\setminus(M_{\sdimn}\cup M_{\sdimn-1}\cup M_{\sdimn-2})$ has zero $(\sdimn-2)$-dimensional Hausdorff measure, so the assertion follows e.g. by \cite[Lemma 1]{barchiesi16}.
\end{proof}

\begin{lemma}[\embolden{Non-Compact Variations}]\label{lemma97}
Let $\Omega_{1},\ldots,\Omega_{m}$ satisfy Assumption \ref{as1}.  Let $F,G\in\mathcal{F}$.  Assume that
$$\sum_{1\leq i<j\leq m}\int_{\Sigma_{ij}}\abs{L_{ij}f_{ij}}^{2}\gamma_{\sdimn}(x)dx<\infty.$$
Assume $\forall\,1\leq i<j<k\leq m$, $\forall$ $x\in(\redS_{ij})\cap(\redS_{jk})\cap(\redS_{ki})$, the following holds at $x$.
\begin{equation}\label{four75}
\nabla_{\nu_{ij}}f_{ij}+q_{ij}f_{ij}=\nabla_{\nu_{jk}}f_{jk}+q_{jk}f_{jk}=\nabla_{\nu_{ki}}f_{ki}+q_{ki}f_{ki}.
\end{equation}
Then $Q(F,F)$ and $Q(F,G)$ are well-defined real numbers.  Moreover,
$$Q(F,F)=-\langle LF,F\rangle,\qquad Q(F,G)=-\langle LF,G\rangle.$$
Also, $\exists$ a sequence $\phi_{1},\phi_{2},\ldots\in C_{0}^{\infty}(\Sigma)$ with  $0\leq\phi_{1}\leq\phi_{2}\leq\cdots \leq1$ on $\R^{\adimn}$ converging pointwise to $1$ such that
$$\lim_{u\to\infty}Q(\phi_{u}F,G)=\lim_{u\to\infty}Q(\phi_{u}F,\phi_{u}G)=Q(F,G).$$
\end{lemma}
\begin{proof}
Let  $\Sigma\colonequals\cup_{1\leq i<j\leq m}\Sigma_{ij}$.  Let $\phi\in C_{0}^{\infty}(\Sigma)$ with $0\leq\phi\leq 1$, $\phi=1$ when $\vnorm{x}\leq r$, $\phi=0$ when $\vnorm{x}>r+2$ and $\vnorm{\nabla\phi}\leq1$ on $\Sigma$.  From Lemma \ref{lemma32.5},
$$Q(\phi F,G)-Q(F,\phi G)=\sum_{1\leq i<j\leq m}\int_{\Sigma_{ij}}\Big(f_{ij}\langle\nabla \phi,\nabla g_{ij}\rangle+g_{ij}\langle\nabla \phi,\nabla g_{ij}\rangle\Big)\gamma_{\sdimn}(x)dx.$$
So, as $r\to\infty$, $\abs{Q(\phi F,G)-Q(F,\phi G)}$ converges to $0$ by the Dominated Convergence Theorem and the Cauchy-Schwarz inequality, using $F,G\in\mathcal{F}$ and Definition \ref{def1}.
By the assumption \eqref{four75} on $F$, $Q(F,\phi G)\stackrel{\eqref{quaddef}}{=}\sum_{1\leq i<j\leq m}-\int_{\Sigma_{ij}}\phi g_{ij}L_{ij} f_{ij}\gamma_{\sdimn}(x)dx$.  So, as $r\to\infty$, $Q(F,\phi G)$ converges to $\langle-\L F,G\rangle$.  Therefore, as $r\to\infty$, $Q(\phi F,G)$ also converges to $\langle-\L F,G\rangle$.  The second assertion follows from the first, since $\abs{Q(\phi F,\phi G)-Q(F,\phi^{2}G)}$ converges to zero as $r\to\infty$ as well.
\end{proof}

\section{Curvature Bounds}

Below we denote $\Sigma\colonequals\cup_{1\leq i<j\leq m}\Sigma_{ij}$.

\begin{remark}\label{rk30}
Let $v\in\R^{\adimn}$.  For all $1\leq i<j\leq m$ let $f_{ij}\colon \Sigma_{ij}\to\R$ be defined by $f_{ij}\colonequals\langle v,N_{ij}\rangle$.  Then for all $1\leq i<j\leq m$, $L_{ij}f_{ij}=f_{ij}$ by Lemma \ref{lemma45} and Lemma \ref{lemma24}.  Also, the term in Lemma \ref{lemma29.9} is zero, since $X\colonequals v$ is the constant vector field in this case, i.e. \eqref{four75} holds.
\end{remark}

\begin{lemma}\label{lemma55}
Let $\Lambda$ be the set of solutions of the middle equation of \eqref{nine5.3}.  Then $\Lambda$ is a vector space of dimension equal to $m-1$.  Also, $\Lambda$ has an orthonormal basis (with respect to $\langle\cdot,\cdot\rangle$ defined in Lemma \ref{lemma25.3}) consisting of vectors all of whose components are nonzero.
\end{lemma}
\begin{proof}
From Lemma \ref{lemma25.3}, $\Lambda$ has dimension equal to $m-1$.  Consider the sets described in Conjecture \ref{conj0}.  These sets satisfy $H_{ij}(x)=0$ for every $x\in\Sigma_{ij}$, and they also satisfy all equations from \eqref{nine5.3}, for any $y\in\R^{m-1}$.  We can then treat $N_{ij}$ as being constant functions of $y$, so that $\lambda_{ij}(y)=-\langle y,N_{ij}\rangle$ for all $1\leq i<j\leq m$ is a solution of the equations \eqref{nine5.3}.  By considering any $y\in\R^{m-1}$, linear algebra also implies then that $\Lambda$ has dimension at least $m-1$, since the only $y\in\R^{m-1}$ such that $\langle y,N_{ij}\rangle=0$ for all $1\leq i<j\leq m$ is $y=0$.  Finally, choosing an orthonormal basis of $y$'s of $\R^{m-1}$ so that each basis element is not perpendicular to $N_{ij}$ for all $1\leq i<j\leq m$, then we have $m-1$ nonvanishing solutions of  \eqref{nine5.3}.
\end{proof}

For any hypersurface $\Sigma\subset\R^{\adimn}$ (possibly with boundary), we define
\begin{equation}\label{seven0}
\pcon=\pcon(\Sigma)
\colonequals-\inf_{G\in \mathcal{F}\cap C_{0}^{\infty}(\Sigma)\colon\langle G,G\rangle=1}Q(G,G).
\end{equation}
By the definition of $\pcon$,
\begin{equation}\label{seven1}
\Sigma_{1}\subset\Sigma_{2}\qquad\Longrightarrow\qquad\pcon(\Sigma_{1})\leq\pcon(\Sigma_{2}).
\end{equation}

\begin{lemma}[\embolden{Existence of Fundamental Tone}]\label{lemma33}
Assume $\pcon\colonequals\pcon(\Sigma)<\infty$.  Then there exists $F\in\mathcal{F}$ such that
\begin{equation}\label{eqstar}
Q(F,F)=\min_{G\in\mathcal{F}\colon\langle G,G\rangle=1}
Q(G,G).
\end{equation}
If $F\in\mathcal{F}$ satisfies \eqref{eqstar}, then the following hold.  $F$ is an eigenfunction of $\L$ so that
$$\L F=\pcon F.$$
Moreover, $\forall\,1\leq i<j<k\leq m$, $\forall$ $x\in(\redS_{ij})\cap(\redS_{jk})\cap(\redS_{ki})$, the following holds at $x$.
$$\nabla_{\nu_{ij}}f_{ij}+q_{ij}f_{ij}=\nabla_{\nu_{jk}}f_{jk}+q_{jk}f_{jk}=\nabla_{\nu_{ij}}f_{ij}+q_{ij}f_{ij}.$$
\end{lemma}
\begin{proof}
First note that the set of functions $G$ specified in \eqref{eqstar} is nonempty by Lemma \ref{lemma55}.

Fix $x$ in the interior of $\Sigma$.  Let $\Sigma_{1}\subset\Sigma_{2}\subset\cdots$ be a sequence of compact $C^{\infty}$ hypersurfaces (with boundary) such that $\cup_{k=1}^{\infty}\Sigma_{k}=\Sigma$.  For each $k\geq1$, let $F_{k}$ be a Dirichlet eigenfunction of $\L$ on $\Sigma_{k}$ such that $\L F_{k}=\pcon(\Sigma_{k})F_{k}$, and such that $F_{k}$ does not change sign on any connected component of $\Sigma$.  By multiplying by a constant, we may assume $F_{k}(x)=1$ for all $k\geq1$.  Since $\pcon(\Sigma_{k})$ increases to $\pcon(\Sigma)<\infty$ as $k\to\infty$ by \eqref{seven0}, the Harnack inequality implies that there exists $c=c(\Sigma_{k},\pcon(\Sigma))$ such that $1\leq\sup_{y\in B}F_{k}(y)\leq c\inf_{y\in B}F_{k}(y)\leq c$ for some neighborhood $B$ of $x$.  Elliptic theory then gives uniform $C^{2,\sigma}$ bounds for the functions $F_{1},F_{2},\ldots$ on each compact subset of $\Sigma$.  So, by Arzel\`{a}-Ascoli there exists a uniformly convergent subsequence of $F_{1},F_{2},\ldots$ which converges to a solution $LF=\pcon(\Sigma)F$ on $\Sigma$ with $F(x)=1$ such that $F$ does not change sign on any connected component of $\Sigma$.  The Harnack inequality then implies $F$ is nonzero on any connected component of $\Sigma$.

Let $G\in\mathcal{F}$.  For any $t\in\R$, define
$$c(t)\colonequals\frac{Q(F+tG,F+tG)}{\langle F+tG,F+tG\rangle}-Q(F,F).$$
By definition of $F$, we have $c(t)\geq0$ $\forall$ $t\in\R$.  Therefore, $c'(0)=0$.  By Lemma \ref{lemma32.5}, $Q(F,G)=Q(G,F)$, so
$$c(t)=\frac{Q(F,F)+2tQ(F,G)+t^{2}Q(G,G)}{\langle F,F\rangle+2t\langle F,G\rangle+t^{2}\langle G,G\rangle}-Q(F,F),$$
$$0=c'(0)=\frac{\langle F,F\rangle Q(F,Q)-Q(F,F)\langle F,G\rangle}{\langle F,F\rangle^{2}}.$$
Therefore, for any $G\in\mathcal{F}$,
$$Q(F,G)=\frac{Q(F,F)}{\langle F,F\rangle}\langle F,G\rangle.$$
Fix $1\leq i<j\leq m$.  Choosing $G$ smooth and localized away from $C$ then implies that $\L F=-\frac{Q(F,F)}{\langle F,F\rangle}F\equalscolon-\pcon F$ on $\Sigma_{ij}$ for every $1\leq i<j\leq m$, away from their boundaries.  Fix $1\leq i<j<k\leq m$.  Choose the vector field $X$ (where $g_{pq}\colonequals\langle X,N_{pq}\rangle$ for all $1\leq p<q\leq m$) now such that $g_{ij}=-g_{jk}=1$ and $g_{ki}=0$ at some $x\in (\redS_{ij})\cap (\redS_{jk})\cap (\redS_{ki})$, and such that $X$ is supported in a neighborhood of $x$.  Then the definition of $Q(F,G)$ implies that $\nabla_{\nu_{ij}}f_{ij}+q_{ij}f_{ij}=\nabla_{\nu_{jk}}f_{jk}+q_{jk}f_{jk}$.  (This argument is valid as long as the sign of $f_{ij},f_{jk},f_{ki}$ are not all the same at $x$.  It cannot occur that all three of these numbers have the same sign, since they must sum to zero at $x$, and by the definition of $F$, these three functions cannot all have the same sign in a neighborhood of $x$, by the limiting definition of $F$.)

\end{proof}

A rearrangement argument implies the following decay for the Gaussian surface area of optimal sets far from the origin.

\begin{lemma}[{\cite[Lemma 4.3]{milman18a}}]\label{lemma28.8}
Let $\Omega_{1},\ldots,\Omega_{m}$ minimize Problem \ref{prob1}.  Then for all $r>\sqrt{\adimn}+\vnormt{w}$
$$\sum_{1\leq i<j\leq m}\gamma_{\sdimn}(\Sigma_{ij}\cap\{x\in\R^{\adimn}\colon\vnorm{x-w}>r\})
\leq 3m\gamma_{\sdimn}(\{x\in\R^{\adimn}\colon\vnorm{x}=r\}).$$
Also, for all $u\in\R^{\adimn}$ with $\vnormt{u}=1$, and for all $r>1+\vnormt{w}$,
$$\sum_{1\leq i<j\leq m}\gamma_{\sdimn}(\Sigma_{ij}\cap\{x\in\R^{\adimn}\colon\abs{\langle x-w,u\rangle}>r\})
\leq 2me^{-(r+\vnormt{w})^{2}/2}.$$
\end{lemma}
\begin{proof}
We prove both statements simultaneously.  Let $H(w,r)$ denote either the ball $\{x\in\R^{\adimn}\colon \vnorm{x-w}\leq r\}$ or the slab $\{x\in\R^{\adimn}\colon \abs{\langle x-w,u\rangle}\leq r\}$. Let $\alpha,r>0$.  Without loss of generality, assume that
$$\gamma_{\adimn}(\Omega_{1}\cap H(w,r)^{c})\leq \gamma_{\adimn}(\Omega_{2}\cap H(w,r)^{c})\leq\cdots\leq \gamma_{\adimn}(\Omega_{m}\cap H(w,r)^{c}).$$
Let $r=b_{0}\leq b_{1}\leq\cdots\leq b_{m}$ such that for all $1\leq i\leq m$,
$\gamma_{\adimn}(\Omega_{i}\cap H(w,r)^{c})=\gamma_{\adimn}(H(w,b_{i})\setminus H(w,b_{i-1}))$.  For any $1\leq i\leq m$, define
$$\widetilde{\Omega}_{i}\colonequals (\Omega_{i}\setminus H(w,r))\cup (H(w,b_{i})\setminus H(w,b_{i-1}))
,\qquad\widetilde{\Sigma}_{ij}\colonequals(\redb\widetilde{\Omega}_{i})\cap(\redb\widetilde{\Omega}_{j}).$$
Then $\gamma_{\adimn}(\Omega_{i})=\gamma_{\adimn}(\widetilde{\Omega}_{i})$ for all $1\leq i\leq m$.  Also, since $\Omega_{1},\ldots,\Omega_{m}$ minimize Problem \ref{prob1},
\begin{flalign*}
&\sum_{1\leq i<j\leq m}\gamma_{\sdimn}(\Sigma_{ij})
+\epsilon\pen\\
&\qquad\qquad\qquad\leq \sum_{1\leq i<j\leq m}\gamma_{\sdimn}(\widetilde{\Sigma}_{ij})
+\epsilon\pentil.
\end{flalign*}
Using $\pentil-\pen\leq0$ and rearranging gives
$$
\sum_{1\leq i<j\leq m}\gamma_{\sdimn}(\Sigma_{ij}\cap\{x\in\R^{\adimn}\colon\vnorm{x}>r\})
 \sum_{i=1}^{m}\gamma_{\sdimn}(\partial H(w,b_{i-1})\cup\partial H(w,b_{i}))
 $$
In the case that $H(w,b_{i})$ is a ball for each $1\leq i\leq m$, this quantity is bounded by $2m\gamma_{\sdimn}(\{x\in\R^{\adimn}\colon\vnorm{x-w}=r\})$, using also $r>\sqrt{\adimn}+\vnormt{w}$.  In the case that $H(w,b_{i})$ is a slab for each $1\leq i\leq m$, this quantity is bounded by $2m\gamma_{\sdimn}(\{x\in\R^{\adimn}\colon\vnorm{x_{1}-\vnormt{w}}=r\})$, using also $r>1+\vnormt{w}$.  In the
\end{proof}

\begin{lemma}\label{zlem}
Let $\Omega\subset\R^{\adimn}$.  Let $z\colonequals\int_{\Omega}x\gamma_{\adimn}(x)dx$.  Then
$$\vnormt{z}\leq \frac{1}{\sqrt{2\pi}}.$$
\end{lemma}
\begin{proof}
By duality of the $\ell_{2}$ norm, let $\omega\in\R^{\adimn}$ with $\vnormt{\omega}=1$ such that
$$\vnormt{\int_{\Omega}x\gamma_{\adimn}(x)dx}=\big\langle\int_{\Omega}x\gamma_{\adimn}(x)dx,\omega\big\rangle
=\int_{\Omega}\langle x,\omega\rangle\gamma_{\adimn}(x)dx.$$
Then
$$\vnormt{\int_{\Omega}x\gamma_{\adimn}(x)dx}
\leq \int_{\{x\in\R^{\adimn}\colon \langle x,\omega\rangle\geq0\}}\langle x,\omega\rangle\gamma_{\adimn}(x)dx
=\int_{0}^{\infty}x_{1}\gamma_{1}(x_{1})dx_{1}
=\frac{1}{\sqrt{2\pi}}.
$$
\end{proof}

\begin{lemma}\label{nlemma}
Let $\Omega_{1},\ldots,\Omega_{m}$ minimize Problem \ref{prob1p}.  Then $\pcon<\infty$.
\end{lemma}
\begin{proof}
Consider a Dirichlet eigenfunction supported in a large ball.  By Lemma \ref{lemma33},
$$\L F=\pcon F.$$
Moreover, $\forall\,1\leq i<j<k\leq m$, $\forall$ $x\in(\redS_{ij})\cap(\redS_{jk})\cap(\redS_{ki})$, the following holds at $x$.
\begin{equation}\label{feq}
\nabla_{\nu_{ij}}f_{ij}+q_{ij}f_{ij}=\nabla_{\nu_{jk}}f_{jk}+q_{jk}f_{jk}=\nabla_{\nu_{ij}}f_{ij}+q_{ij}f_{ij}.
\end{equation}

Let $\ell$ from Lemma \ref{nolowdim}.  Suppose $\ell<m-1$, so that, after rotation we may assume that $\Omega_{i}=\Omega_{i}'\times\R$ for all $1\leq i\leq m$.  Assume for the sake of contradiction that $\pcon$ can be arbitrarily large.  By replacing $F$ with $F(\cdots,x_{\adimn})+F(\cdots,-x_{\adimn})$, we may assume that $F(\cdots,x_{\adimn})=F(\cdots,-x_{\adimn})$.  Then $G\colonequals x_{\adimn}F$ contradicts the minimality of $\Omega_{1},\ldots,\Omega_{m}$ since

$$
L_{ij}g_{ij}\stackrel{\eqref{zero4}}{=}x_{\adimn}L_{ij}f_{ij}+f_{ij}\mathcal{L}_{ij}x_{\adimn}+\langle\nabla f,\nabla(x_{\adimn})\rangle
=x_{\adimn}L_{ij}f_{ij}-f_{ij}x_{\adimn}.
$$
Here $\mathcal{L}_{ij}\colonequals\Delta-\langle x,\nabla\rangle$.  By \eqref{zero1} and Fubini's Theorem, if we first integrate with respect to $x_{\adimn}$, we see that $(g_{ij})_{1\leq i<j\leq m}$ is automatically Gaussian volume-preserving, so that \eqref{zero1} is zero for all $1\leq i<j\leq m$.  Then the second variation condition from Lemma \ref{lemma28} applies, and we get
\begin{flalign*}
&\frac{d^{2}}{ds^{2}}|_{s=0}\left(\sum_{1\leq i<j\leq m}\int_{\Sigma_{ij}^{(s)}}\gamma_{\sdimn}(x)dx
+\epsilon\pens\right)\\
&\qquad=\sum_{1\leq i<j\leq m}-\int_{\Sigma_{ij}}f_{ij}[L_{ij}f_{ij}-\epsilon\langle X,z^{(i)}-z^{(j)}\rangle]\gamma_{\sdimn}(x)dx
+\frac{d}{ds}|_{s=0}\int_{\redS_{ij}^{(s)}}\langle X,\nu_{ij}\rangle\gamma_{\sdimn}(x)dx\\
&\qquad\qquad\qquad\qquad\qquad\qquad+\epsilon\sqrt{2\pi}\pender.
\end{flalign*}
The middle term is zero by Lemma \ref{lemma29.9} and \eqref{feq}.  The $\epsilon$ term is bounded by $\vnorm{z^{(i)}-z^{(j)}}f_{ij}^{2}$, and the last term is bounded by the Cauchy-Schwarz inequality as
$$\epsilon\sqrt{2\pi}\sum_{1\leq i,j\leq m}\int_{\Sigma_{ij}}\vnorm{x-\overline{w}^{(i)}}^{2}\gamma_{\adimn}(x)dx\cdot
\int_{\Sigma_{ij}}f_{ij}^{2}\gamma_{\adimn}(x)dx.$$
The term on the left is bounded by Lemma \ref{lemma28.8}.  In summary,
\begin{flalign*}
&\frac{d^{2}}{ds^{2}}|_{s=0}\left(\sum_{1\leq i<j\leq m}\int_{\Sigma_{ij}^{(s)}}\gamma_{\sdimn}(x)dx
+\epsilon\pens\right)\\
&\qquad\leq -\sum_{1\leq i<j\leq m}\int_{\Sigma_{ij}}f_{ij}^{2}\gamma_{\sdimn}(x)dx \cdot\Big(\pcon-\epsilon\vnorm{z^{(i)}-z^{(j)}}-\int_{\Sigma_{ij}}\vnorm{x-\overline{w}^{(i)}}^{2}\gamma_{\adimn}(x)dx\Big).
\end{flalign*}
So, in the case $\ell<m-1$, we cannot have arbitrarily large $\pcon$, since this would contradict the minimality of $\Omega_{1},\ldots,\Omega_{m}$ for Problem \ref{prob1p}.  (The two right-most terms are bounded by Lemmas \ref{lemma28.8} and \ref{zlem}.)

In the case $\ell=m-1$, there exists $v\in\R^{\adimn}$, such that $\langle v,N\rangle+F$ is volume preserving, i.e. \eqref{zero1} holds for all $1\leq i\leq m$.  Then Lemma \ref{directlem} together with Lemma \ref{lemma5} gives
\begin{flalign*}
&\frac{d^{2}}{ds^{2}}|_{s=0}\left(\sum_{1\leq i<j\leq m}\int_{\Sigma_{ij}^{(s)}}\gamma_{\sdimn}(x)dx
+\epsilon\pens\right)\\
&=\sum_{1\leq i<j\leq m}\int_{\Sigma_{ij}}\Big(-f_{ij}L_{ij}f_{ij}-f_{ij}\langle v,N_{ij}\rangle-\langle X,v\rangle+\epsilon\langle X,z^{(i)}-z^{(j)}\rangle-\vnorm{v}^{2}+\langle x,v\rangle^{2}\\
&\qquad\qquad\qquad\quad-2\int_{\Sigma_{ij}}(H_{ij}-\langle x,N_{ij}\rangle)f_{ij}\langle x,v\rangle\gamma_{\sdimn}(x)dx
+\int_{\Sigma_{ij}}(H_{ij}-\langle x,N_{ij}\rangle)[f_{ij}\gamma_{\sdimn}(x)dx]'\\
&\qquad+\epsilon\sum_{1\leq i<j\leq m}\int_{\Sigma_{ij}}\Big(\langle x,z^{(i)}-z^{(j)}\rangle+\langle w,\overline{w}^{(i)}-\overline{w}^{(j)}+\overline{z}^{(j)}-\overline{z}^{(i)}\rangle\Big)[f_{ij}\gamma_{\sdimn}(x)dx]'\\
&\qquad+\epsilon\sqrt{2\pi}\sum_{i=1}^{m}\vnormtf{\sum_{j\neq i}\int_{\Sigma_{ij}}(f_{ij}+\langle v,N_{ij}\rangle)(x-\overline{w}^{(i)})}^{2}\\
\end{flalign*}

Applying \eqref{nine5.3}, and using $\sum_{1\leq i<j\leq m}\lambda_{ij}\int_{\Sigma_{ij}}[f_{ij}\gamma_{\sdimn}(x)dx]'=0$ by Lemma \ref{lemma25.3},
\begin{flalign*}
&\frac{d^{2}}{ds^{2}}|_{s=0}\left(\sum_{1\leq i<j\leq m}\int_{\Sigma_{ij}^{(s)}}\gamma_{\sdimn}(x)dx
+\epsilon\pens\right)\\
&=\sum_{1\leq i<j\leq m}\int_{\Sigma_{ij}}\Big(-f_{ij}L_{ij}f_{ij}-f_{ij}\langle v,N_{ij}\rangle-\langle X,v\rangle+\epsilon\langle X,z^{(i)}-z^{(j)}\rangle-\vnorm{v}^{2}+\langle x,v\rangle^{2}\\
&\qquad\qquad\qquad\qquad-2\int_{\Sigma_{ij}}(H_{ij}-\langle x,N_{ij}\rangle)f_{ij}\langle x,v\rangle\gamma_{\sdimn}(x)dx
+\lambda_{ij}\int_{\Sigma_{ij}}[f_{ij}\gamma_{\sdimn}(x)dx]'\\
&\qquad+\epsilon\sqrt{2\pi}\sum_{i=1}^{m}\vnormtf{\sum_{j\neq i}\int_{\Sigma_{ij}}(f_{ij}+\langle v,N_{ij}\rangle)(x-\overline{w}^{(i)})}^{2}\\
&=\sum_{1\leq i<j\leq m}\int_{\Sigma_{ij}}\Big(-f_{ij}L_{ij}f_{ij}-f_{ij}\langle v,N_{ij}\rangle-\langle X,v\rangle+\epsilon\langle X,z^{(i)}-z^{(j)}\rangle-\vnorm{v}^{2}+\langle x,v\rangle^{2}\\
&\quad-2\int_{\Sigma_{ij}}(H_{ij}-\langle x,N_{ij}\rangle)f_{ij}\langle x,v\rangle\gamma_{\sdimn}(x)dx
+\epsilon\sqrt{2\pi}\sum_{i=1}^{m}\vnormtf{\sum_{j\neq i}\int_{\Sigma_{ij}}(f_{ij}+\langle v,N_{ij}\rangle)(x-\overline{w}^{(i)})}^{2}
\end{flalign*}

Using again \eqref{nine5.3} and several applications of the Cauchy-Schwarz inequality, the first term is $\pcon f_{ij}^{2}$, while the remaining terms are all bounded by products of $(\sum_{1\leq i<j\leq m}\int_{\Sigma_{ij}}f_{ij}^{2}\gamma_{\sdimn}(x)dx)^{1/2}$, $(\sum_{1\leq i<j\leq m}\int_{\Sigma_{ij}}\langle v,N_{ij}\rangle^{2}\gamma_{\sdimn}(x)dx)^{1/2}$ and other finite terms (repeating the bounds from the case $\ell<m-1$.)  So, it is once again impossible that $\pcon$ can be made arbitrarily large.  That is, we have found a contradiction in either case.  We conclude that $\pcon<\infty$.
\end{proof}

\begin{lemma}\label{directlem}
Let $x^{(0)}=x$, ${x^{(s)}}'=v$,  ${x^{(s)}}''=0$, and let $\Omega^{(s)}$, $\Sigma^{(s)}$ as in Section \ref{fsvar}.  Then
\begin{flalign*}
&\frac{d}{ds}\int_{\Sigma^{(s)}+x^{(s)}}\gamma_{\sdimn}(x)dx\\
&=\sum_{1\leq i<j\leq m}\int_{\Sigma_{ij}}\Big(-f_{ij}Lf_{ij}-f_{ij}\langle v,N_{ij}\rangle-\langle X,v\rangle-\vnormt{v}^{2}\Big)\gamma_{\sdimn}(x)dx+\int_{\Sigma_{ij}}\langle x,v\rangle^{2}\gamma_{\sdimn}(x)dx\\
&\qquad
-2\lambda_{ij}\int_{\Sigma_{ij}}f_{ij}\langle x,v\rangle\gamma_{\sdimn}(x)dx
+\lambda_{ij}\int_{\Sigma_{ij}}[f_{ij}\gamma_{\sdimn}(x)dx]'\\
&\qquad+\frac{d}{ds}|_{s=0}\int_{\redS_{ij}^{(s)}}\langle X+v,\nu_{ij}\rangle\gamma_{\sdimn}(x)dx.
\end{flalign*}\end{lemma}
\begin{proof}
Let $\Sigma$ be an $n$-dimensional $C^{\infty}$ hypersurface with boundary.  We let $'$ denote $\frac{\partial}{\partial s}|_{s=0}$.

\begin{flalign*}
&\frac{d}{ds}\int_{\Sigma^{(s)}+x^{(s)}}\gamma_{\sdimn}(x)dx
=\int_{\Sigma^{(s)}}\big[(H(x)-\langle N(x),x+x^{(s)}\rangle)f(x)-\langle x+x^{(s)},{x^{(s)}}'\rangle\big]\gamma_{\sdimn}(x+x^{(s)})dx\\
&\qquad\qquad\qquad\qquad\qquad
+\int_{\redS^{(s)}}\langle X+{x^{(s)}}',\nu\rangle\gamma_{\sdimn}(x+{x^{(s)}})dx.
\end{flalign*}

Taking another derivative,
\begin{flalign*}
&\frac{d^{2}}{ds^{2}}|_{s=0}\int_{\Sigma^{(s)}+x^{(s)}}\gamma_{\sdimn}(x)dx
=\int_{\Sigma}(H(x)-\langle N(x),x+{x^{(s)}}\rangle)'f(x)-\langle (x+{x^{(s)}})',{x^{(s)}}'\rangle\gamma_{\sdimn}(x)dx\\
&\qquad+\int_{\Sigma}[(H-\langle x,N\rangle)f(x)-\langle x,{x^{(s)}}'\rangle]^{2}\gamma_{\sdimn}(x)dx\\
&\qquad+\int_{\Sigma}(H-\langle x,N\rangle)f'(x)-\langle x,{x^{(s)}}''\rangle
+\frac{d}{ds}|_{s=0}\int_{\redS^{(s)}}\langle X+{x^{(s)}}',\nu\rangle\gamma_{\sdimn}(x+{x^{(s)}})dx.
\end{flalign*}
From \eqref{nine2.3}, $x'=X=X^{N}+X^{T}=fN+X^{T}$.  Also, $H'=-\Delta f-\vnormt{A}^{2}f$, $N'=-\nabla f$, \cite[A.3, A.4]{colding12a} (the latter calculations require writing $\Sigma^{(s)}$ in the form $\{x+ sN(x)+O_{x}(s^{2})\colon x\in\Sigma\}$).  So,
$$(H-\langle N,x+{x^{(s)}}\rangle)'=-\Delta f-\vnormt{A}^{2}f-\langle N,fN+X^{T}+{x^{(s)}}'\rangle-\langle x,\nabla f\rangle\stackrel{\eqref{three4.5}}{=}-Lf-\langle v,N\rangle.$$
$$\langle (x+{x^{(s)}})',{x^{(s)}}'\rangle=\langle X+v,v\rangle=\langle X,v\rangle+\vnormt{v}^{2}.$$
In summary,
\begin{flalign*}
&\frac{d^{2}}{ds^{2}}|_{s=0}\int_{\Sigma^{(s)}}\gamma_{\sdimn}(x)dx
=\int_{\Sigma}\Big(-fLf-f\langle v,N\rangle-\langle X,v\rangle-\vnormt{v}^{2}\Big)\gamma_{\sdimn}(x)dx\\
&\qquad+\int_{\Sigma}[(H-\langle x,N\rangle)f(x)-\langle x,v\rangle]^{2}\gamma_{\sdimn}(x)dx\\
&\qquad+\int_{\Sigma}(H-\langle x,N\rangle)f'(x)+\frac{d}{ds}|_{s=0}\int_{\redS^{(s)}}\langle X+v,\nu\rangle\gamma_{\sdimn}(x)dx\\
&\qquad+\int_{\redS}\langle X+{x^{(s)}}'',\nu\rangle\gamma_{\sdimn}(x)dx
+\int_{\redS}\langle X+v,\nu\rangle\langle x,-v\rangle\gamma_{\sdimn}(x)dx.
\end{flalign*}
Summing over all $1\leq i<j\leq m$,
\begin{flalign*}
&\frac{d^{2}}{ds^{2}}|_{s=0}\sum_{1\leq i<j\leq m}\int_{\Sigma_{ij}^{(s)}}\gamma_{\sdimn}(x)dx
=\sum_{1\leq i<j\leq m}\int_{\Sigma_{ij}}\Big(-f_{ij}Lf_{ij}-f_{ij}\langle v,N_{ij}\rangle-\langle X,v\rangle\\
&\qquad-\vnormt{v}^{2}\Big)\gamma_{\sdimn}(x)dx+\int_{\Sigma_{ij}}((H_{ij}-\langle x,N_{ij}\rangle)f_{ij}-\langle x,v\rangle)^{2}\gamma_{\sdimn}(x)dx\\
&\qquad+\int_{\Sigma_{ij}}(H_{ij}-\langle x,N_{ij}\rangle)f_{ij}'\gamma_{\sdimn}(x)dx +\frac{d}{ds}|_{s=0}\int_{\redS_{ij}^{(s)}}\langle X+v,\nu_{ij}\rangle\gamma_{\sdimn}(x)dx\\
&=\sum_{1\leq i<j\leq m}\int_{\Sigma_{ij}}\Big(-f_{ij}Lf_{ij}-f_{ij}\langle v,N_{ij}\rangle-\langle X,v\rangle-\vnormt{v}^{2}\Big)\gamma_{\sdimn}(x)dx+\int_{\Sigma_{ij}}\langle x,v\rangle^{2}\gamma_{\sdimn}(x)dx\\
&\qquad
-2\int_{\Sigma_{ij}}(H_{ij}-\langle x,N_{ij}\rangle)f_{ij}\langle x,v\rangle\gamma_{\sdimn}(x)dx\\
&\qquad+\int_{\Sigma_{ij}}(H_{ij}-\langle x,N_{ij}\rangle)(f_{ij}'+H_{ij}-\langle x,N_{ij}\rangle)\gamma_{\sdimn}(x)dx
+\frac{d}{ds}|_{s=0}\int_{\redS_{ij}^{(s)}}\langle X+v,\nu_{ij}\rangle\gamma_{\sdimn}(x)dx
\end{flalign*}
\begin{flalign*}
&=\sum_{1\leq i<j\leq m}\int_{\Sigma_{ij}}\Big(-f_{ij}Lf_{ij}-f_{ij}\langle v,N_{ij}\rangle-\langle X,v\rangle-\vnormt{v}^{2}\Big)\gamma_{\sdimn}(x)dx+\int_{\Sigma_{ij}}\langle x,v\rangle^{2}\gamma_{\sdimn}(x)dx\\
&\qquad
-2\int_{\Sigma_{ij}}(H_{ij}-\langle x,N_{ij}\rangle)f_{ij}\langle x,v\rangle\gamma_{\sdimn}(x)dx
+\int_{\Sigma_{ij}}(H_{ij}-\langle x,N_{ij}\rangle)[f_{ij}\gamma_{\sdimn}(x)dx]'\\
&\qquad+\frac{d}{ds}|_{s=0}\int_{\redS_{ij}^{(s)}}\langle X+v,\nu_{ij}\rangle\gamma_{\sdimn}(x)dx.
\end{flalign*}
\end{proof}

The following curvature bound is adapted from \cite[Lemma 5.1]{mcgonagle15}.

\begin{lemma}\label{lemma29b}
Let $\Omega_{1},\ldots,\Omega_{m}$ minimize Problem \ref{prob1}.  Then $\forall$ $\phi\in C_{0}^{\infty}(\Sigma)$,
$$
\sum_{1\leq i<j\leq m}\int_{\Sigma_{ij}}\phi^{2}\vnorm{A}^{2}\gamma_{\sdimn}(x)dx
\leq \sum_{1\leq i<j\leq m}\int_{\Sigma_{ij}}\Big((\pcon-1)\phi^{2}+\vnorm{\nabla\phi}^{2}\Big)\gamma_{\sdimn}(x)dx.
$$
\end{lemma}
\begin{proof}
Let $G\colonequals\{\alpha_{ij}\}_{1\leq i<j\leq m}$ be a solution to the system of middle equations of \eqref{nine5.3}.  Let $\phi\in C_{0}^{\infty}(\R^{\adimn})$.  By the definition \eqref{seven0} of $\pcon$,
$$-Q(\phi G,\phi G)\leq\pcon\langle \phi G,\phi G\rangle.$$
That is, by Lemma \ref{lemma32.5} %\snote{added extra term now; then what happens to it?  does it matter?}
\begin{flalign*}
&\sum_{1\leq i<j\leq m}\int_{\Sigma_{ij}}\alpha_{ij}^{2}(-\vnorm{\nabla\phi}^{2}+\phi^{2}(\vnorm{A}^{2}+1))\gamma_{\sdimn}(x)dx\\
&\qquad\qquad
+\sum_{1\leq i<j<k\leq m}\int_{\redS_{ij}\cap\redS_{jk}\cap \redS_{ki}}
\phi^{2}\Big(q_{ij}\alpha_{ij}^{2}+q_{jk}\alpha_{jk}^{2}+q_{ki}\alpha_{ki}^{2}\Big)\gamma_{\sdimn}(x)\\
%&\qquad\qquad+\sum_{i=1}^{m}\sum_{j\in\{1,\ldots,m\}\colon j\neq i}\int_{\Sigma_{ij}}\alpha_{ij}^{2}\langle z^{(i)},N_{ij}\rangle\gamma_{\sdimn}(x)dx\\
&\qquad\leq \pcon\sum_{1\leq i<j\leq m}\int_{\Sigma_{ij}}\phi^{2}\alpha_{ij}^{2}\gamma_{\sdimn}(x)dx.
\end{flalign*}
Summing these quantities over all permutations of $\{1,\ldots,m\}$, i.e. permuting $\{\alpha_{ij}\}_{1\leq i<j\leq m}$, the middle term vanishes by Remark \ref{rk5}, and we get
$$
\sum_{1\leq i<j\leq m}\int_{\Sigma_{ij}}(-\vnorm{\nabla\phi}^{2}+\phi^{2}(\vnorm{A}^{2}+1))\gamma_{\sdimn}(x)dx
\leq \pcon\sum_{1\leq i<j\leq m}\int_{\Sigma_{ij}}\phi^{2}\gamma_{\sdimn}(x)dx.
$$
Rearranging completes the proof.

\end{proof}

\begin{lemma}[{\cite[Lemma 6.2]{zhu16}}]\label{lemma28.5}
Let $\Omega_{1},\ldots,\Omega_{m}$ minimize Problem \ref{prob1}.  If $\int_{\Sigma}(\abs{\phi}^{2}+\vnormt{\nabla \phi}^{2})\gamma_{\sdimn}(x)dx<\infty$ and if $\phi$ is bounded, then

$$\int_{\Sigma}\phi^{2}(\vnormt{A}^{2}+1)\gamma_{\sdimn}(x)dx
\leq \int_{\Sigma}(\vnormt{\nabla\phi}^{2}+(\pcon-1)\phi^{2})\gamma_{\sdimn}(x)dx.$$
\end{lemma}
\begin{proof}

Apply Lemma \ref{lemma29b}, Lemma \ref{lemma60} and Fatou's Lemma.
\end{proof}

The following Lemmas follow from Lemma \ref{lemma28.5}.

\begin{lemma}\label{lemma61}
Let $\Omega_{1},\ldots,\Omega_{m}$ minimize Problem \ref{prob1}.  Then
$$\sum_{1\leq i<j\leq m}\int_{\redb\Omega_{ij}}\vnormt{A}^{2}\gamma_{\sdimn}(x)dx<\infty.$$
Consequently, for any $v\in\R^{\adimn}$, by \eqref{three2},
$$\sum_{1\leq i<j\leq m}\int_{\redb\Omega_{ij}}\vnorm{\nabla\langle v,N\rangle}^{2}\gamma_{\sdimn}(x)dx<\infty$$
\end{lemma}
\begin{proof}
Use Lemma \ref{lemma28.5}, \eqref{three2} and Lemma \ref{nlemma}.
\end{proof}

\section{Dimension Reduction}\label{secred}

Recall that for all $1\leq i\leq m$,
\begin{equation}\label{zdef}
z^{(i)}\colonequals\int_{\Omega_{i}}x\gamma_{\adimn}(x)dx\in\R^{\adimn}.
\end{equation}

\begin{theorem}[\embolden{Dimension Reduction for Gaussian Minimal Bubbles}]\label{dimredthmp}
Suppose $\Omega_{1},\ldots\Omega_{m}\subset\R^{\adimn}$ minimize Problem \ref{prob1p} with $\epsilon\colonequals \epschoice$.  There exists $0\leq \ell\leq m-1$ and there exist $\Omega_{1}',\ldots,\Omega_{m}'\subset\R^{\ell}$ such that, after rotating $\Omega_{1},\ldots,\Omega_{m}$, we have
$$\Omega_{i}=\Omega_{i}'\times\R^{\sdimn-\ell+1}.$$
Moreover $\ell$ can be chosen to be the dimension of the span of
$$\Big\{\Big(\int_{\redb\Omega_{1}}\sum_{\substack{j\in\{1,\ldots,m\}\colon\\ j\neq 1}}\langle v,N_{1j}\rangle\gamma_{\sdimn}(x)dx,
\ldots,\int_{\redb\Omega_{m}}\sum_{\substack{j\in\{1,\ldots,m\}\colon\\ j\neq m}}\langle v,N_{mj}\rangle\gamma_{\sdimn}(x)dx\Big)\in\R^{m}\colon v\in\R^{\adimn}\Big\}.
$$
\end{theorem}

\begin{proof}
For any $v\in\R^{\adimn}$, define
$$T(v)\colonequals\Big(\int_{\redb\Omega_{1}}\sum_{j\in\{1,\ldots,m\}\colon j\neq 1}\langle v,N_{1j}\rangle\gamma_{\sdimn}(x)dx,
\ldots,\int_{\redb\Omega_{m}}\sum_{j\in\{1,\ldots,m\}\colon j\neq m}\langle v,N_{mj}\rangle\gamma_{\sdimn}(x)dx\Big).$$
Then $T\colon\R^{\adimn}\to\R^{m}$ is linear.  By the rank-nullity theorem, the dimension of the kernel of $T$ plus the dimension of the image of $T$ is $\adimn$.  Since the sum of the indices of $T(v)$ is zero for any $v\in\R^{\adimn}$ (since $N_{ij}=-N_{ji}$ $\forall$ $1\leq i<j\leq m$ by Definition \ref{defnote}), the dimension $\ell$ of the image of $T$ is at most $m-1$.

Let $v$ in the kernel of $T$.  For any $1\leq i<j\leq m$, let $f_{ij}\colonequals\phi\langle v,N_{ij}\rangle$.  Let $X\colonequals \phi v$ be the chosen vector field.  Since $\Omega_{1},\ldots,\Omega_{m}$ minimize Problem \ref{prob1p},
$$0\leq \frac{d^{2}}{ds^{2}}|_{s=0}\sum_{1\leq i<j\leq m}\int_{\Sigma_{ij}^{(s)}}\gamma_{\sdimn}(x)dx+\epsilon\pens.$$
From Lemmas \ref{lemma28}, \ref{lemma29.9}, \ref{lemma97}, \ref{lemma27}, and then letting $\phi$ increase monotonically to $1$ (as in Lemma \ref{lemma97}),
\begin{flalign*}
0&\leq\sum_{1\leq i<j\leq m}-\int_{\Sigma_{ij}}f_{ij}[L_{ij}f_{ij}-\epsilon\langle X,z^{(i)}-z^{(j)}\rangle]\gamma_{\sdimn}(x)dx\\
&\qquad\qquad+\epsilon\sqrt{2\pi}\pender.
\end{flalign*}
Since $T(v)=0$, $\langle v,z^{(i)}\rangle=0$ for all $1\leq i\leq m$ by Lemma \ref{lemma80p}.  So,
$$0\leq\sum_{1\leq i<j\leq m}-\int_{\Sigma_{ij}}f_{ij}L_{ij}f_{ij}\gamma_{\sdimn}(x)dx+\epsilon\sqrt{2\pi}\pender.$$
By Lemma \ref{lemma45},
\begin{flalign*}
0&\leq\sum_{1\leq i<j\leq m}\int_{\Sigma_{ij}}-f_{ij}^{2}+\epsilon f_{ij}\langle v,z^{(i)}-z^{(j)}\rangle-\epsilon f_{ij}^{2}\langle z^{(i)}-z^{(j)},N\rangle\gamma_{\sdimn}(x)dx\\
&\qquad\qquad+\epsilon\sqrt{2\pi}\pender.
\end{flalign*}
Using again $T(v)=0$,  we get
\begin{flalign*}
0&\leq\sum_{1\leq i<j\leq m}\int_{\Sigma_{ij}}-f_{ij}^{2}-\epsilon f_{ij}^{2}\langle z^{(i)}-z^{(j)},N\rangle\gamma_{\sdimn}(x)dx\\
&\qquad\qquad+\epsilon\sqrt{2\pi}\pender.
\end{flalign*}
For each $1\leq i\leq m$, by duality of the $\ell_{2}$ norm, $\exists$ $\omega^{(i)}\in\R^{\adimn}$ with $\vnormtf{\omega^{(i)}}=1$ such that
\begin{flalign*}
&\vnormtf{\sum_{j\in\{1,\ldots,m\}\colon j\neq i}\int_{\Sigma_{ij}}(x-\overline{w}^{(i)})f_{ij}\gamma_{\adimn}(x)dx}
=\Big\langle\sum_{j\in\{1,\ldots,m\}\colon j\neq i}\int_{\Sigma_{ij}}(x-\overline{w}^{(i)})f_{ij}\gamma_{\adimn}(x)dx,\omega^{(i)}\Big\rangle\\
&\qquad\qquad\qquad\qquad\qquad\qquad=\sum_{j\in\{1,\ldots,m\}\colon j\neq i}\int_{\Sigma_{ij}}\langle(x-\overline{w}^{(i)}),\omega^{(i)}\rangle f_{ij}\gamma_{\adimn}(x)dx.
\end{flalign*}
So, we apply the Cauchy-Schwarz inequality to the last term to get
$$0\leq\sum_{1\leq i<j\leq m}-\int_{\Sigma_{ij}}f_{ij}^{2}\gamma_{\sdimn}(x)dx\cdot\left(1-\epsilon\vnorm{z^{(i)}-z^{(j)}}-\epsilon \sqrt{2\pi}\int_{\Sigma_{ij}}\langle x-\overline{w}^{(i)},\omega^{(i)}\rangle^{2}\gamma_{\sdimn}(x)dx\right).$$
By Lemma \ref{zlem} and the second part of Lemma \ref{lemma28.8}, if $\epsilon<\epschoice$, then the last term is positive.  In summary, for any $v$ in the kernel of $T$, $\forall$ $1\leq i<j\leq m$, $f_{ij}(x)=\langle v,N_{ij}(x)\rangle=0$ for all $x\in\Sigma_{ij}$.  That is, $\exists$ $0\leq\ell\leq m-1$ as stated in the conclusion of Theorem \ref{dimredthm}, since the image of $T$ is the span of
$$\Big\{\Big(\int_{\redb\Omega_{1}}\sum_{\substack{j\in\{1,\ldots,m\}\colon\\ j\neq 1}}\langle v,N_{1j}\rangle\gamma_{\sdimn}(x)dx,
\ldots,\int_{\redb\Omega_{m}}\sum_{\substack{j\in\{1,\ldots,m\}\colon\\ j\neq m}}\langle v,N_{mj}\rangle\gamma_{\sdimn}(x)dx\Big)\in\R^{m}\colon v\in\R^{\adimn}\Big\}.$$
\end{proof}

The Divergence Theorem implies the following

\begin{lemma}\label{lemma80p}
Let $\Omega\subset\R^{\adimn}$ be a $C^{\infty}$ manifold with boundary.  Then
\begin{equation}\label{zeq}
\int_{\Omega}x\gamma_{\adimn}(x)dx=-\int_{\partial\Omega}N\gamma_{\adimn}(x)dx.
\end{equation}
\end{lemma}

\section{Flatness}\label{secflat}

Recall the definition of $z^{(i)}$ from \eqref{zdef}.

\begin{lemma}[\embolden{Flatness, Version 1}]\label{nolowdimp}
Suppose $\Omega_{1},\ldots\Omega_{m}\subset\R^{\adimn}$ minimize Problem \ref{prob1p} with $\epsilon\colonequals\epschoice$.   Then for all $1\leq i<j\leq m$, $z^{(i)}-z^{(j)}$ is parallel to $N_{ij}$; moreover if $z^{(i)}-z^{(j)}\neq0$, then $\Sigma_{ij}$ is a union of relatively open subsets of hyperplanes.

\end{lemma}
\begin{proof}[Proof of Lemma \ref{nolowdimp}]
Assume for now that $\adimn>m-1$ so that $\sdimn-\ell+1\geq\sdimn-(m-1)+1>0$ in Theorem \ref{dimredthm}.  Then after rotating $\Omega_{1},\ldots,\Omega_{m}$, we have
$$\Omega_{i}=\Omega_{i}'\times\R.$$
Let $\{\alpha_{ij}\}_{1\leq i<j\leq m}$ be constants guaranteed to exist by Lemma \ref{lemma55}.  Let $f_{ij}\colonequals\alpha_{ij}$ for all $1\leq i<j\leq m$. Define now a new function $g_{ij}\colonequals x_{\adimn}f_{ij}$.  Since $f_{ij}$ is only a function of the variables $x_{1},\ldots,x_{\sdimn}$, we have $\langle\nabla f,\nabla(x_{\adimn})\rangle=0$.  So, for any $1\leq i<j\leq m$, the product rule for $L_{ij}$ (Remark \ref{rk20}) gives
\begin{equation}\label{deq}
L_{ij}g_{ij}\stackrel{\eqref{zero4}}{=}x_{\adimn}L_{ij}f_{ij}+f_{ij}\mathcal{L}_{ij}x_{\adimn}+\langle\nabla f,\nabla(x_{\adimn})\rangle
=x_{\adimn}L_{ij}f_{ij}-f_{ij}x_{\adimn}.
\end{equation}
Here $\mathcal{L}_{ij}\colonequals\Delta-\langle x,\nabla\rangle$.  By \eqref{zero1} and Fubini's Theorem, if we first integrate with respect to $x_{\adimn}$, we see that $(g_{ij})_{1\leq i<j\leq m}$ is automatically Gaussian volume-preserving, so that \eqref{zero1} is zero for all $1\leq i<j\leq m$.  Then the second-variation condition applies, and we get, using Lemma \ref{lemma28}
\begin{flalign*}
&\frac{d^{2}}{ds^{2}}|_{s=0}\Big(\sum_{1\leq i<j\leq m}\int_{\Sigma_{ij}^{(s)}}\gamma_{\sdimn}(x)dx
+\epsilon\pens\Big)\\
&=\sum_{1\leq i<j\leq m}-\int_{\Sigma_{ij}}g_{ij}[L_{ij}g_{ij}-\epsilon\langle X,z^{(i)}-z^{(j)}\rangle]\gamma_{\sdimn}(x)dx\\
&\qquad\qquad\qquad\qquad\qquad\qquad+\epsilon\sqrt{2\pi}\pender.
\end{flalign*}

Using the definition of $g_{ij}$ this simplifies to

\begin{flalign*}
&\frac{d^{2}}{ds^{2}}|_{s=0}\Big(\sum_{1\leq i<j\leq m}\int_{\Sigma_{ij}^{(s)}}\gamma_{\sdimn}(x)dx
+\epsilon\pens\Big)\\
&=\sum_{1\leq i<j\leq m}\int_{\Sigma_{ij}}[-x_{\adimn}^{2}\alpha_{ij}^{2}\vnormt{A}^{2}+\epsilon\alpha_{ij}x_{\adimn}^{2}\langle X,z^{(i)}-z^{(j)}\rangle]\gamma_{\sdimn}(x)dx\\
&\qquad\qquad\qquad\qquad\qquad\qquad+\epsilon\sqrt{2\pi}\sum_{i=1}^{m}\vnormt{\sum_{j\in\{1,\ldots,m\}\colon j\neq i}\int_{\Sigma_{ij}}(x-\overline{w}^{(i)})x_{\adimn}\alpha_{ij}\gamma_{\adimn}(x)dx}^{2}.
\end{flalign*}
Or, after integrating with respect to $x_{\adimn}$,
\begin{flalign*}
&\frac{d^{2}}{ds^{2}}|_{s=0}\sum_{1\leq i<j\leq m}\int_{\Sigma_{ij}^{(s)}}\gamma_{\sdimn}(x)dx
+\epsilon\pens\\
&=\sum_{1\leq i<j\leq m}\int_{\Sigma_{ij}'}[-\alpha_{ij}^{2}\vnormt{A}^{2}+\epsilon\alpha_{ij}\langle X,z^{(i)}-z^{(j)}\rangle]\gamma_{\sdimn-1}(x)dx\\
&\qquad\qquad\qquad\qquad\qquad\qquad+\epsilon\sqrt{2\pi}\sum_{i=1}^{m}\Big(\sum_{j\in\{1,\ldots,m\}\colon j\neq i}\int_{\Sigma_{ij}'}\alpha_{ij}\gamma_{\sdimn}(x)dx\Big)^{2}.
\end{flalign*}
The quantity $\langle X,z^{(i)}-z^{(j)}\rangle$ is the only term in the above expression that can possibly depend on the tangential component of $X$ (i.e. $X-\langle X,N_{ij}\rangle N_{ij}$).  This term can be changed arbitrarily by adding a tangential component to $X$ while leaving the other terms the same.  Therefore, for all $1\leq i<j\leq m$, we must have $z^{(i)}-z^{(j)}$ parallel to $N_{ij}$.  So, if any $z^{(i)}-z^{(j)}$ is nonzero, $N_{ij}$ must be a constant multiple of $z^{(i)}-z^{(j)}$, so that $\Sigma_{ij}$ is flat
\end{proof}
\begin{remark}
The above argument crucially relies on Theorem \ref{dimredthmp}.  Without Theorem \ref{dimredthmp}, the vector field $X$ could not satisfy the properties used above, while preserving the Gaussian volumes of all of the sets.
\end{remark}
\begin{remark}\label{zrk}
It is possible to show that $z^{(i)}-z^{(j)}\neq0$ for all $1\leq i<j\leq m$ in Lemma \ref{nolowdimp}, albeit with a nonexplicit constant $\epsilon$.

To see this, suppose there exists a sequence $\epsilon_{1}>\epsilon_{2}>\cdots$ tending to zero such that, without loss of generality, $z^{(1)}-z^{(2)}\neq0$ for a minimizer of
$$\sum_{1\leq i<j\leq m}\int_{\Sigma_{ij}^{(s)}}\gamma_{\sdimn}(x)dx
+\epsilon_{k}\pen,\qquad\forall k\geq1.$$
Suppose the minimal sets are $\Omega_{1,k},\ldots,\Omega_{m,k}$.  By Theorem \ref{dimredthmp}, we may assume $\Omega_{1,k},\ldots,\Omega_{m,k}\subset\R^{m-1}$.  In fact, by Theorem \ref{dimredthmp}, we may assume that $\Omega_{1,k},\ldots,\Omega_{m,k}\subset\R^{m-2}$.  We now argue as in \cite[Proposition 1]{barchiesi16}.  By taking a subsequence (and relabeling the original sequence as this subsequence), $\exists$ $\Omega_{1},\ldots,\Omega_{m}\subset\R^{m-2}$ such that $1_{\Omega_{i,k}}$ converges to $1_{\Omega_{i}}$ as $k\to\infty$, for all $1\leq i\leq m$, in the local $L_{1}(\gamma_{m-2})$ sense.  But this violates Conjecture \ref{conj0}, i.e. the (uniqueness part of the) main result of \cite{milman18b}.  The sets of minimal surface area cannot be described as subsets of $\R^{m-2}$.
\end{remark}

In the cases $m\leq 4$, we can upgrade Remark \ref{zrk} in order to get an explicit bound on $\epsilon$.  Cases of larger $m$ seem increasingly difficult at present.

\begin{lemma}[\embolden{Flatness, Version 2}]\label{flat2}
Suppose $\Omega_{1},\ldots\Omega_{m}\subset\R^{\adimn}$ minimize Problem \ref{prob1p} with $\epsilon\colonequals\epschoice$.     If $1\leq m\leq 5$, then for all $1\leq i\leq m$, $\partial\Omega_{i}$ consists of a union of relatively open subsets of hyperplanes.
\end{lemma}
\begin{proof}
From Lemma \ref{lemma55}, there exists an $(m-1)$-dimensional space of $X$ such that for each $1\leq i<j\leq m$, $\langle X,N_{ij}\rangle\equalscolon\alpha_{ij}$ is constant on $\Sigma_{ij}$.  Let $v\in\R^{\adimn}$ and for any $1\leq i<j\leq m$ consider the functions
\begin{equation}\label{zero5.5p}
f_{ij}\colonequals\alpha_{ij}+\langle v,N_{ij}\rangle.
\end{equation}
More specifically, consider the second variation of this function from Lemma \ref{lemma28}.  This quantity is equal to
\begin{flalign*}
&\frac{d^{2}}{ds^{2}}|_{s=0}\Big(\sum_{1\leq i<j\leq m}\int_{\Sigma_{ij}^{(s)}}\gamma_{\sdimn}(x)dx+\epsilon\pens\Big)\\
&\qquad
=\sum_{1\leq i<j\leq m}-\int_{\Sigma_{ij}}(\alpha_{ij}+\langle v,N_{ij}\rangle)[L_{ij}(\alpha_{ij}+\langle v,N_{ij}\rangle)-\epsilon\langle X,z^{(i)}-z^{(j)}\rangle]\gamma_{\sdimn}(x)dx\\
&\qquad\qquad\qquad+\epsilon\sqrt{2\pi}\pender\\
&\qquad\qquad\qquad+\sum_{1\leq i<j<k\leq m}\int_{\redS_{ij}\cap\redS_{jk}\cap \redS_{ki}}
\Big([\nabla_{\nu_{ij}}f_{ij}+q_{ij}f_{ij}]f_{ij}+[\nabla_{\nu_{jk}}f_{jk}+q_{jk}f_{jk}]f_{jk}\\
 &\qquad\qquad\qquad\qquad\qquad\qquad\qquad\qquad\qquad+[\nabla_{\nu_{ki}}f_{ki}+q_{ki}f_{ki}]f_{ki}\Big)\gamma_{\sdimn}(x)dx.
 \end{flalign*}
From Lemma \ref{lemma45}, $L_{ij}\langle v,N_{ij}\rangle=\langle v,N_{ij}\rangle-\epsilon\langle v,z^{(i)}-z^{(j)}\rangle+\epsilon\langle v,N_{ij}\rangle\langle N_{ij},z^{(i)}-z^{(j)}\rangle$ and by Remark \ref{rk30}, $\forall$ $1\leq i<j<k\leq m$, $\forall$ $x\in\redb\Sigma_{ij}\cap\redb\Sigma_{jk}\cap\redb\Sigma_{ki}$,
 \begin{equation}\label{zero7p}
 \nabla_{\nu_{ij}}\langle v,N_{ij}\rangle+q_{ij}\langle v,N_{ij}\rangle
 =\nabla_{\nu_{jk}}\langle v,N_{jk}\rangle+q_{jk}\langle v,N_{jk}\rangle
 =\nabla_{\nu_{ki}}\langle v,N_{ki}\rangle+q_{ki}\langle v,N_{ki}\rangle.
 \end{equation}
 Therefore, using also \eqref{zero4}, $\nabla \alpha_{ij}=0$ for all $1\leq i<j\leq m$, and $f_{ij}+f_{jk}+f_{ki}=\langle X,N_{ij}+N_{jk}+N_{ki}\rangle=0$, since $N_{ij}+N_{jk}+N_{ki}=0$ by Lemma \ref{lemma52.6},
 \begin{equation}\label{zero8p}
 \begin{aligned}
&\frac{d^{2}}{ds^{2}}|_{s=0}\Big(\sum_{1\leq i<j\leq m}\int_{\Sigma_{ij}^{(s)}}\gamma_{\sdimn}(x)dx+\epsilon\pens\Big)\\
&\qquad
=\sum_{1\leq i<j\leq m}-\int_{\Sigma_{ij}}(\alpha_{ij}+\langle v,N_{ij}\rangle)\Big(\alpha_{ij}(\vnormt{A}^{2}+1)+\langle v,N_{ij}\rangle-\epsilon\langle v,z^{(i)}-z^{(j)}\rangle\\
&\qquad\qquad\qquad\qquad\qquad\qquad\qquad\quad+\epsilon\langle v,N_{ij}\rangle\langle N_{ij},z^{(i)}-z^{(j)}\rangle-\epsilon\langle X,z^{(i)}-z^{(j)}\rangle\Big)\gamma_{\sdimn}(x)dx\\
&\qquad\qquad\qquad+\epsilon\sqrt{2\pi}\pender\\
&\qquad\qquad\qquad+\sum_{1\leq i<j<k\leq m}\int_{\redS_{ij}\cap\redS_{jk}\cap \redS_{ki}}
\Big([q_{ij}\alpha_{ij}]f_{ij}+[q_{jk}\alpha_{jk}]f_{jk}+[q_{ki}\alpha_{ki}]f_{ki}\Big)\gamma_{\sdimn}(x)dx.
 \end{aligned}
 \end{equation}
 Integrating by parts, (which is justified by Lemmas \ref{lemma32.5} and \ref{lemma61}),
 \begin{flalign*}
 &\sum_{1\leq i<j\leq m}\int_{\Sigma_{ij}}\alpha_{ij}\Big(\langle v,N_{ij}\rangle-\epsilon\langle v,z^{(i)}-z^{(j)}\rangle+\epsilon\langle v,N_{ij}\rangle\langle N_{ij},z^{(i)}-z^{(j)}\rangle\Big)\gamma_{\sdimn}(x)dx\\
 & \stackrel{\eqref{three0p}}{=}\sum_{1\leq i<j\leq m}\int_{\Sigma_{ij}}\alpha_{ij}L_{ij}\langle v,N_{ij}\rangle\gamma_{\sdimn}(x)dx\\
 &=\sum_{1\leq i<j\leq m}\int_{\Sigma_{ij}}\alpha_{ij}\mathrm{div}_{\tau}[\gamma_{\sdimn}(x)\nabla\langle v,N_{ij}\rangle]+\alpha_{ij}\langle v,N_{ij}\rangle(\vnormt{A}^{2}+1)\gamma_{\sdimn}(x)dx\\
 &=\sum_{1\leq i<j<k\leq m}\int_{\redS_{ij}\cap\redS_{jk}\cap \redS_{ki}}
\Big(\alpha_{ij}\nabla_{\nu_{ij}}\langle v,N_{ij}\rangle
+\alpha_{jk}\nabla_{\nu_{jk}}\langle v,N_{jk}\rangle
+\alpha_{ki}\nabla_{\nu_{ki}}\langle v,N_{ki}\rangle\Big)\gamma_{\sdimn}(x)dx\\
&\qquad+ \sum_{1\leq i<j\leq m}\int_{\Sigma_{ij}}\alpha_{ij}\langle v,N_{ij}\rangle(\vnormt{A}^{2}+1)\gamma_{\sdimn}(x)dx.
 \end{flalign*}
Above we denoted $\mathrm{div}_{\tau}$ as the (tangential) divergence on $\Sigma_{ij}$.  Subtracting the last term from the first term, we get
\begin{flalign*}
&\sum_{1\leq i<j<k\leq m}\int_{\redS_{ij}\cap\redS_{jk}\cap \redS_{ki}}
\alpha_{ij}\nabla_{\nu_{ij}}\langle v,N_{ij}\rangle
+\alpha_{jk}\nabla_{\nu_{jk}}\langle v,N_{jk}\rangle
+\alpha_{ki}\nabla_{\nu_{ki}}\langle v,N_{ki}\rangle\gamma_{\sdimn}(x)dx\\
&\quad=-\sum_{1\leq i<j\leq m}\int_{\Sigma_{ij}}\alpha_{ij}\Big(\langle v,N_{ij}\rangle\vnormt{A}^{2}+\epsilon\langle v,z^{(i)}-z^{(j)}\rangle-\epsilon\langle v,N_{ij}\rangle\langle N_{ij},z^{(i)}-z^{(j)}\rangle\Big)\gamma_{\sdimn}(x)dx.
 \end{flalign*}
 Substituting back into \eqref{zero8p}, and then canceling terms using \eqref{zero7p},

\begin{flalign*}
&\frac{d^{2}}{ds^{2}}|_{s=0}\Big(\sum_{1\leq i<j\leq m}\int_{\Sigma_{ij}^{(s)}}\gamma_{\sdimn}(x)dx+\epsilon\pens\Big)\\
&=\sum_{1\leq i<j\leq m}-\int_{\Sigma_{ij}}\Big((\alpha_{ij}+\langle v,N_{ij}\rangle)^{2}+\vnormt{A}^{2}\alpha_{ij}^{2}\Big)\gamma_{\sdimn}(x)dx\\
&\qquad+\epsilon\sum_{1\leq i<j\leq m}\int_{\Sigma_{ij}}
\alpha_{ij}\Big(
\langle v,z^{(i)}-z^{(j)}\rangle
-\langle v,N_{ij}\rangle\langle N_{ij},z^{(i)}-z^{(j)}\rangle
\Big)
\gamma_{\sdimn}(x)dx\\
&\qquad
+\epsilon\sum_{1\leq i<j\leq m}\int_{\Sigma_{ij}}
(\alpha_{ij}+\langle v,N_{ij}\rangle)\Big(
\langle v,z^{(i)}-z^{(j)}\rangle
-\langle v,N_{ij}\rangle\langle N_{ij},z^{(i)}-z^{(j)}\rangle\\
&\qquad\qquad\qquad\qquad\qquad\qquad\qquad\qquad\qquad
+\langle X,z^{(i)}-z^{(j)}\rangle
\Big)
\gamma_{\sdimn}(x)dx\\
%-\epsilon\sum_{1\leq i<j\leq m}\epsilon\int_{\Sigma_{ij}}\langle v,N_{ij}\rangle^{2}\langle N_{ij},z^{(i)}-z^{(j)}\rangle\gamma_{\sdimn}(x)dx\\
&\qquad+\epsilon\sqrt{2\pi}\pender\\
&\qquad+\sum_{1\leq i<j<k\leq m}\int_{\redS_{ij}\cap\redS_{jk}\cap \redS_{ki}}
\Big(\alpha_{ij}[f_{ij}q_{ij}+\nabla_{\nu_{ij}}\langle v,N_{ij}\rangle]+\alpha_{jk}[f_{jk}q_{jk}+\nabla_{\nu_{jk}}\langle v,N_{jk}\rangle]\\
 &\qquad\qquad\qquad\qquad\qquad\qquad\qquad\qquad\qquad+\alpha_{ki}[f_{ki}q_{ki}+\nabla_{\nu_{ki}}\langle v,N_{ki}\rangle]\Big)\gamma_{\sdimn}(x)dx\\
%\end{flalign*}
%\begin{flalign*}
 &=\sum_{1\leq i<j\leq m}-\int_{\Sigma_{ij}}\Big((\alpha_{ij}+\langle v,N_{ij}\rangle)^{2}+\vnormt{A}^{2}\alpha_{ij}^{2}\Big)\gamma_{\sdimn}(x)dx\\
&\qquad+\epsilon\sum_{1\leq i<j\leq m}\int_{\Sigma_{ij}}
(\alpha_{ij}+\langle v,N_{ij}\rangle)\Big(
\langle v,z^{(i)}-z^{(j)}\rangle
-\langle v,N_{ij}\rangle\langle N_{ij},z^{(i)}-z^{(j)}\rangle
\Big)\gamma_{\sdimn}(x)dx\\
&\qquad+\epsilon\sum_{1\leq i<j\leq m}\int_{\Sigma_{ij}}(\alpha_{ij}+\langle v,N_{ij}\rangle)\langle X,z^{(i)}-z^{(j)}\rangle\gamma_{\sdimn}(x)dx\\
&\qquad+\epsilon\sqrt{2\pi}\pender\\
&\qquad+\sum_{1\leq i<j<k\leq m}\int_{\redS_{ij}\cap\redS_{jk}\cap \redS_{ki}}
\Big(\alpha_{ij}^{2}q_{ij}+\alpha_{jk}^{2}q_{jk}+\alpha_{ki}^{2}q_{ki}\Big)\gamma_{\sdimn}(x)dx.
 \end{flalign*}

 From Lemma \ref{nolowdimp}, $\langle v-\langle v,N_{ij}\rangle, z^{(i)}-z^{(j)}\rangle=0$ for all $1\leq i<j\leq m$.  Since $X=\alpha N_{ij}+v$, we then rewrite the middle term and get
 \begin{flalign*}
 &\frac{d^{2}}{ds^{2}}|_{s=0}\Big(\sum_{1\leq i<j\leq m}\int_{\Sigma_{ij}^{(s)}}\gamma_{\sdimn}(x)dx+\epsilon\pens\Big)\\
&=\sum_{1\leq i<j\leq m}-\int_{\Sigma_{ij}}\Big((\alpha_{ij}+\langle v,N_{ij}\rangle)^{2}(1-\epsilon\vnormt{z^{(i)}-z^{(j)}})+\vnormt{A}^{2}\alpha_{ij}^{2}\Big)\gamma_{\sdimn}(x)dx\\
&\qquad+\epsilon\sqrt{2\pi}\pender\\
&\qquad+\sum_{1\leq i<j<k\leq m}\int_{\redS_{ij}\cap\redS_{jk}\cap \redS_{ki}}
\Big(\alpha_{ij}^{2}q_{ij}+\alpha_{jk}^{2}q_{jk}+\alpha_{ki}^{2}q_{ki}\Big)\gamma_{\sdimn}(x)dx.
 \end{flalign*}

 For each $1\leq i\leq m$, by duality of the $\ell_{2}$ norm, $\exists$ $\omega^{(i)}\in\R^{\adimn}$ with $\vnormtf{\omega^{(i)}}=1$ such that
\begin{flalign*}
&\vnormtf{\sum_{j\in\{1,\ldots,m\}\colon j\neq i}\int_{\Sigma_{ij}}(x-\overline{w}^{(i)})f_{ij}\gamma_{\adimn}(x)dx}
=\Big\langle\sum_{j\in\{1,\ldots,m\}\colon j\neq i}\int_{\Sigma_{ij}}(x-\overline{w}^{(i)})f_{ij}\gamma_{\adimn}(x)dx,\omega^{(i)}\Big\rangle\\
&\qquad\qquad\qquad\qquad\qquad\qquad=\sum_{j\in\{1,\ldots,m\}\colon j\neq i}\int_{\Sigma_{ij}}\langle(x-\overline{w}^{(i)}),\omega^{(i)}\rangle f_{ij}\gamma_{\adimn}(x)dx.
\end{flalign*}
So, we apply the Cauchy-Schwarz inequality to the middle term to get
\begin{flalign*}
 &\frac{d^{2}}{ds^{2}}|_{s=0}\Big(\sum_{1\leq i<j\leq m}\int_{\Sigma_{ij}^{(s)}}\gamma_{\sdimn}(x)dx+\epsilon\pens\Big)\\
&=\sum_{1\leq i<j\leq m}-\int_{\Sigma_{ij}}\Big((\alpha_{ij}+\langle v,N_{ij}\rangle)^{2}\Big[1-\epsilon\vnormt{z^{(i)}-z^{(j)}}-\epsilon\sqrt{2\pi}\int_{\Sigma_{ij}}\langle x-\overline{w}^{(i)},\omega^{(i)}\rangle\gamma_{\sdimn}(x)dx\Big]\\
&\qquad\qquad\qquad\qquad\qquad\qquad\qquad\qquad\qquad+\vnormt{A}^{2}\alpha_{ij}^{2}\Big)\gamma_{\sdimn}(x)dx\\
&\qquad+\sum_{1\leq i<j<k\leq m}\int_{\redS_{ij}\cap\redS_{jk}\cap \redS_{ki}}
\Big(\alpha_{ij}^{2}q_{ij}+\alpha_{jk}^{2}q_{jk}+\alpha_{ki}^{2}q_{ki}\Big)\gamma_{\sdimn}(x)dx.
 \end{flalign*}
By Lemma \ref{zlem} and the second part of Lemma \ref{lemma28.8}, if $\epsilon<\epschoice$, then the first term is positive.  So, it remains to find $\alpha_{ij}$ values such that the last term is nonpositive.

In the case $m=3$, there is nothing to prove.  If (without loss of generality) $z^{(1)}-z^{(2)}=0$, then Theorem \ref{dimredthmp} and Lemma \ref{lemma80p} imply that $\Omega_{1},\ldots,\Omega_{m}\subset\R$.  On the other hand, $z^{(i)}-z^{(j)}\neq0$ for all $1\leq i<j\leq m$, then Lemma \ref{nolowdimp} implies that each $\Sigma_{ij}$ is flat.  So, in either case there is nothing to prove.

We therefore consider $m=4$.  If (without loss of generality) $z^{(1)}-z^{(2)}=0$ and $z^{(3)}-z^{(4)}=0$, then Theorem \ref{dimredthmp} and Lemma \ref{lemma80p} imply that $\Omega_{1},\ldots,\Omega_{m}\subset\R$.  If, $z^{(i)}-z^{(j)}\neq0$ for all $1\leq i<j\leq m$, then Lemma \ref{nolowdimp} implies that each $\Sigma_{ij}$ is flat.  So, the only remaining case to consider is that (without loss of generality) $z^{(1)}-z^{(2)}=0$ and $z^{(i)}-z^{(j)}\neq0$ for every other $1\leq i<j\leq m$.  Lemma \ref{nolowdimp} implies that each $\Sigma_{ij}$ is flat for every $1\leq i<j\leq m$ except $i=1,j=2$.  Looking back at the definition of $q_{ij}$ before \eqref{zero1.5}, we see that most terms of the following sum are zero except for four of them:
\begin{flalign*}
&\sum_{1\leq i<j<k\leq m}\int_{\redS_{ij}\cap\redS_{jk}\cap \redS_{ki}}
\Big(\alpha_{ij}^{2}q_{ij}+\alpha_{jk}^{2}q_{jk}+\alpha_{ki}^{2}q_{ki}\Big)\gamma_{\sdimn}(x)dx\\
&\qquad=\int_{\redS_{12}\cap\redS_{23}\cap \redS_{31}}\Big(\alpha_{23}^{2}q_{23}+\alpha_{31}^{2}q_{31}\Big)\gamma_{\sdimn}(x)dx\\
&\qquad\qquad\qquad+\int_{\redS_{12}\cap\redS_{24}\cap \redS_{41}}\Big(\alpha_{24}^{2}q_{24}+\alpha_{41}^{2}q_{41}\Big)\gamma_{\sdimn}(x)dx.
\end{flalign*}
Moreover, since $N_{ij}=-N_{ji}$, we have $q_{23}=-q_{31}$ and $q_{24}=-q_{41}$ so that
\begin{flalign*}
&\sum_{1\leq i<j<k\leq m}\int_{\redS_{ij}\cap\redS_{jk}\cap \redS_{ki}}
\Big(\alpha_{ij}^{2}q_{ij}+\alpha_{jk}^{2}q_{jk}+\alpha_{ki}^{2}q_{ki}\Big)\gamma_{\sdimn}(x)dx\\
&\quad=\int_{\redS_{12}\cap\redS_{23}\cap \redS_{31}}\Big(\alpha_{23}^{2}-\alpha_{31}^{2}\Big)q_{23}\gamma_{\sdimn}(x)dx
+\int_{\redS_{12}\cap\redS_{24}\cap \redS_{41}}\Big(\alpha_{24}^{2}-\alpha_{41}^{2}\Big)q_{24}\gamma_{\sdimn}(x)dx.
\end{flalign*}

So, choosing $\alpha_{ij}$ such that $\alpha_{ij}=1_{\{i=3\}}-1_{\{j=3\}}$ or $\alpha_{ij}=1_{\{i=4\}}-1_{\{j=4\}}$, or any linear combination of these two, the above quantity is zero.

In summary, there exists a $4$-dimensional space of vector fields of the form $X=v+\alpha_{ij}N_{ij}$ (two dimensions of vectors $v$ since $\Omega_{1},\ldots,\Omega_{m}\subset\R^{2}$, and two dimensions of the $\alpha_{ij}$ terms just described) such that
$$\frac{d^{2}}{ds^{2}}|_{s=0}\Big(\sum_{1\leq i<j\leq m}\int_{\Sigma_{ij}^{(s)}}\gamma_{\sdimn}(x)dx+\epsilon\pens\Big)<0,$$
if the Gaussian volumes of the sets are preserved by the vector field.  Since this space is $4$-dimensional, there must exist a nonzero vector field $X$ such that
 $$\frac{d^{2}}{ds^{2}}|_{s=0}\Big(\sum_{1\leq i<j\leq m}\int_{\Sigma_{ij}^{(s)}}\gamma_{\sdimn}(x)dx+\epsilon\pens\Big)<0,$$
and $\frac{d}{ds}|_{s=0}\gamma_{\adimn}(\Omega_{i}^{(s)})=0$ for all $1\leq i\leq 3$ (and also for $i=4$ since $\cup_{i=1}^{4}\Omega_{i}=\R^{\adimn}$.)  This violates the minimality of the sets, achieving a contradiction, and concluding the proof when $m=4$.

The case $m=5$ seems more difficult since there are more possibilities for the sign of the quadratic form in $\alpha_{ij}^{2}$ terms, so we do not attempt it here.
 \end{proof}

\section{A General Second Variation Formula}\label{secvar}

The main technical ingredient to investigate noise stability is a general second variation formula for quadratic integral functionals, Theorem \ref{thm4} below.  We investigate this formula in the next few sections.

Let $\Omega\subset\R^{\adimn}$ be a $C^{\infty}$ manifold with boundary.  Suppose $\Omega$ has reduced boundary $\redA$, and let $N\colon\redA\to S^{n-1}$ denote the unit exterior normal to $\redA$.  Let $X\colon\R^{\adimn}\to\R^{\adimn}$ be a vector field.
%Unless otherwise stated, we always assume that $X(x)$ is parallel to $N(x)$ for all $x\in\redA$.
Let $\mathrm{div}$ denote the divergence of a vector field.  We write $X$ in its components as $X=(X_{1},\ldots,X_{\adimn})$, so that $\mathrm{div}X=\sum_{i=1}^{\adimn}\frac{\partial}{\partial x_{i}}X_{i}$.  Let $\Psi\colon\R^{\adimn}\times(-1,1)\to\R^{\adimn}$ such that $\Psi(x,0)=x$ and such that $\frac{d}{ds}|_{s=0}\Psi(x,s)=X(\Psi(x,s))$ for all $x\in\R^{\adimn},s\in(-1,1)$.  For any $s\in(-1,1)$, let $\Omega^{(s)}=\Psi(\Omega,s)$.    Note that $\Omega^{(0)}=\Omega$.  Let $G\colon\R^{\adimn}\times\R^{\adimn}\to\R$ be a Schwartz function.
%, e.g. we let $G(x,y)=e^{\frac{-\vnorm{x}_{2}^{2}-\vnorm{y}_{2}^{2}+2\rho\langle x,y\rangle}{2(1-\rho^{2})}}(1-\rho^{2})^{-n/2}(2\pi)^{-n}$, for all $x,y\in\R^{\adimn}$, as in Definition \ref{noisedef}.
For any $x\in\R^{\adimn}$ and any $s\in(-1,1)$, define
\begin{equation}\label{two9c}
V(x,s)\colonequals\int_{\Omega^{(s)}}G(x,y)dy.
\end{equation}

Below, when appropriate, we let $dx$ denote Lebesgue measure, restricted to a surface $\redA\subset\R^{\adimn}$.

\begin{theorem}[\embolden{General Second Variation Formula}, {\cite[Theorem 2.6]{chokski07}}; also {\cite[Theorem 1.10]{heilman15}}]\label{thm4}
Let $F(\Omega)\colonequals \int_{\R^{\adimn}}\int_{\R^{\adimn}} 1_{\Omega}(x)G(x,y)1_{\Omega}(y)dxdy$.  Then

\begin{flalign*}
\frac{1}{2}\frac{d^{2}}{ds^{2}}F(\Omega^{(s)})|_{s=0}
&=\int_{\redA}\int_{\redA}G(x,y)\langle X(x),N(x)\rangle\langle X(y),N(y)\rangle dxdy\\
&\qquad+\int_{\redA}\mathrm{div}(V(x,0)X(x))\langle X(x),N(x)\rangle dx.
\end{flalign*}

\end{theorem}
\begin{remark}
The functional $F(\Omega)$ is sometimes called an interaction energy in the optimal transport literature.
\end{remark}

\section{Applications of the Second Variation Formula}

\subsection{Second Variation of Gaussian Measure}\label{secmeas}

Our first use of Theorem \ref{thm4} will be the computation of the second variation of the Gaussian measure of a set.

\begin{lemma}\label{lemma5p}
Let $G(x,y)=\gamma_{\adimn}(x)\gamma_{\adimn}(y)$, $\forall$ $x,y\in\R^{\adimn}$.  Let $\Omega\subset\R^{\adimn}$ be a $C^{\infty}$ manifold with boundary.  Let
$$F(\Omega)\colonequals \int_{\R^{\adimn}}\int_{\R^{\adimn}} 1_{\Omega}(x)G(x,y)1_{\Omega}(y)dxdy=(\gamma_{\adimn}(\Omega))^{2}.$$
Let $z=\int_{\Omega}yd\gamma_{\adimn}(y)\in\R^{\adimn}$.  Then
\begin{flalign*}
\frac{1}{2}\frac{d^{2}}{ds^{2}}F(\Omega^{(s)})|_{s=0}
&=(\int_{\redA}\langle X(x),N(x)\rangle \gamma_{\adimn}(x)dx)^{2}\\
&\quad+\gamma_{\adimn}(\Omega)\int_{\redA}(\mathrm{div}(X(x))-\langle X(x),x\rangle)\langle X(x),N(x)\rangle \gamma_{\adimn}(x).
\end{flalign*}
\end{lemma}
\begin{proof}
We compute $\frac{d^{2}}{ds^{2}}F(\Omega^{(s)})=\sum_{i=1}^{\adimn}\frac{d^{2}}{ds^{2}}F_{i}(\Omega^{(s)})$.  Using \eqref{two9c} define

\begin{equation}\label{two20p}
V(x,s)=\int_{\Omega^{(s)}}G(x,y)dy=(\int_{\Omega^{(s)}}d\gamma_{\adimn}(y))\gamma_{\adimn}(x),\qquad\forall\,x\in\R^{\adimn},\,\forall\,s\in(-1,1).
\end{equation}
Applying Theorem \ref{thm4},
\begin{flalign*}
&\frac{1}{2}\frac{d^{2}}{ds^{2}}F(s)|_{s=0}
=\int_{\redA}\int_{\redA}\langle X(x),N(x)\rangle\langle X(y),N(y)\rangle \gamma_{\adimn}(x)\gamma_{\adimn}(y)dxdy\\
&\qquad\qquad+\int_{\redA}\mathrm{div}(V(x,0)X(x))\langle X(x),N(x)\rangle dx\\
&\stackrel{\eqref{two20p}}{=}(\int_{\redA}\langle X(x),N(x)\rangle \gamma_{\adimn}(x)dx)^{2}
+\int_{\redA}\mathrm{div}(\gamma_{\adimn}(x)X(x))\langle X(x),N(x)\rangle dx(\int_{\Omega}\gamma_{\adimn}(y)dy)\\
&=(\int_{\redA}\langle X(x),N(x)\rangle \gamma_{\adimn}(x)dx)^{2}
+\gamma_{\adimn}(\Omega)\int_{\redA}(\mathrm{div}(X(x))-\langle X(x),x\rangle)\langle X(x),N(x)\rangle \gamma_{\adimn}(x).
\end{flalign*}
\end{proof}

\subsection{Second Variation of Gaussian Moments}\label{secmoments}

Our next use of Theorem \ref{thm4} will be the computation of the second variation of the squared Gaussian moment of a set.

\begin{lemma}\label{lemma5}
Let $G(x,y)=\sum_{i=1}^{\adimn}(x_{i}-\overline{w}_{i})(y_{i}-\overline{w}_{i})\gamma_{\adimn}(x)\gamma_{\adimn}(y)$, $\forall$ $x,y\in\R^{\adimn}$.  Let $\Omega\subset\R^{\adimn}$ be a $C^{\infty}$ manifold with boundary $\Sigma$.  Let
$$F(\Omega)\colonequals \int_{\R^{\adimn}}\int_{\R^{\adimn}} 1_{\Omega}(x)G(x,y)1_{\Omega}(y)dxdy
=\vnormtf{\int_{\Omega}(x-\overline{w})\gamma_{\adimn}(x)dx}^{2}.$$
Let $z\colonequals\int_{\Omega}yd\gamma_{\adimn}(y)\in\R^{\adimn}$ and let $f(x)\colonequals\langle X(x),N(x)\rangle$ for all $x\in\Sigma$.  Then
\begin{flalign*}
\frac{1}{2}\frac{d^{2}}{ds^{2}}F(\Omega^{(s)})|_{s=0}
&=\vnormtf{\int_{\Sigma}(x-\overline{w}) f(x)\gamma_{\adimn}(x)dx}^{2}+\int_{\redA}f(x)\langle X(x),z-w\rangle \gamma_{\adimn}(x)dx\\
&\quad+\int_{\redA}\langle x-\overline{w},z-w\rangle\Big(\mathrm{div}(X(x))-\langle X(x),x\rangle\Big)\langle X(x),N(x)\rangle\gamma_{\adimn}(x).
\end{flalign*}
\end{lemma}
\begin{proof}
For any $i\in\{1,\ldots,n\}$, let $G_{i}(x,y)= (x_{i}-\overline{w}_{i})(y_{i}-\overline{w}_{i})\gamma_{\adimn}(x)\gamma_{\adimn}(y)$, for all $x,y\in\R^{\adimn}$, and let
$$F_{i}(\Omega)\colonequals\int_{\R^{\adimn}} \int_{\R^{\adimn}} 1_{\Omega}(x) G_{i}(x,y)1_{\Omega}(y)dxdy
=(\int_{\Omega} (x_{i}-\overline{w}_{i})d\gamma_{\adimn}(x))^{2}.$$
Then $F(\Omega)=\sum_{i=1}^{\adimn}F_{i}(\Omega)$.%, and
%$$F(\Omega)=\vnorm{\int_{\Omega} x d\gamma_{\adimn}(x)}_{\ell_{2}^{\adimn}}^{2}.$$
So, let us compute $\frac{d^{2}}{ds^{2}}F(\Omega^{(s)})=\sum_{i=1}^{\adimn}\frac{d^{2}}{ds^{2}}F_{i}(\Omega^{(s)})$.  For any $i\in\{1,\ldots,n\}$, using \eqref{two9c} define

\begin{equation}\label{two20}
\begin{aligned}
V_{i}(x,s)&=\int_{\Omega^{(s)}}G_{i}(x,y)dy=(\int_{\Omega^{(s)}}(y_{i}-\overline{w}_{i})d\gamma_{\adimn}(y))(x_{i}-\overline{w}_{i})\gamma_{\adimn}(x),\\
&\qquad\qquad\qquad\qquad\forall\,x\in\R^{\adimn},\,\forall\,s\in(-1,1).
\end{aligned}
\end{equation}
Applying Theorem \ref{thm4}, for any $i\in\{1,\ldots,n\}$,
\begin{equation}\label{two21}
\begin{aligned}
\frac{1}{2}\frac{d^{2}}{ds^{2}}F_{i}(s)|_{s=0}
&=\int_{\redA}\int_{\redA}(x_{i}-\overline{w}_{i})(y_{i}-\overline{w}_{i})\langle X(x),N(x)\rangle\langle X(y),N(y)\rangle \gamma_{\adimn}(x)\gamma_{\adimn}(y)dxdy\\
&\qquad+\int_{\redA}\mathrm{div}(V_{i}(x,0)X(x))\langle X(x),N(x)\rangle dx\\
&\stackrel{\eqref{two20}}{=}(\int_{\redA}(x_{i}-\overline{w}_{i})\langle X(x),N(x)\rangle \gamma_{\adimn}(x)dx)^{2}\\
&\qquad+\int_{\redA}\mathrm{div}((x_{i}-\overline{w}_{i})\gamma_{\adimn}(x)X(x))\langle X(x),N(x)\rangle dx(\int_{\Omega}(y_{i}-\overline{w}_{i})\gamma_{\adimn}(y)dy).
\end{aligned}
%sum of partial_j derivatives
\end{equation}

Note that, for any $j\in\{1,\ldots,n\}$,
\begin{equation}\label{two22}
\frac{\partial}{\partial x_{j}}((x_{i}-\overline{w}_{i})\gamma_{\adimn}(x)X^{(j)}(x))
=(x_{i}-\overline{w}_{i})(-x_{j}X^{(j)}+\frac{\partial}{\partial x_{j}}X^{(j)})\gamma_{\adimn}(x)+1_{\{i=j\}}X^{(j)}\gamma_{\adimn}(x).
\end{equation}

%Assume that $\partial_{1}X^{(1)}=0$, on $\redA$.  %Then, $X^{(1)}=c$ is constant in the $x_{1}$ direction
Substituting \eqref{two22} into \eqref{two21},
\begin{equation}\label{two23}
\begin{aligned}
&\frac{1}{2}\frac{d^{2}}{ds^{2}}F_{i}(s)|_{s=0}
=(\int_{\redA}(x_{i}-\overline{w}_{i})\langle X(x),N(x)\rangle \gamma_{\adimn}(x)dx)^{2}\\
&\qquad+\int_{\redA}\Big(-(x_{i}-\overline{w}_{i})\langle X(x),x\rangle+(x_{i}-\overline{w}_{i})\mathrm{div}(X)+X^{(i)}\Big)\\
&\qquad\qquad\qquad\qquad\qquad\qquad\cdot\langle X(x),N(x)\rangle \gamma_{\adimn}(x)dx(\int_{\Omega}(y_{i}-\overline{w}_{i})d\gamma_{\adimn}(y)).
\end{aligned}
\end{equation}

The definition of $z$ then concludes the proof.

\end{proof}

\section{Proof of Main Theorem}\label{secmain}

Given the results of Sections \ref{secred} and \ref{secflat}, the proof of Theorem \ref{mainthm} is essentially identical to the argument from the $\epsilon=0$ case given in Section \ref{secsim}.

\begin{proof}[Proof of Theorem \ref{mainthm}]

Let $\Omega_{1},\ldots,\Omega_{m}\subset\R^{\adimn}$ minimize Problem \ref{prob1p}.  By Lemma \ref{nolowdimp}, $L_{ij}$ simplifies to just $\Delta-\langle x,\nabla\rangle+1$.  For every connected component $C$ of $\Omega_{1},\ldots,\Omega_{m}$, let $X_{C}$ be the vector field that is equal to the exterior unit normal vector field of $C$ on $\redb C$ and $X_{C}=0$ on every other connected component of $\Omega_{1},\ldots,\Omega_{m}$ that does not intersect $C$.  (Such a vector field exists by Lemma \ref{lemma52.6}.)  Let $U$ be the linear span of all such vector fields $X_{C}$, as $C$ ranges over the connected components of $\Omega_{1},\ldots,\Omega_{m}$.  If there are more than $m$ connected components of $\Omega_{1},\ldots,\Omega_{m}$, then $\frac{d^{2}}{ds^{2}}|_{s=0}\sum_{1\leq i<j\leq m}\int_{\Sigma_{ij}^{(s)}}\gamma_{\sdimn}(x)dx<0$ for all nonzero $X\in U$.  So, there exist nonzero $(f_{ij})_{1\leq i<j\leq m}$ such that $\frac{d}{ds}|_{s=0}\gamma_{\adimn}(\Omega_{i}^{(s)})=0$ for all $1\leq i\leq m-1$ (and also for $i=m$ since $\cup_{i=1}^{m}\Omega_{i}=\R^{\adimn}$).  we form a nontrivial linear combination $X$ of each of these vector fields to once again get $\frac{d}{ds}|_{s=0}\gamma_{\adimn}(\Omega_{i}^{(s)})=0$ for all $1\leq i\leq m$ and $\frac{d^{2}}{ds^{2}}|_{s=0}\sum_{1\leq i<j\leq m}\int_{\Sigma_{ij}^{(s)}}\gamma_{\sdimn}(x)dx<0$.  So, there must be exactly $m$ connected components of $\Omega_{1},\ldots,\Omega_{m}$.  The regularity condition, Lemma \ref{lemma52.6} then concludes the proof.  We know that each of $\Omega_{1},\ldots,\Omega_{m}$ is connected with flat boundary pieces, the sets $\Omega_{1},\ldots,\Omega_{m}$ meet in threes at $120$ degree angles, and they meet in fours like the cone over the three-dimensional regular simplex.  In the case $m=3$, there are only three possible configurations for the sets (up to rotation).  And for general $m$, there are only finitely many possible configurations of the sets.  So, we can conclude the proof by either (i) checking the surface area of each such case directly, (ii) using the matrix-valued partial differential inequality and maximal principle argument of \cite{milman18b}, or (iii) appealing directly to the result of \cite{milman18b}.
\end{proof}

\section{Noise Stability}\label{secnoise}

\begin{proof}[Proof of Corollary \ref{mainthm3}]

For any $s>0,x\in\R^{\adimn}$, and for any $f\colon\R^{\adimn}\to[0,1]$, let $T_{\rho}f(x)$ as in \eqref{oudef}.  It is well known \cite[Section 9]{heilman12} that, for all $0<\rho<1,x\in\R^{\adimn}$.
$$\frac{d}{d\rho}T_{\rho}f=\frac{1}{\rho}(-\overline{\Delta}T_{\rho}f+\langle x,\overline{\nabla}T_{\rho}f\rangle).$$
Using the divergence theorem,
\begin{equation}\label{ten1}
\begin{aligned}
&\frac{d}{d\rho}\int_{\Omega}T_{\rho}1_{\Omega}(x)\gamma_{\adimn}(x)dx
=\frac{1}{\rho}\int_{\Omega}-\overline{\mathcal{L}}T_{\rho}1_{\Omega}(x)\gamma_{\adimn}(x)dx
=-\frac{1}{\rho}\int_{\Omega}\mathrm{div}(\gamma_{\adimn}(x)\nabla T_{\rho}1_{\Omega}(x))dx\\
&\qquad\qquad=-\frac{1}{\rho}\int_{\partial\Omega}\langle\nabla T_{\rho}1_{\Omega}(x), N(x)\rangle dx.
\end{aligned}
\end{equation}
Changing variables and differentiating,
$$\nabla T_{\rho}1_{\Omega^{c}}(x)=\frac{\rho}{\sqrt{1-\rho^{2}}}\int_{\R^{\adimn}}y 1_{\Omega}(\rho x+y\sqrt{1-\rho^{2}})\gamma_{\adimn}(y)dy,\qquad\forall\,x\in\R^{\adimn}.$$
% if it is a half space and x=0, get
%\frac{\rho}{\sqrt{1-\rho^{2}}}\int_{0}^{\infty}y e^{-y^{2}/2}dy/\sqrt{2\pi}=
Therefore, $\lim_{\rho\to1^{-}}\rho^{-1}\sqrt{2\pi(1-\rho^{2})}\,\nabla T_{\rho}1_{\Omega^{c}}(x)=-N(x)$ for all $x\in\partial\Omega$.  That is,
\begin{equation}\label{ten2}
\nabla T_{\rho}1_{\Omega}(x)=-\frac{\sqrt{2\pi(1-\rho^{2})}}{\rho}N(x)+o(\sqrt{1-\rho^{2}}),\qquad\forall\,x\in\partial\Omega.
\end{equation}
Also, $\lim_{\rho\to1^{-}}\rho^{-1}\sqrt{2\pi(1-\rho^{2})}\,\nabla T_{\rho}1_{\Omega^{c}}(x)=0$ for all $x\notin\partial\Omega$.  So, using $f=1_{\Omega}$,
\begin{flalign*}
&\gamma_{\adimn}(\Omega)-\int_{\R^{\adimn}}1_{\Omega}(x)T_{\rho}1_{\Omega}(x))\gamma_{\adimn}(x)dx\\
&\quad=\int_{\eta=\rho}^{\eta=1}\frac{d}{d\eta}\int_{\R^{\adimn}}1_{\Omega}(x)T_{\eta}1_{\Omega}(x))\gamma_{\adimn}(x)dxd\eta\\
&\quad\stackrel{\eqref{ten1}}{=}\int_{\eta=\rho}^{\eta=1}-\frac{1}{\eta}\int_{\partial\Omega}\langle\nabla T_{\eta}1_{\Omega}(x), N(x)\rangle\gamma_{\adimn}(x)dxd\eta\\
&\quad\stackrel{\eqref{ten2}}{=}\int_{\eta=\rho}^{\eta=1}\Big(o((1-\eta^{2})^{-1/2})+\frac{1}{\sqrt{2\pi(1-\eta^{2})}}\int_{\partial\Omega}\langle N(x), N(x)\rangle \gamma_{\adimn}(x)dx\Big)d\eta\\
&\quad=o(\sqrt{1-\rho^{2}})+\frac{\sqrt{1-\rho^{2}}}{\sqrt{\pi}}\int_{\partial\Omega}\gamma_{\adimn}(x)dx
=o(\sqrt{1-\rho^{2}})+\frac{\sqrt{1-\rho^{2}}}{2\pi}\int_{\partial\Omega}\gamma_{\sdimn}(x)dx.
\end{flalign*}
\end{proof}

\section{A Closing Remark}

When $m>3$, Theorem \ref{mainthm} relies on a solution of Problem \ref{prob3}.  We include here an intriguing observation concerning Problem \ref{prob3}.

\begin{lemma}\label{endlemma}
Let $\Omega_{1},\ldots,\Omega_{m}$ maximize Problem \ref{prob3}.  Suppose $\Omega_{i}=\Omega_{i}'\times\R$ for all $1\leq i\leq m$.  Let $v\in\R^{\adimn}$ be a fixed vector.  Let $X$ be the vector field
$$X\colonequals x_{\adimn}v.$$
Then
$$\frac{d^{2}}{ds^{2}}|_{s=0}\pens=0.$$
\end{lemma}
\begin{proof}
As above, we denote $f_{ij}\colonequals\langle X,N_{ij}\rangle$ and $z^{(i)}\colonequals \int_{\Omega_{i}}x\gamma_{\adimn}(x)dx$.   By \eqref{zero1} and Fubini's Theorem, if we first integrate with respect to $x_{\adimn}$, we see that $(f_{ij})_{1\leq i<j\leq m}$ is automatically Gaussian volume-preserving, so that \eqref{zero1} is zero for all $1\leq i<j\leq m$.  Consequently, the last term of Lemma \ref{lemma5} is automatically zero, and we get
\begin{flalign*}
&\frac{d^{2}}{ds^{2}}|_{s=0}\pens\\
&\qquad=\sum_{i=1}^{m}\sqrt{2\pi}\vnormtf{\sum_{j\in\{1,\ldots,m\}\colon j\neq i}\int_{\Sigma_{ij}}(x-\overline{w}^{(i)})f_{ij}\gamma_{\adimn}(x)dx}^{2}\\
&\qquad\qquad+\sum_{1\leq i<j\leq m}\int_{\Sigma_{ij}} f_{ij}\langle v,z^{(i)}-w\rangle\gamma_{\sdimn}(x)dx
\end{flalign*}
We now show that each of these terms are negatives of each other.
\begin{flalign*}
&\sum_{1\leq i<j\leq m}\int_{\Sigma_{ij}} f_{ij}\langle v,z^{(i)}-w\rangle\gamma_{\sdimn}(x)dx
=\sum_{1\leq i<j\leq m}\int_{\Sigma_{ij}} x_{\adimn}^{2}\langle v,N_{ij}\rangle\langle v,z^{(i)}-w\rangle\gamma_{\sdimn}(x)dx\\
&\,=\sum_{1\leq i<j\leq m}\int_{\Sigma_{ij}'} \langle v,N_{ij}\rangle\langle v,z^{(i)}-w\rangle\gamma_{\sdimn-1}(x)dx
=\sum_{1\leq i<j\leq m}\langle v,z^{(i)}-w\rangle\int_{\Sigma_{ij}'} \langle v,N_{ij}\rangle\gamma_{\sdimn-1}(x)dx\\
&\,=\sum_{1\leq i<j\leq m}\langle v,z^{(i)}-w\rangle\int_{\Sigma_{ij}} \langle v,N_{ij}\rangle\gamma_{\sdimn}(x)dx
=\sum_{1\leq i<j\leq m}\langle v,z^{(i)}-w\rangle \left\langle v,\int_{\Sigma_{ij}}N_{ij}\gamma_{\sdimn}(x)dx\right\rangle\\
&\,=\sqrt{2\pi}\sum_{1\leq i<j\leq m}\langle v,z^{(i)}-w\rangle \left\langle v,\int_{\Sigma_{ij}}N_{ij}\gamma_{\adimn}(x)dx\right\rangle
\stackrel{\eqref{zeq}}{=}\sqrt{2\pi}\sum_{i=1}^{m}\langle v,z^{(i)}-w\rangle \langle v,z^{(i)}\rangle\\
&=-\sqrt{2\pi}\sum_{i=1}^{m}\langle v,z^{(i)}\rangle^{2}.
\end{flalign*}
The last line used $\sum_{i=1}^{m}z_{i}=0$ which follows by $\cup_{i=1}^{m}\Omega_{i}=\R^{\adimn}$.  For the other term we first integrate with respect to $x_{\adimn}$ and use Fubini's Theorem to get
\begin{flalign*}
%&\sum_{i=1}^{m}\sqrt{2\pi}\vnormtf{\sum_{j\in\{1,\ldots,m\}\colon j\neq i}\int_{\Sigma_{ij}}(x-\overline{w}^{(i)})f_{ij}\gamma_{\adimn}(x)dx}^{2}\\
&\sum_{i=1}^{m}\sqrt{2\pi}\vnormtf{\sum_{j\in\{1,\ldots,m\}\colon j\neq i}\int_{\Sigma_{ij}}(x-\overline{w}^{(i)})f_{ij}\gamma_{\adimn}(x)dx}^{2}\\
&\qquad=\sum_{i=1}^{m}\sqrt{2\pi}\Big(\sum_{j\in\{1,\ldots,m\}\colon j\neq i}\int_{\Sigma_{ij}}(x_{\adimn}-\overline{w}^{(i)}_{\adimn})x_{\adimn}\langle v,N_{ij}\rangle\gamma_{\adimn}(x)dx\Big)^{2}\\
&\qquad=\sum_{i=1}^{m}\sqrt{2\pi}\Big(\sum_{j\in\{1,\ldots,m\}\colon j\neq i}\int_{\Sigma_{ij}'}\langle v,N_{ij}\rangle\gamma_{\sdimn}(x)dx\Big)^{2}\\
&\qquad=\sum_{i=1}^{m}\sqrt{2\pi}\Big(\sum_{j\in\{1,\ldots,m\}\colon j\neq i}\int_{\Sigma_{ij}}\langle v,N_{ij}\rangle\gamma_{\adimn}(x)dx\Big)^{2}
\stackrel{\eqref{zeq}}{=}\sum_{i=1}^{m}\sqrt{2\pi}\langle v,z^{(i)}\rangle^{2}.
\end{flalign*}
\end{proof}

\medskip
\noindent\textbf{Acknowledgement}.  Thanks to Larry Goldstein, Elchanan Mossel and Joe Neeman for helpful discussions.

\bibliographystyle{amsalpha}
\bibliography{12162011}

\end{document}

%% file: optset.pdf_tex
%% Creator: Inkscape inkscape 0.92.1, www.inkscape.org
%% Accompanies image file 'optset.pdf' (pdf, eps, ps)
\begingroup%
  \makeatletter%
  \providecommand\color[2][]{%
    \errmessage{(Inkscape) Color is used for the text in Inkscape, but the package 'color.sty' is not loaded}%
    \renewcommand\color[2][]{}%
  }%
  \providecommand\transparent[1]{%
    \errmessage{(Inkscape) Transparency is used (non-zero) for the text in Inkscape, but the package 'transparent.sty' is not loaded}%
    \renewcommand\transparent[1]{}%
  }%
  \providecommand\rotatebox[2]{#2}%
  \ifx\svgwidth\undefined%
    \setlength{\unitlength}{442.21899799bp}%
    \ifx\svgscale\undefined%
      \relax%
    \else%
      \setlength{\unitlength}{\unitlength * \real{\svgscale}}%
    \fi%
  \else%
    \setlength{\unitlength}{\svgwidth}%
  \fi%
  \global\let\svgwidth\undefined%
  \global\let\svgscale\undefined%
  \makeatother%
  \begin{picture}(1,0.66217246)%
    \put(0,0){\includegraphics[width=\unitlength,page=1]{optset.pdf}}%
    \put(0.11488632,0.38564913){\color[rgb]{0,0,0}\makebox(0,0)[lb]{\smash{$\Omega_{1}$}}}%
    \put(0.36991422,0.04180723){\color[rgb]{0,0,0}\makebox(0,0)[lb]{\smash{$\Omega_{3}$}}}%
    \put(0.59074735,0.38606375){\color[rgb]{0,0,0}\makebox(0,0)[lb]{\smash{$\Omega_{2}$}}}%
    \put(0,0){\includegraphics[width=\unitlength,page=2]{optset.pdf}}%
    \put(0.25112809,0.22593037){\color[rgb]{0,0,0}\makebox(0,0)[lb]{\smash{\tiny$120^{o}$}}}%
  \end{picture}%
\endgroup%

%% file: tripod.pdf_tex
%% Creator: Inkscape inkscape 0.92.1, www.inkscape.org
%% Accompanies image file 'tripod.pdf' (pdf, eps, ps)
\begingroup%
  \makeatletter%
  \providecommand\color[2][]{%
    \errmessage{(Inkscape) Color is used for the text in Inkscape, but the package 'color.sty' is not loaded}%
    \renewcommand\color[2][]{}%
  }%
  \providecommand\transparent[1]{%
    \errmessage{(Inkscape) Transparency is used (non-zero) for the text in Inkscape, but the package 'transparent.sty' is not loaded}%
    \renewcommand\transparent[1]{}%
  }%
  \providecommand\rotatebox[2]{#2}%
  \ifx\svgwidth\undefined%
    \setlength{\unitlength}{368.50393701bp}%
    \ifx\svgscale\undefined%
      \relax%
    \else%
      \setlength{\unitlength}{\unitlength * \real{\svgscale}}%
    \fi%
  \else%
    \setlength{\unitlength}{\svgwidth}%
  \fi%
  \global\let\svgwidth\undefined%
  \global\let\svgscale\undefined%
  \makeatother%
  \begin{picture}(1,0.85305023)%
    \put(0,0){\includegraphics[width=\unitlength,page=1]{tripod.pdf}}%
    \put(0.00006017,0.70825526){\color[rgb]{0,0,0}\makebox(0,0)[lb]{\smash{$\Sigma_{ki}$}}}%
    \put(0.78785514,0.76549855){\color[rgb]{0,0,0}\makebox(0,0)[lb]{\smash{}}}%
    \put(1.26071668,0.77680613){\color[rgb]{0,0,0}\makebox(0,0)[lb]{\smash{}}}%
    \put(0.78579931,0.79119752){\color[rgb]{0,0,0}\makebox(0,0)[lb]{\smash{$\Sigma_{jk}$}}}%
    \put(0.25531148,0.02356111){\color[rgb]{0,0,0}\makebox(0,0)[lb]{\smash{$\Sigma_{ij}$}}}%
    \put(-0.00675769,0.36048245){\color[rgb]{0,0,0}\makebox(0,0)[lb]{\smash{$\Omega_{i}$}}}%
    \put(0.39202143,0.78897654){\color[rgb]{0,0,0}\makebox(0,0)[lb]{\smash{$\Omega_{k}$}}}%
    \put(0.62180324,0.30009345){\color[rgb]{0,0,0}\makebox(0,0)[lb]{\smash{$\Omega_{j}$}}}%
    \put(0.1674235,0.58748511){\color[rgb]{0,0,0}\makebox(0,0)[lb]{\smash{$N_{ki}$}}}%
    \put(0.5767978,0.68775026){\color[rgb]{0,0,0}\makebox(0,0)[lb]{\smash{$N_{jk}$}}}%
    \put(0.34095156,0.2253748){\color[rgb]{0,0,0}\makebox(0,0)[lb]{\smash{$N_{ij}$}}}%
    \put(0.46713171,0.56401839){\color[rgb]{0,0,0}\makebox(0,0)[lb]{\smash{$\nu_{ij}$}}}%
    \put(0.22736437,0.49848975){\color[rgb]{0,0,0}\makebox(0,0)[lb]{\smash{$\nu_{jk}$}}}%
    \put(0.43217775,0.38030456){\color[rgb]{0,0,0}\makebox(0,0)[lb]{\smash{$\nu_{ki}$}}}%
  \end{picture}%
\endgroup%